\documentclass{amsart}
\usepackage{a4wide}

\usepackage{pgf,tikz,pgfplots}
\pgfplotsset{compat=1.15}
\usepackage{mathrsfs}
\usetikzlibrary{arrows}

\usepackage[T1]{fontenc}
\usepackage{microtype}
\usepackage{lmodern}
\usepackage[colorlinks=true,urlcolor=blue, citecolor=red,linkcolor=blue,linktocpage,pdfpagelabels, bookmarksnumbered,bookmarksopen]{hyperref}
\usepackage[hyperpageref]{backref}
\usepackage{amsthm} 
\usepackage{latexsym,amsmath,amssymb}

\usepackage{bbm}

\usepackage{accents}
\usepackage{esint}

\usepackage{soul}
\usepackage{mathtools} 
\usepackage{xparse} 
\usepackage[capitalize]{cleveref}

\usepackage[shortlabels]{enumitem}

\newcommand{\col}[1]{{#1}}

\title[Cubic fractional Schr\"odinger equation]{Local well-posedness for  cubic fractional Schr\"odinger equations with derivatives on the right-hand side}

\newcommand{\tell}{\mathfrak{l}}
\newcommand{\tn}{\mathfrak{n}}
\newcommand{\tm}{\mathfrak{m}}

\newcommand{\N}{{\mathbb N}}

\renewcommand{\S}{{\mathbb S}}

\newtheorem{theorem}{Theorem}
\newtheorem{lemma}[theorem]{Lemma}

\newtheorem{proposition}[theorem]{Proposition}

\theoremstyle{definition}

\theoremstyle{remark}

\newcommand\supp{{\rm supp\,}}

\newcommand{\bc}{\mathfrak{c}}
\newcommand{\ind}{{\mathbbm{1}} }

\newcommand{\R}{\mathbb{R}}
\newcommand{\C}{\mathbb{C}}

\newcommand{\Z}{\mathbb{Z}}

\newcommand{\brac}[1]{\left (#1 \right )}
\newcommand{\norm}[1]{\left \|#1 \right \|}

\newcommand{\abs}[1]{\left\lvert #1 \right \rvert}

\renewcommand{\vec}[1]{{\bf #1}}

\renewcommand{\i}{{\rm \bf i}}

\newcommand{\barint}{
\rule[.036in]{.12in}{.009in}\kern-.16in \displaystyle\int }

\newcommand{\barcal}{\text{$ \rule[.036in]{.11in}{.007in}\kern-.128in\int $}}

\def\mvint_#1{\mathchoice
          {\mathop{\vrule width 6pt height 3 pt depth -2.5pt
                  \kern -8pt \intop}\nolimits_{\kern -3pt #1}}%
          {\mathop{\vrule width 5pt height 3 pt depth -2.6pt
                  \kern -6pt \intop}\nolimits_{#1}}%
          {\mathop{\vrule width 5pt height 3 pt depth -2.6pt
                  \kern -6pt \intop}\nolimits_{#1}}%
          {\mathop{\vrule width 5pt height 3 pt depth -2.6pt
                  \kern -6pt \intop}\nolimits_{#1}}}

\numberwithin{theorem}{section} \numberwithin{equation}{section}

\newcommand{\aleq}{\lesssim}
\newcommand{\ageq}{\succsim}
\newcommand{\aeq}{\approx}

\newcommand{\Ds}[1]{|\nabla|^{#1}}

\newcommand{\Dels}[1]{(-\Delta)^{#1}}

\usepackage{scalerel}[2014/03/10]
\usepackage[usestackEOL]{stackengine}
\def\avint{\,\ThisStyle{\ensurestackMath{%
			\stackinset{c}{.2\LMpt}{c}{.5\LMpt}{\SavedStyle-}{\SavedStyle\phantom{\int}}}%
		\setbox0=\hbox{$\SavedStyle\int\,$}\kern-\wd0}\int}

\let\latexchi\chi
\makeatletter
\renewcommand\chi{\@ifnextchar_\sub@chi\latexchi}
\newcommand{\sub@chi}[2]{
  \@ifnextchar^{\subsup@chi{#2}}{\latexchi^{}_{#2}}%
}
\newcommand{\subsup@chi}[3]{
  \latexchi_{#1}^{#3}%
}
\makeatother
\newcommand{\eps}{\varepsilon}

\author{Ahmed Dughayshim}
\email[Ahmed Dughayshim]{aha80@pitt.edu}

\author{Silvino Reyes Farina}
\email[Silvino Reyes Farina]{sir25@pitt.edu}

\author{Armin Schikorra}
\email[Armin Schikorra]{armin@pitt.edu}
\address{Department of Mathematics,
University of Pittsburgh,
301 Thackeray Hall,
Pittsburgh, PA 15260, USA}

\begin{document}
\begin{abstract}
For $s \in (\frac{1}{2},1]$ we investigate well-posedness of the equation
\[
\brac{ \i \partial_t + \Dels{s} } u = \brac{\Ds{1-2s} |u|^2}\ \Ds{2s-1} u
\]
under small initial data $\|u(0)\|_{\dot{H}^{\frac{n-2s}{2}}(\R^n)} \ll 1$.

This equation is a model equation for for $s$-Schr\"odinger map equation
\[
\partial_t \vec{\psi} = \vec{\psi} \wedge \Dels{s} \vec{\psi}: \quad \psi: \R^n \times \R \to \S^{2},
\]

\end{abstract}
\maketitle
\tableofcontents

\section{Introduction}
Let $s \in (\frac{1}{2},1]$. In this work we are interested in the short-time well-posedness of the equation
\begin{equation}\label{eq:mainpdeR}
\brac{ \i \partial_t + \Dels{s} } u = \brac{\Ds{-(2s-1)} |u|^2}\ \Ds{2s-1} u : \quad \text{in $\R^n \times \R$}
\end{equation}
for maps $u: \R^n \times \R \to \C$, subject to initial data in $\dot{H}^{\frac{n-2s}{2}}(\R^n)$.

Our main result is the following short-term wellposedness result

\begin{theorem}\label{th:main}
Fix $s \in (\frac{1}{2},1)$ and let $n \geq 4$. There exists small $\eps > 0$ such that the following holds:

Assume $u_0 \in \dot{H}^{\frac{n{-2s}}{2}}(\R^n)$ such that
\[
\|u_0\|_{\dot{H}^{\frac{n{-2s}}{2}}(\R^n)} \leq \eps 
\]
Then there exists a solution $u \in L^\infty_t ((-1,1),\dot{H}^{\frac{n-2s}{2}}(\R^n))$ such that
\begin{equation}\label{eq:mainpde}
\brac{ \i \partial_t + \Dels{s} } u = \brac{\Ds{-(2s-1)} |u|^2}\ \Ds{2s-1} u \quad \text{in }\R^d \times (-1,1)
\end{equation}
with
\[
 \sup_{t} \|u(t)\|_{\dot{H}^{\frac{n-2s}{2}}(\R^n)} \leq C \|u_0\|_{\dot{H}^{\frac{n{-2s}}{2}}(\R^n)}.
\]
where $C$ depends only on the dimension $n$ and order of differentiability $s$.
\end{theorem}

\Cref{eq:mainpdeR} is of interest as the model equation for the $s$-Schr\"odinger maps equation
\begin{equation}\label{eq:sschroedinger}
\partial_t \vec{\psi} = \vec{\psi} \wedge \Dels{s} \vec{\psi}: \quad \psi: \R^n \times \R \to \S^{2}.
\end{equation}
For $s=1$ \eqref{eq:sschroedinger} is the Schr\"odinger maps equation, see e.g. \cite{BIKTAnnals} and references within. Indeed, for the case $s=1$ \Cref{th:main} follows from the arguments in e.g. \cite{IKCMP,BIK07}.

For $s=\frac{1}{2}$ \eqref{eq:sschroedinger}  is the halfwave map equation whose analysis has been initiated in \cite{KS18,LS2018},  with several results since, e.g. \cite{Berntson_2020,SWZ21,GL18,Berntson_2020,LenzmannSok20,KK21,GerardLenzmann,ESRS,Silvino24,SS02}. For results especially for well-posedness we refer to \cite{KS18,KK21,Liu2021,Liu2023,Marsden24,FS24,GerardLenzmann} and references within. In this halfwave map case, $s=\frac{1}{2}$, taking another time derivative, the halfwave map equation can be reduced to an equation reminiscent of the wave map equation, see \cite{T98,Tao01I,Tao01II}, and this is the basis several well-posedness results mentioned above (in particular those in higher dimension). 

For $s \in (1/2,1)$, \eqref{eq:sschroedinger} can instead be written as a Schr\"odinger system, that -- under certain choices of frames -- becomes similar to \eqref{eq:mainpdeR}. We will discuss the details and implications for $s$-Schr\"odinger maps in a future work. Here, we focus on the model equation \eqref{eq:mainpdeR} whose well-posedness we believe to be interesting in its own right.

An analogous result for \Cref{th:main} holds also for the case $s=\frac{1}{2}$ and follows from standard Strichartz estimates -- observe that there are no derivatives on the right-hand side if $s =\frac{1}{2}$ -- this was used for well-posedness for the halfwave maps e.g. \cite{KS18,KK21,FS24}.

Indeed, for $s \in (1/2,1)$, the derivatives of the right-hand side of the equation \eqref{eq:mainpde} mean that ``simple'' Strichartz estimates do not suffice to treat this PDE. Instead we need to consider estimates on frequency cones to obtain local smoothing estimates. These arguments are strongly inspired from the techniques developed for the Schr\"odinger map equation \cite{IKCMP,BIK07,BIKTAnnals}. For $s=1$ this was used in particular in \cite{IKDIE,IKCMP} and our works follows many of their ideas, extending them to the case $\frac{1}{2}<s<1$.

{\bf Outline.} The paper is organized as follows: In \Cref{s:prelims} we introduce our notation and discuss the resolution spaces $Z_k$ used in our argument. In \eqref{eq:weirdNdef} we also introduce the multiplier $N$ that on frequency cones in direction $e \in \S^{n-1}$ is crucial for the local smoothening estimates. In \Cref{s:resspaceest} we discuss the basic resolution space estimates, in particular the aforementioned smoothening and maximal estimates. In \Cref{s:linearest} we discuss the homogeneous and inhomogeneous linear estimates of the operator $\i \partial_t +\Dels{s}$ with respect to our resolution spaces. In \Cref{s:trilinearest} we then estimate the right-hand side of \eqref{eq:mainpdeR}. Lastly, in \Cref{s:existence} we put things together to obtain \Cref{th:main}. There we also discuss the other properties regarding well-posedness that follow from our arguments, see \Cref{th:mainshort}.

\subsection*{Acknowledgement} Discussions with S. Herr and E. Eyeson are gratefully acknowledged. This project was partially funded by NSF Career DMS-2044898. A.S. is an Alexander-von-Humboldt Fellow.

\section{Notation, Conventions, Definition of spaces and Preliminaries}\label{s:prelims}
Generally, we will use Greek letters $\xi,\tau$ for phase space, and Roman letters $x,t$ for physical space.

Throughout this work we assume, unless otherwise stated, that $n \geq 4$ and $s \in (\frac{1}{2},1)$. Constants, like $C$ can change from line to line and generally may depend on dimension and $s$. We use the notation $A \aleq B$ if there is a multiplicative, nonnegative constant $C$ such that $A \leq C B$, and $A \aeq B$ if $A \aleq B$ and $A \ageq B$.

We will write the Fourier transform as 
\[
 \mathcal{F}_{\R^k} f(\zeta) = \int_{\R^k} e^{-\i \zeta \cdot x} f(\zeta) d\zeta
\]
and it's inverse (up to a multiplicative constant which we will ignore for the sake of readability)
\[
 \mathcal{F}_{\R^k}^{-1} f(\zeta) = \int_{\R^k} e^{+\i \zeta \cdot x} f(\zeta) d\zeta
\]
When the space is clear we will drop the subscript. Another abuse of notation is that we will ``define'' the Fourier symbol of the operator $\i \partial_t + \Dels{s}$ as
 \[
   \mathcal{F} \brac{(\i \partial_t + \Dels{s})u}(\xi,\tau) =  -(\tau +  |\xi|^{2s})\mathcal{F} u(\xi,\tau)
\]
the multiplicative constant that should be on the right-hand side plays no role in the analysis.

%
Sometimes, for simplicity, we will denote the Fourier transform as follows
\[
 f(\hat{\xi},t) = \mathcal{F}_{\R^{n}} f(\cdot,t) (\xi)
\]
\[
 f(\hat{\xi},\hat{\tau}) = \mathcal{F}_{\R^n \times \R} f(\cdot,\cdot) (\xi,\tau)
\]
\[
 f(x,\hat{\tau}) = \mathcal{F}_{\R} f(x,\cdot) (\tau).
\]

\subsection*{Frequency projections}
Let $\eta$ be the typical Littlewood-Paley bump function, $\eta \in C_c^\infty((-2,2),[0,1])$, $\eta \equiv 1$ in $[-1,1]$, $\eta(r) = \eta(|r|)$. Also let $\varphi(r) := \eta(r/) - \eta(r/2)$ and observe that we have 
\[
 1 = \eta(r) + \sum_{k=1}^\infty \varphi(r/2^{k}) \quad \forall r \in [0,\infty).
\]
For vectors $\xi \in \R^n$ we will tacitly mean 
\[
\eta(\xi) \equiv \eta(|\xi|), \quad \varphi(\xi) := \varphi(|\xi|).
\]
We use the notation 
$
\eta_{\leq k}(r) = \eta(r/2^{k})
$,
and 
$$
\eta_{[k_{1},k_{2}]}(r) = \sum_{m=k_{1}}^{k_{2}} \varphi(r/2^{m})
$$

We also use $\eta^+$ to restrict to $r > 0$, e.g. for $k_1,k_2 \in \Z$

\[ \eta^{+}_{[k_{1},k_{2}]}(r) = \ind_{[0,\infty)}(r) \eta_{[k_{1},k_{2}]}(r) \]

For $f: \R^{n+1} \to \R$, we denote the frequency projections by
\[
 \Delta_{k} f(x,t) := \mathcal{F}^{-1}_{\R^{n+1}}(\varphi(\xi/2^k) \mathcal{F}_{\R^{n+1}}f(\xi,\tau))  \quad k \in \Z,
\] 
and for $j \in \N \cup \{0\}$ we denote the modulation 
\[
 Q_j f(x,t) :=  \begin{cases}
                   \mathcal{F}_{\R^{n+1}}^{-1} \brac{\eta({\tau {+} |\xi|^{2s}})\,  \mathcal{F}_{\R^{n+1}} f(\xi,\tau)}   \quad &j = 0, \\
                   \mathcal{F}_{\R^{n+1}}^{-1} \brac{\varphi\brac{{\tau {+} |\xi|^{2s}}/2^j}\,  \mathcal{F}_{\R^{n+1}} f(\xi,\tau)} \quad & j \geq 1.
                   \end{cases}
\]
We also set  
\[
 \Delta_{\leq k}  := \sum_{\ell =-\infty}^{k} \Delta_{\ell} 
\]
and similarly for $Q_{\leq 2^j}$.

We will need one more decomposition: fix an even function
\begin{equation}\label{IntegFunc}
     \chi \in C_c^\infty((-2,2),[0,1]), \quad \chi \equiv 1 \text{ in } [-2/3,2/3]
\end{equation}

such that
\[
 \sum_{\tell \in \Z} \chi(x-\tell) = 1 \quad \forall x \in \R.
\]
We set 
\[
 \chi_{k;\tell}(\xi) := \chi( (\xi^1 -\tell)/2^k)\, \chi( (\xi^2 -\tell)/2^k)\, \ldots \chi( (\xi^n -\tell)/2^k)
\]
Then we have 
\[
 \sum_{\tell \in 2^k \Z^n} \chi_{k;\tell}(\xi) = 1 \quad \forall \xi \in \R^n, k \in \Z.
\]
Set 
\begin{equation}\label{eq:Pktell}
 P_{k,\tell} f:= \mathcal{F}^{-1}\brac{\chi_{k,\tell}\mathcal{F} f}
\end{equation}
Thus, for any $k \in \Z$
\begin{equation}\label{eq:weirdfinedecomp}
 f = \sum_{\tell \in 2^k \Z^n} P_{k,\tell} f.
\end{equation}
Observe that $P_{k,\tell}$ is the only projection operator in this work that is \emph{not} rotation invariant.

\subsection{Resolution spaces}
Throughout this paper we fix the constant
\[
 \bc{} \in (0,1)
\]
and 
\[
 \mathscr{E} \subset \S^{n-1},
\]
a finite, but suitably dense set of vectors in $\S^{n-1}$
so that we have the conclusion of the following ``geometric observation'' Lemma.
\begin{lemma}\label{la:geometricobserv}
For any $\bc{} \in (0,1)$ there exists a finite subset $\mathscr{E} \subset \S^{n-1}$ such that the following holds:

 \item For any $\xi \in \R^n$ there exists $e \in \mathscr{E}$ such that 
 \[
  \abs{\xi-\abs{\xi}e} \leq \sqrt{2-2\bc{}} |\xi|.
 \]
and in particular
 \[
  \langle \xi,e\rangle \geq \bc{} |\xi|.
 \]

 Moreover there exists a decomposition of unity on the $\S^{n-1}$-sphere denoted by $(\vartheta_e)_{e \in \mathscr{E}}$ with 
 \[
  \supp \vartheta_e \subset \{\xi \in \S^{n-1}: \quad \langle \xi, e\rangle \geq \bc\}
 \]
and 
 \[
  \sum_{e \in \mathscr{E}} \vartheta_e(\xi/|\xi|) = 1 \quad \forall \xi \in \R^n \setminus \{0\}
 \]

 We will denote 
 \[
  \vartheta_{\xi \cdot e \ageq |\xi|} := \vartheta_{e} (\xi/|\xi|)
 \]
and with some abuse of notation will also use this notation for $e \in \S^{n-1}$ for $e \not \in \mathscr{E}$ as simply a cutoff function on the cone.

Observe that $\vartheta_{\xi \cdot e \ageq |\xi|} \not \in C^\infty$ since it has a singularity at $0$.
\end{lemma}

For $f: \R^{n+1} \to \C$ and $k \in \Z$ we define the $X_k$-spaces
\begin{equation}\label{eq:Xdef}
 \|f\|_{X_k} := \begin{cases}
                                       \sum_{j=0}^\infty 2^{\frac{j}{2}} \|Q_{j} f\|_{L^{2}_{t,x}}  \quad &\text{if $\supp \hat{f} \subset \{(\xi,\tau): 2^{k} \leq |\xi| \leq 2^{k+1} \}$ } \\
                                       \infty \quad \text{otherwise}.
                                      \end{cases}
\end{equation}

Also $k \in \Z$ and $e \in \mathscr{E}$ we define the $Y_{k}^{e}$ on maps whose frequency support in the cone
\[
 \|f\|_{Y_{k}^{e}} = \begin{cases}
                                      2^{-k\frac{2s-1}{2}} \|(\i \partial_t + \Dels{s}+\i) f \|_{L^1_e L^2_{t,e^\perp}} \quad &\text{if }\supp \hat{f} \subset \{(\xi,\tau): 2^{k} \leq |\xi| \leq 2^{k+1} \}\\
                                      & \text{and }\supp \hat{f} \subset \{\xi: \xi \cdot e > 0 \text{ and } \xi \cdot e \geq \bc{}  2^{k}\} \\
                                       \infty \quad \text{otherwise}
                                      \end{cases}
\]

Then we define the space
\[
 Z_{k} := X_{k}^{\frac{1}{2}} + \sum_{e \in \mathscr{E}} Y_{k}^e
\]
equipped with the norm
\[
 \|f\|_{Z_{k}} = \inf_{f = f_1 + \sum_{e} f_{e}} \brac{\|f_1\|_{X_{k}^{\frac{1}{2}}} + \sum_{e \in \mathscr{E}}\|f_e\|_{Y_{k}^e}}
\]

For $\sigma > 0$ the main space resolution space $F^\sigma$ is induced by the norm
\begin{equation}\label{eq:Fsigma}
\|u\|_{F^\sigma} := \brac{\sum_{k = -\infty}^\infty 2^{2k\sigma} \|\Delta_{k} u\|_{Z_k}^{2} }^{\frac{1}{2}},
\end{equation}
and the space for the right-hand side is induced by the norm
\begin{equation}\label{eq:Nsigma}
\|F\|_{N^\sigma} := \brac{\sum_{k=-\infty}^\infty \brac{2^{k\sigma} \|(\i \partial_t +\Dels{s} + \i)^{{-1}} \Delta_{k} F\|_{Z_k}}^{{2}} }^{\frac{1}{2}}
\end{equation}

\subsection{Frequency estimates on the cone}
Since there are derivatives on the right-hand side of \eqref{eq:mainpdeR}, the estimates for the proof of \Cref{th:main} rely on the local smoothening property which holds for maps whose frequency is restricted to cones $\langle \xi, e\rangle \geq \bc{} |\xi|$, for a given direction $e \in \mathscr{E}$. The direction $e$ serves as the ``new time'', and the ``new space'' directions are the old time $t$ and the directions $e^\perp$. To obtain these estimates we need to obtain a new equation in the new space time and this is related to the quantity $N(\xi_{e}',\tau)$ below. The local smoothening estimates will be proved in \Cref{la:locsmooth}.

Fix $e \in \S^{n-1}$ and let $\xi \in \R^n$, $\tau \in \R$.
We write $\xi = \xi_{e,1} e + \xi_e'$ in the following way:
\begin{equation}\label{eq:xieande1decomp}
 \xi_{e}' := \xi - \langle \xi,e\rangle e, \quad \text{and} \quad  \xi_{e,1} = \langle \xi,e\rangle
\end{equation}
We also introduce for $\zeta' \in e^\perp \subset \R^n$, $\tau < 0$ such that $\brac{-\tau}^{\frac{1}{s}} >|\zeta'|^{2} $
\begin{equation}\label{eq:weirdNdef}
 N_e(\zeta',\tau) := \brac{\brac{-\tau}^{\frac{1}{s}} -|\zeta'|^{2}}^{\frac{1}{2}}
\end{equation}

When $e = (1,0,\ldots,0) \in \R^n$ we will, with a slight abuse of notation, identify for $\xi = (\xi_1,\xi') \in \R^n$ the vectors $\xi' \hat{=} \xi_{e}'\in \R^{n-1}$, and of course we have $\xi_1 \hat{=} \xi_{e,1} \in \R$.

Observe that if $s=1$
\[
-(|\xi|^2 + \tau) = \brac{N_e(\xi_{e}',\tau)+\xi_{e,1}}  \brac{N_e(\xi_{e}',\tau)-\xi_{e,1}},
\]
and if $\xi_{e,1} \geq \bc |\xi|$ the term $\brac{N_e(\xi_{e}',\tau)+\xi_{e,1}} $ is elliptic, whereas the $\brac{N_e(\xi_{e}',\tau)-\xi_{e,1}}$ corresponds to a Schr\"odinger term where the new time is $\xi_{e,1}$. 

The main role of $N_e(\xi',\tau)$ from \eqref{eq:weirdNdef} is to suitably emulate a similar property for $s <1$. This is captured in the following elementary

\begin{lemma}\label{la:Nproperties}
Let $s \in (0,1]$, $C_1 \geq 1$, $\tilde{c}> 0$. There exist $c',c'_2 > 0$ such that $ 2^{-c'} \ll \min\{ 2^{-\tilde{c}},C^{-1}_{1} \}$  and the following holds for any $k \in \Z$, $e \in \S^{n-1}$, $\xi \in \R^n$, $\tau \in \R$.

If
\begin{itemize}
\item $(C_1)^{-1} 2^k \leq |\xi| \leq C_1 2^k$, and
\item 
$\xi_{e,1}\geq 2^{k-\tilde{c}}$, and
\item $|\tau+|\xi|^{2s}| \leq 2^{2sk-c'}$
\end{itemize}
then
\begin{enumerate}
 \item \label{it:Nprops:1} We have 
 \begin{equation}\label{P3} 
  \tau + | \xi' |^{2s} \aeq-2^{2sk}
 \end{equation}
and in particular $\tau < 0$, $\brac{-\tau}^{\frac{1}{s}} >|\xi_{e}'|^{2}$, i.e. $N_e(\xi_e',\tau)$ is well defined
 \item \label{it:Nprops:2} $N(\xi_{e}',\tau) \aeq_{C_1,\tilde{c},c'} 2^k$, and
 \item \label{it:Nprops:3} we also have 
 \begin{equation}\label{eq:Xi1Nrel}
 |\xi_{e,1} - N(\xi_{e}',\tau)| \aeq_{C_1,\tilde{c}} 2^{-k(2s-1)} |\tau + |\xi|^{2s}|.
\end{equation}
\end{enumerate}

\end{lemma}
\begin{proof}
For simplicity of notation we assume w.l.o.g. $e = (1,0,\ldots,0)$.

We begin by proving \eqref{it:Nprops:1}. First we claim that we have 
\begin{equation}\label{P1}
\begin{split}
0 \leq | \xi |^{2s} - | \xi' |^{2s}  \aeq2^{2sk}.
\end{split}
\end{equation}
Since $|\xi_1| \aeq 2^k$ we need to establish equivalently,
\[
 \brac{| \xi' |^{2} + |\xi_1|^2}^s  - | \xi' |^{2s}  \aeq|\xi_1|^{2s}
\]
Dividing by $|\xi'|^{2s}$ and taking $a = \frac{|\xi_1|}{|\xi'|} \ageq 1$ we see that \eqref{P1} follows from the following elementary fact which holds for any $s \in (0,1]$,
\[
 \brac{1 + a}^s - 1 \aeq_s a^s \quad \forall a \ageq 1.
\]

Next, using \eqref{P1} and the assumptions made on $(\xi,\tau)$, we have
\begin{equation*}
\begin{split}
& 2^{2sk-c'} \geq  \tau + | \xi |^{2s}  
\\
& = \tau + | \xi' |^{2s} + \left( | \xi |^{2s} - | \xi' |^{2s} \right)
\\
& \geq  \tau + | \xi' |^{2s} + C2^{2sk}
\end{split}
\end{equation*}
And thus for $c'$ suitably large, $ -\tau - | \xi' |^{2s} \ageq 2^{2sk}$.  We also have
$$
-\tau -| \xi' |^{2s} =-\tau - |\xi|^{2s}  + |\xi|^{2s} - |\xi'|^{2s} \aleq  2^{2sk}.
$$
We thus conclude \eqref{P3}. Observe that \eqref{P3} is impossible if $\tau \geq 0$, so we must have $\tau < 0$, and again since $-\tau - | \xi' |^{2s} > 0$ we must have $(-\tau)^{\frac{1}{s}} > |\xi'|^{2s}$, thus $N(\xi',\tau)$ is well-defined. This establishes \eqref{it:Nprops:1}.

We proceed to proving \eqref{it:Nprops:2}. Notice that by definition of $N(\xi',\tau)$ we have by putting $ a = \left( -\tau - | \xi' |^{2s} \right)$ and $ b = | \xi' |^{2s}$, using also that $s \leq 1$,
\begin{equation*}
\begin{split}
&N(\xi',\tau)^{2} = - \tau^{1/s} - | \xi' |^{2} 
\\
&= ( a+ b)^{1/s} - b^{1/s} 
\\
& = \frac{a}{s} \int_{0}^{1} ( ta + b)^{\frac{1}{s} -1} dt\\
& \overset{s \leq 1}{\geq} \frac{a^{\frac{1}{s}}}{s} \int_{0}^{1} t^{\frac{1}{s} -1} dt\\
& =c_{s} a^{\frac{1}{s}}
\\
& \aeq(- \tau - | \xi' |^{2s})^{1/s}\\
&\overset{\eqref{P3}}{\aeq} 2^k.
\end{split}
\end{equation*}
The other bound is even simpler
\begin{equation*}
N(\xi',\tau)^{2} = ( - \tau - | \xi' |^{2s} + | \xi' |^{2s})^{1/s} - | \xi' |^{2}\overset{\eqref{P3}}{\aleq}  (2^{2sk} + | \xi'|^{2s})^{1/s} - | \xi'|^{2} \aleq 2^{2k}.
\end{equation*}
This proves \eqref{it:Nprops:2}. 

We proceed to proving \eqref{it:Nprops:3}. We have
\begin{equation}\label{P4}
\begin{split}
& \tau + | \xi |^{2s}
\\
& = \left( | \xi' |^{2} + \xi_1^{2} \right)^{s}-\left( | \xi' |^{2} +N(\xi',\tau)^{2} \right)^{s}.
\end{split}
\end{equation}
Using the mean value theorem for the function $ g(t)= ( | \xi' |^{2}+ t)^{s}$ we find for some $r \in (\xi_1^2,N(\xi',\tau)^2)$, i.e. in view of \eqref{it:Nprops:2} and the assumption on $\xi_1 \aeq 2^k$ for some $r \aeq 2^{2k}$,
\begin{equation*}
\begin{split}
&\tau + | \xi |^{2s}  
\\
=& s (|\xi'|^2+r)^{s-1} \, \brac{\xi_{1} - N(\xi',\tau)} \brac{\xi_{1} + N(\xi',\tau)}\\
\end{split}
\end{equation*}
Since $|\xi'| \aleq 2^k$, $\xi_1 \aeq 2^k$, $N(\xi',\tau) \aeq 2^k$ this readily implies 
 $$
 | \tau + | \xi |^{2s} | \approx_{s} 2^{k(2s-1)} | \xi_{1} - N(\xi',\tau) |
 $$
This establishes \eqref{it:Nprops:3}.
\end{proof}

Next, we use \Cref{la:Nproperties} above, to write, on the cone, the Schr\"odinger operator $\i \partial_t + \Dels{s}$ (here represented by $\frac{1}{ \tau + | \xi |^{2s} +i}$ in terms of the ``new time direction'' $e$ (here represented by $\frac{1}{\xi_{e,1} - N(\xi_{e}',\tau)) + c}$).

\begin{lemma}\label{Nprop2}
Let $s \in (1/2,1]$. 

Fix $C_{1} \geq 1, \tilde{c} > 0$. There exists $c' >0$ such that  $ 2^{-c'} \ll \min\{ 2^{-\tilde{c}},C_{1}^{-1} \}$ and the following holds: For $k\in \mathbb{Z}$, $e \in \S^{n-1}$, and $ s \in (1/2,1]$ we have 
\begin{equation}\label{NNN}
\begin{split}
& \ind_{| \xi_{e}' | \leq C_{1} 2^{k}} \eta^{+}_{[k-\tilde{c},k+\tilde{c}]}(\xi_{e,1}) \frac{ \eta_{ \leq 2sk-c'}(\tau + | \xi |^{2s})}{ \tau + | \xi |^{2s} +\i} 
\\
& = \col{\ind_{\tau + |\xi_{e}'|^{2s} \leq -2^{2s(k-c')}}}\ind_{| \xi_{e}' | \leq C_{1} 2^{k}} \eta^{+}_{[k-\tilde{c},k+\tilde{c}]}(N(\xi_{e}',\tau)) \frac{ \eta_{ \leq k-c'}(\xi_{e,1} - N(\xi_{e}',\tau))}{ K(\xi_{e}',\tau) ( \xi_{e,1}- N(\xi_{e}',\tau) + \i 2^{-k\frac{2s-1}{2}})} + E(\xi,\tau)
\end{split}
\end{equation}
Where $ | E |\aleq  2^{-2sk} + | \tau + | \xi |^{2s} |^{-2}$ and $E$ is supported in the set
$$
S := \{(\xi,\tau) \in \R^{n+1} : | \tau + | \xi |^{2s} |\aleq  2^{2sk} \text{ and } \xi_{e,1} \aeq2^{k} \},
$$
$ K (\xi_{e}',\tau)  $ is given by 
\begin{equation}\label{eq:Kdef}
K(\xi_{e}',\tau) = 2s \left( N(\xi_{e}',\tau)^{2} + | \xi_{e}' |^{2} \right)^{s-1} N(\xi_{e}',\tau)
\end{equation}
and satisfies 
\begin{equation}\label{eq:Kdefest}
| K (\xi_{e}',\tau) | \aeq2^{k(2s-1)}
\end{equation}
\end{lemma}

\begin{proof}
The statement is invariant under rotation, so we assume $e = (1,0,\ldots,0)$.

We take $c'$ from \Cref{la:Nproperties}, possibly larger so that \Cref{la:Nproperties}\eqref{it:Nprops:1} implies $\tau + |\xi_{e}'|^{2s} \leq -2^{2s(k-c')}$ in the support of the left-hand side of \eqref{NNN}.

From \Cref{la:Nproperties} we find that the support of the left hand side of \eqref{NNN} is contained in the set
\begin{equation}\label{eq:defSp}
 S':=  \{ (\xi',\tau) \in \R^{n-1}\times \R: | \xi' | \aleq  2^{k} \text{ and } | N(\xi',\tau) | \aeq2^{k} \}
\end{equation}
Therefore, 
\begin{equation*}
\begin{split}
&\ind_{| \xi' | \aleq 2^{k}}(\xi',\tau)\, \eta^{+}_{[k-\tilde{c},k+\tilde{c}]}(\xi_1) \frac{ \eta_{ \leq 2sk-c'}(\tau + | \xi |^{2s})}{ \tau + | \xi |^{2s} +\i} \\
& = \col{\ind_{\tau + |\xi'|^{2s} \leq -2^{2s(k-c')}}}\ind_{S'}(\xi',\tau)\, \eta^{+}_{[k-\tilde{c},k+\tilde{c}]}(\xi_1) \frac{ \eta_{ \leq 2sk-c}(\tau + | \xi |^{2s})}{ \tau + | \xi |^{2s} +\i} 
\end{split}
\end{equation*}
Next, we need to decompose $ \frac{1}{ \tau + | \xi |^{2s} + \i}$. To that end, we use \eqref{P4} to obtain
\begin{equation}\label{P5}
\begin{split}
& \tau + | \xi |^{2s}
\\
=& \left( | \xi_1^{2} + | \xi' |^{2} \right)^{s} - ( N(\xi',\tau)^{2} + | \xi' |^{2})^{s}
\\
& = L(\xi,\tau)\, ( \xi_1 - N(\xi',\tau) )+ K(\xi',\tau) (\xi_1 -N(\xi',\tau)).
\end{split}
\end{equation}
Here we defined $K$ as in \eqref{eq:Kdef} and
$$
L(\xi,\tau) =  \frac{ \left( | \xi' |^{2}+ \xi_1^{2}\right)^{s} - (| \xi' |^{2}+N(\xi',\tau)^{2} )^{s}  }{\xi_1-N(\xi',\tau)} - K(\xi',\tau)
$$
Observe that \eqref{eq:Kdefest} follows from \Cref{la:Nproperties}\eqref{it:Nprops:2} and the fact that $|\xi'| \aleq 2^k$.

We need to estimate $L(\xi,\tau)$, and we claim 
\begin{equation}\label{eq:Lest3.15}
| L(\xi,\tau) | \aeq2^{k(2s-2)} | \xi_1 -N(\xi',\tau) |
\end{equation}
To prove the above estimate, apply the second order Taylor approximation to the function $g(t) = ( t^{2} + | \xi' |^{2})^{s}$ to get 
$$
g(\xi_1) = g(N(\xi',\tau)) + g'(N(\xi',\tau))( \xi_1 - N) + \frac{1}{2} g''(r) (\xi_1 -N(\xi',\tau))^{2}
$$
For some point $ r$ between $ \xi_1,N(\xi',\tau)$ -- in particular we have $r \aeq 2^k$. Using the definition of $ K(\xi',\tau), L(\xi,\tau)$ we obtain 
\[
\begin{split}
& | L(\xi,\tau) | \aeq|\xi_1 -N(\xi',\tau)| | (r^{2} + | \xi' |^{2} )^{s-2}\left( r^{2} + | \xi' |^{2} + 2(s-1) r^{2}   \right) 
\\
& \overset{s>\frac{1}{2}}{\aeq}| \xi_1 -N(\xi',\tau) |\, 2^{2k(s-1)}.
\end{split}
\]
This implies \eqref{eq:Lest3.15}.

\begin{equation*}
\begin{split}
\frac{1}{ \tau + | \xi |^{2s} + \i}
& = \frac{1}{L(\xi,\tau)\, \brac{\xi_1 - N(\xi',\tau)} + K(\xi',\tau) \brac{\xi_1 -N(\xi',\tau)}+\i}
\\
& = \frac{1}{ K(\xi',\tau) \brac{ \xi_1 - N(\xi',\tau)+ \frac{\i}{2^{k(2s-1)}} } } + E_{1} + E_{2}
\end{split}
\end{equation*}
where (for better readability we now drop the arguments for $L$ and $K$)
$$
E_{1} = \frac{1}{ L\, \brac{\xi_1 - N} + K \brac{\xi_1 -N}+\i} - \frac{1}{K \left( \xi_1 - N + \frac{i}{K}\right)}
$$
and
$$
E_{2} =  \frac{1}{ K \left( \xi_1 - N+ \frac{i}{K} \right)} - \frac{1}{ K \left( \xi_1 - N+ \frac{\i}{2^{k(2s-1)}}\right)}.
$$
We show $E_{i}$ satisfy the claimed decay estimate. More precisely we show
\begin{equation}\label{eq:E12est}
| E_{i} |\aleq  \left( | \tau + | \xi |^{2s} | +1\right)^{-2} + 2^{-2sk}, \quad i=1,2
\end{equation}
We start with $E_{1}$, which we can write as
\begin{equation*}
\begin{split}
E_{1} =& \frac{ -(\xi_1 -N)\, L }{ \left( L (\xi_1 -N) + K (\xi_1 -N) + \i \right)\, \left( K ( \xi_1 -N) + \i \right)}\\
\overset{\eqref{P5}}{=}&\frac{1}{K} \frac{ -(\xi_1 -N)\, L }{ \left( \tau+|\xi|^{2s} + \i \right)\, \left( ( \xi_1 -N) + \frac{\i}{K} \right)}\\
\end{split}
\end{equation*}
Therefore, using \eqref{eq:Kdefest}, and \eqref{eq:Lest3.15}
\begin{equation*}
\begin{split}
&| E_{1} | \aeq 2^{-k(2s-1)} \frac{1}{ | \tau + | \xi |^{2s} |+1} \frac{1}{|\xi_1-N| + C2^{-k(2s-1)}}
\, 2^{k(2s-2)} \abs{\xi_{1} - N }^2  
\end{split}
\end{equation*}
 Using \Cref{la:Nproperties},
\begin{equation}\label{P7}
\begin{split}
&| E_{1} | \aeq \frac{1}{ | \tau + | \xi |^{2s} |+1} \frac{1}{\abs{\tau-|\xi|^{2s}} + 1}
\, 2^{-k2s} \abs{\tau-|\xi|^{2s}}^2  \aleq 2^{-2sk}.
\end{split}
\end{equation}
This proves \eqref{eq:E12est} for $E_{1}$.  
%

Next, we prove \eqref{eq:E12est} for $E_{2}$. We write 
\begin{equation*}
\begin{split}
E_{2} 
& = \frac{1}{K }\, \frac{ \i \brac{ \frac{1}{K} - \frac{1}{2^{k(2s-1)}}}}{ (( \xi_1 - N) + \frac{i}{K}) \brac{ ( \xi_1 -N) + \frac{\i}{2^{k(2s-1)}} }}
\end{split}
\end{equation*}
Therefore, using \Cref{la:Nproperties} we obtain 
\begin{equation*}
\begin{split}
& | E_{2} |\aleq  2^{-k(2s-1)} \left( \frac{ 2^{-k(2s-1)}}{ \big[| \xi_1 -N | + 2^{-k(2s-1)}\big]^{2} } \right)
\\
&\aleq  \frac{1}{ (2^{k(2s-1)} | \xi_1 -N | + 1)^{2}}
\\
& \aeq( | \tau + | \xi |^{2s} | + 1 )^{-2}
\end{split}
\end{equation*}
Which proves the \eqref{eq:E12est} for $E_{2}$. 

Let us recapitulate: So far we proved the following for any point $ ( \xi,\tau) $, with the indicated constraints, we have
\begin{equation*}
\begin{split}
& \frac{ \eta^{+}_{[k-\tilde{c},k+\tilde{c}]}(\xi_1) \eta_{ \leq 2sk-c'}(\tau +| \xi |^{2s})}{\tau+| \xi |^{2s}+\i} 
\\
& =  \col{\ind_{\tau + |\xi'|^{2s} \leq -2^{2s(k-c')}}}\ind_{S'}(\xi',\tau) \frac{ \eta^{+}_{[k-\tilde{c},k+\tilde{c}]}(\xi_1) \eta_{\leq 2sk-c}(\tau +| \xi |^{2s})}{K( \xi_1 - N + \frac{\i}{2^{k(2s-1)}})} + E 
\end{split}
\end{equation*}
Where $E$ satisfies the estimate 
$$
| E |\aleq  \ind_{ \xi_1 \aeq2^{k}}(\xi_1) \brac{ (\tau + | \xi |^{2s} +1)^{-2} + 2^{-2sk}}.
$$

We now write
$$
\frac{ \eta^{+}_{[k-\tilde{c},k+\tilde{c}]}(\xi_1) \eta_{\leq 2sk-c'}(\tau +| \xi |^{2s})}{K( \xi_1 - N + \frac{\i}{2^{k(2s-1)}})} = \col{\ind_{\tau + |\xi'|^{2s} \leq -2^{2s(k-c')}}}\frac{ \eta^{+}_{[k-\tilde{c},k+\tilde{c}]}(N) \eta_{\leq k-c'}(\xi_{1} -N)}{K( \xi_1 - N + \frac{\i}{2^{k(2s-1)}})}
+E_{3}.
$$
Where 
$$
E_{3} = \col{\ind_{\tau + |\xi'|^{2s} \leq -2^{2s(k-c')}}}\ind_{S'}(\xi',\tau) \frac{ \eta^{+}_{[k-\tilde{c},k+\tilde{c}]}(\xi_1)  \eta_{\leq 2sk-c'}(\tau +| \xi |^{2s}) - \eta^{+}_{[k-\tilde{c},k+\tilde{c}]}(N)\, \eta_{\leq k-c'}(\xi_1-N)}{K( \xi_1 - N + \frac{\i}{2^{k(2s-1)}})}
$$
We can clearly see that the support of $E_{3} $ is contained in the set 
$$
S= \{ \xi_1 \aeq2^{k} \text{ and } | \tau + | \xi |^{2s} | \aleq 2^{2sk}\} 
$$
Indeed, since we operate in the set $S'$, see \eqref{eq:defSp}, we know that $N \aeq 2^k$. if $\xi_1 \gg 2^k$ then $|\xi_1 -N| \gg 2^k$ and thus $E_3 = 0$. If $\xi_1 \ll 2^k$ then $|\xi_1-N| \aeq 2^k$, so taking $c'$ suitably large we again see that $\eta_{\leq k-c'}(\xi_1-N) = 0$, and thus $E =0$

Thus, for any $(\xi,\tau) \in \supp E_3$ we have $\xi_1 \aeq 2^k$. But then it is also clear that for any $(\xi',\tau) \in S$ and $\xi_1 \aeq 2^k$ we have $|\tau + |\xi|^{2s}| \aleq 2^{2sk}$. Thus, as claimed, $\supp E_3 \subset S$.
%

Next, we estimate $E_{3}$ as before (using in particular the Lipschitz estimate of $\eta$) 
\begin{equation*}
\begin{split}
| E_{3} | 
& \aleq \frac{ \abs{\eta^{+}_{[k-\tilde{c},k+\tilde{c}]}(\xi_1)  - \eta^{+}_{[k-\tilde{c},k+\tilde{c]}}(N)}\eta_{\leq 2sk-c'}(\tau +| \xi |^{2s})}{\abs{K}\, \abs{\xi_1 - N} + 2^{-k(2s-1)}}
\\
&  + \frac{ \eta^{+}_{[k-\tilde{c},k+\tilde{c}]}(\xi_1)  \abs{\eta_{\leq 2sk-c'}(\tau +| \xi |^{2s}) - \eta_{\leq k-c'}(\xi_{1,e}-N)} }{\abs{K}\, \abs{\xi_1 - N} + 2^{-k(2s-1)}}
\\
& 
\overset{\eqref{eq:Kdefest}}{\aleq} \frac{2^{-k} |\xi_1-N|}{2^{k(2s-1)} |\xi_1-N|+ 2^{-k(2s-1)}}\\
&\quad +\frac{|\eta_{ \leq 2sk-c'}(\tau + | \xi |^{2s}) - \eta_{\leq k -c'}(\xi_1 - N)|}{2^{k(2s-1)} |\xi_1-N| + 2^{-k(2s-1)}} 
\\
&\aleq \frac{2^{-k} |\xi_1-N|}{2^{k(2s-1)} |\xi_1-N|+ 2^{-k(2s-1)}}\\
& +\frac{\abs{\frac{\tau + | \xi |^{2s}}{2^{2sk-c'}} - \frac{\xi_1 - N}{2^{k -c'}}}}{2^{k(2s-1)} |\xi_1-N| + 2^{-k(2s-1)}} 
\\
&\aleq  2^{-2sk} + 2^{-2sk} \frac{\abs{\tau + | \xi |^{2s}}}{2^{k(2s-1)} |\xi_1-N| + 2^{-k(2s-1)}} \\
&\aleq 2^{-2sk},
\end{split}
\end{equation*}
where in the last step we used again \Cref{la:Nproperties} \eqref{it:Nprops:3}.

This estimates $E_3$ and we can conclude.
\end{proof}

\section{Estimates between resolution spaces}\label{s:resspaceest}
First we observe that by the usual convolution argument we have the following \Cref{la:dsbetaest} and \Cref{multiplier}. Cf. \cite[(3.8)]{BIK07}.

\begin{lemma}\label{la:dsbetaest}
For any $\ell \in \Z$, $p,q \in [1,\infty]$, $e \in \S^{n-1}$, $\beta \in \R$ we have 
\[
 \|\Ds{\beta} \Delta_{\ell} f\|_{L^p_eL^q_{e^\perp t}} \aeq 2^{\ell \beta} \sum_{\ell': 2^{\ell'} \aeq 2^\ell} \|\Delta_{\ell'} f\|_{L^p_eL^q_{e^\perp t}}
\]
\end{lemma}

\begin{lemma}\label{multiplier}
Let $ m \in L^{\infty}(\mathbb{R}^{n})$ with $ \mathcal{F}_{\R^{n}}^{-1}(m) \in L^{1}(\R^{n})$. Then for $ f \in Z_{k}$ we have that $ \mathcal{F}^{-1}_{\R^{n+1}}(m \hat{f}) \in Z_{k}$ with the estimate 
$$
\| \mathcal{F}^{-1}_{\R^{n+1}}(m \hat{f}) \|_{X_{k}}\aleq  \| \mathcal{F}^{-1}_{\R^{n}} ( m ) \|_{L^{1}(\R^{n})} \| f \|_{X_{k}}
$$
$$
\| \mathcal{F}^{-1}_{\R^{n+1}}(m \hat{f}) \|_{Y_{k}^e}\aleq  \| \mathcal{F}^{-1}_{\R^{n}} ( m ) \|_{L^{1}(\R^{n})} \| f \|_{Y_{k}^e}
$$
and thus
$$
\| \mathcal{F}^{-1}_{\R^{n+1}}(m \hat{f}) \|_{Z_{k}}\aleq  \| \mathcal{F}^{-1}_{\R^{n}} ( m ) \|_{L^{1}(\R^{n})} \| f \|_{Z_{k}}
$$
\end{lemma}

The following lemma establishes the basic embedding properties between $X_k$ and $Y_k^e$. \eqref{eq:YkisXkforkneg} in particular says that the spaces $Y_k^e$ are not any better than $X_k$ if $k \leq 0$.

\begin{lemma}\label{la:IKDIE:3.1}
For any $j \geq 0$, $k \in \Z$
\[
 \|Q_j f\|_{X_k} \aeq {2^{j\frac{1}{2}}} \|Q_j f\|_{L^2} \aleq \|f\|_{Z_{k}},
\]
more precisely we have for any $f \in Y^e_{k}$, $k \in \Z$, $e \in \S^{n-1}$
\begin{equation}\label{eq:QjXkvsY2k}
 \|Q_j f\|_{X_k} \aleq \min\{2^{k s} 2^{-\frac{1}{2}j}, 1 \} \|f\|_{Y_{k}^e},
\end{equation}
and thus in particular
\begin{equation}\label{eq:DIE:311}
 \norm{\sum_{2^j \ageq 2^{k 2s}} Q_j f}_{X_k} \aleq \|f\|_{Y^e_{k}}.
\end{equation}
and for any $k \in \Z$
\begin{equation}\label{eq:YkisXkforkneg}
\norm{f}_{X_k} \aleq 2^{ks} \|f\|_{Y_k^e}.
\end{equation}

\end{lemma}
\begin{proof}
If $f \in X_k$ the estimate 
\[
 \|Q_j f\|_{X_k} \aleq \|f\|_{X_k}
\]
is obvious. 

Assume now that $f \in Y_{k}^e$, in particular assume  \[\supp \hat{f} \subset \left \{(\xi,\tau): |\xi| \aeq {2^k} \right \} \cap \left \{\xi \cdot e \geq \bc{} 2^k\right \} \]

Write as in \eqref{eq:xieande1decomp} $\xi = \xi_{e,1} e + \xi_e'$, where $\xi_{e}' \in e^\perp$, and for fixed $\xi_{e}'$ and $\tau$ we recall the estimate \eqref{eq:claimasdkljhfs}.
\[
\begin{split}
{2^{j\frac{1}{2}}} \|Q_j f\|_{L^2(\R^{n+1})} \aeq& {2^{j\frac{1}{2}}} \|\varphi((\tau+\vert \xi \vert^{2s})/2^{j}) \hat{f}\|_{L^2(\R^{n+1})}\\
\aleq& \frac{{2^{j\frac{1}{2}}}}{1+{2^j}}\, \|   \varphi((\tau+\vert \xi \vert^{2s})/2^{j})  \mathcal{F}\brac{\brac{ \i \partial_t + \Dels{s} + \i}f}\|_{L^2(\R^{n+1} )}\\
\overset{\eqref{eq:claimasdkljhfs}}{\aleq}& \frac{{2^{j\frac{1}{2}}}}{1+{2^j}}\,  \brac{\min\{{2^k},{2^{k(1-2s)}} {2^j}\}}^{\frac{1}{2}} \| \mathcal{F}_{e,e^\perp,t}\brac{\brac{ \i \partial_t + \Dels{s} + \i}f}\|_{L^2_{e^\perp,t} L^\infty_e}\\
\aleq& \frac{{2^{j\frac{1}{2}}}}{1+{2^j}}\,  \brac{\min\{{2^k},{2^{k(1-2s)}} {2^j}\}}^{\frac{1}{2}} \| \mathcal{F}_{\col{e^\perp,t}}\brac{\brac{\i \partial_t + \Dels{s} + \i}f}\|_{L^2_{e^\perp,t} \col{L^1_e}}\\
\overset{\text{Minkowski}}{\aleq}& \frac{{2^{j\frac{1}{2}}}}{1+{2^j}}\,  \brac{\min\{{2^k},{2^{k(1-2s)}} {2^j}\}}^{\frac{1}{2}} \| \mathcal{F}_{e^\perp,t}\brac{\brac{\i \partial_t + \Dels{s} + " \i}f}\|_{\col{L^1_eL^2_{e^\perp,t} }}\\
\overset{\text{Plancherel}}{\aleq}& \frac{{2^{j\frac{1}{2}}}}{1+{2^j}}\,  \brac{\min\{{2^k},{2^{k(1-2s)}} {2^j}\}}^{\frac{1}{2}} \| {\brac{\i \partial_t + \Dels{s} + \i}f}\|_{L^1_eL^2_{e^\perp,t} }\\
\aleq& \frac{{2^{j\frac{1}{2}}}}{1+{2^j}}\,  \brac{\min\{{2^k},{2^{k(1-2s)}} {2^j}\}}^{\frac{1}{2}} {2^k}^{\frac{2s-1}{2}} \| f\|_{Y^e_{k}}\\
\end{split}
\]
That is we have shown
\begin{equation}\label{eq:DIE311ess}
\begin{split}
{2^{j\frac{1}{2}}} \|Q_j f\|_{L^2(\R^{n+1})} \aleq & \frac{{2^{j\frac{1}{2}}}}{1+{2^j}}\,  \brac{\min\{{2^k},{2^{k(1-2s)}} {2^j}\}}^{\frac{1}{2}} {2^k}^{\frac{2s-1}{2}} \| f\|_{Y^e_{k}}\\
=&\frac{{2^{j\frac{1}{2}}}}{1+{2^j}}\,  \min\{2^{k s}, 2^{\frac{1}{2} j} \}   \| f\|_{Y^e_{k}}\\
\aleq & \min\{2^{k s} 2^{-\frac{1}{2}j}, 1 \}   \| f\|_{Y^e_{k}}\\
\end{split}
\end{equation}
This is \eqref{eq:QjXkvsY2k}.

\eqref{eq:DIE311ess} shows in particular that for any $f \in Y_{k}^e$,
\[
 \|Q_j f\|_{X_k} \aleq \|f\|_{Y_{k}^e}.
\]
\eqref{eq:DIE311ess} also shows 
\[
\sum_{j \ageq 2sk} {2^{j\frac{1}{2}}} \|Q_j f\|_{L^2(\R^{n+1})} \aleq \| f\|_{Y^e_{k}}
\]
which implies \eqref{eq:DIE:311}

Lastly, for \eqref{eq:YkisXkforkneg} observe that by \eqref{eq:QjXkvsY2k}
\[
 \|f\|_{X_k} \aleq \sum_{j = 0}^\infty \|Q_j f\|_{X_k} \aleq \sum_{j = 0}^\infty 2^{k s} 2^{-\frac{1}{2}j}\|f\|_{Y_{k}^e} \aleq 2^{ks} \|f\|_{Y_{k}^e}.
\]
We can conclude.
\end{proof}

\begin{lemma}\label{la:simplessss}
Let $E: \R^{n+1}\to \R$ be a smooth function with the pointwise bound $ | E(\xi,\tau) | \aleq \min\{ 1, 2^{-2sk} | \tau +| \xi |^{2s} | \} $. Then for $e \in \S^{n-1}$ and $f \in Y^{e}_{k}$ we have the estimate
$$
\Vert \mathcal{F}^{-1}_{\R^{n+1}} ( E \hat{f} ) \Vert_{X_{k}} \aleq \Vert f \Vert_{Y^{e}_{k}}
$$
\end{lemma}
\begin{proof}
Write $ \xi = \xi_{e,1} e+ \xi_{e}$, and
notice that we may assume $ f $ is supported in the set $$\{\xi_{e,1} \approx 2^{k} \text{ and } | \tau + | \xi |^{2s} | \leq 2^{2sk} \}.  $$ Indeed, this follows from \eqref{eq:DIE:311} and the assumption that $ |E| \aleq 1$. So assume $f$ is supported in the indicated set, then we estimate (when $j=0$ we have to replace $\varphi$ by $\eta$)

\begin{equation*}
\begin{split}
& 2^{\frac{j}{2}} \Vert \varphi(( \tau+ | \xi |^{2s})/2^{j}) E \hat{f} \Vert_{L^{2}(\R^{n+1})}
\\
& \aleq 2^{\frac{j}{2}} 2^{-2sk} \Vert \varphi(( \tau+ | \xi |^{2s})/2^{j}) ( \tau +| \xi |^{2s} ) \hat{f} \Vert_{L^{2}(\R^{n+1})}
\\
& \aleq 2^{\frac{j}{2}} 2^{-2sk} \Vert \varphi((\tau+| \xi |^{2s})/2^{j}) \mathcal{F}_{\R^{n+1}} \big[ ( \i \partial_{t} +(-\Delta^{s})+\i)f\big] \Vert_{L^{2}_{\xi,\tau}}
\\
& = 2^{\frac{j}{2}}2^{-2sk} \Vert \varphi((\tau+| \xi |^{2s})/2^{j}) \int_{\R} e^{\i x_{{e},1} \xi_{{e},1}} \mathcal{F}_{{e}^{\perp},t}  \big[ ( \i \partial_{t} +(-\Delta^{s})+\i)f\big] \Vert_{L^{2}_{\xi,\tau}}
\\
& \aleq 2^{\frac{j}{2}} 2^{-2sk} \Vert \mathcal{F}_{{e}^{\perp},t}
\big[ ( \i \partial_{t} +(-\Delta^{s})+\i)f\big] \Vert_{L^{1}_{x_{{e},1}}L^{2}_{\xi'_{{e}},t}}
\\
& \times | \{ \xi_{{e},1} \approx 2^{k} :  | \tau + | \xi |^{2s} | \aleq 2^{j} \}|^{1/2}.
\end{split}
\end{equation*}
And using \Cref{la:setestimate} we obtain
\begin{equation*}
\begin{split}
& 2^{\frac{j}{2}} \Vert \varphi(( \tau+ | \xi |^{2s})/2^{j}) E \hat{f} \Vert_{L^{2}(\R^{n+1})}
\\
& \aleq 2^{j} 2^{-2sk}  2^{-k\frac{2s-1}{2}}  \Vert \mathcal{F}_{\tilde{e}^{\perp},t}
\big[ ( \i \partial_{t} +(-\Delta^{s})+\i)f\big] \Vert_{L^{1}_{x_{{e},1}}L^{2}_{\xi'_{{e}},t}}
\\
& = 2^{j} 2^{-2sk} \Vert f \Vert_{Y^{{e}}_{k}}.
\end{split}
\end{equation*}
Summing over $0 \leq j \leq 2^{2sk}$ gives the result.
\end{proof}

The next lemma are the $L^\infty L^2$-estimates with respect to the resolution space $Z_k$.
\begin{lemma}\label{la:IKDIE:3.33}
Let $n \geq 3$. For $f \in X_k$ we have
\begin{equation}\label{eq:LinftyL2Xk}
 \|f\|_{L^\infty_t L^2_x} \aleq \|f\|_{X_{k}}.
\end{equation}
Moreover if for some $e \in \S^{n-1}$ we have $f \in Y^e_{k}$ then
\begin{equation}\label{eq:LinftyL2Yk}
 \|f\|_{L^\infty_t L^2_x} \aleq \|f\|_{Y_{k}^e}.
\end{equation}
In particular we have 
\begin{equation}\label{eq:LinftyL2Zk}
 \|f\|_{L^\infty_t L^2_x} \aleq \|f\|_{Z_{k}}.
\end{equation}
and thus for any $\sigma \in \R$ 
\begin{equation}\label{eq:LinftyL2F}
 \|f\|_{L^\infty_t \dot{H}^{\sigma}_x} \aleq \|f\|_{F^{\sigma}}
\end{equation}

\end{lemma}
\begin{proof}
\eqref{eq:LinftyL2Zk} is a consequence of \eqref{eq:LinftyL2Xk} and \eqref{eq:LinftyL2Yk}. \eqref{eq:LinftyL2F} follows from squaring \eqref{eq:LinftyL2Zk} and writing $f = \sum_{k} \Delta_k f$. So we only need to prove \eqref{eq:LinftyL2Xk} and \eqref{eq:LinftyL2Yk}.

First, we \underline{prove \eqref{eq:LinftyL2Xk}}. Let $f \in X_k$. We have
\[
\|Q_j f\|_{L^\infty_t L^2_x} \aleq \|\mathcal{F} (Q_j f)\|_{L^1_\tau L^2_\xi} \aleq 2^{\frac{j}{2}} \|Q_j f\|_{L^2(\R^{n+1})}.
\]
In particular 
\[
\|f\|_{L^\infty_t L^2_x} \aleq \sum_{j=0}^\infty 2^{\frac{j}{2}} \|Q_j f\|_{L^2(\R^{n+1})} = \|f\|_{X_k}.
\]
This implies \eqref{eq:LinftyL2Xk}.

Now we are going to \underline{prove \eqref{eq:LinftyL2Yk}}.
Let $f \in Y_{k}^e$, where w.l.o.g. $e = (1,0,\ldots)$.

We write
\[
 \hat{f}(\xi,\tau) = \frac{1}{\tau + |\xi|^{2s} + \i} \mathcal{F}\brac{\brac{\partial_t + \Dels{s} + \i} f}
\]
That is, for $g := \brac{\partial_t + \Dels{s} + \i}f$,
\[
\begin{split}
 f(\hat{\xi},t) =& \int_{\tau \in \R} \frac{1}{\tau + |\xi|^{2s} + \i} \mathcal{F}\brac{g}(\xi,\tau) e^{\i \tau t}d\tau\\
 =& \int_{\tau \in \R} \frac{1}{\tau + |\xi|^{2s} + \i} \int e^{-\i x_1 \xi_1} g(x_1,\hat{\xi'},\hat{\tau}) e^{\i \tau t} dx_1 d\tau \\
 =& \int  e^{-\i x_1 \xi_1}  \int_{\tau \in \R} \frac{1}{\tau + |\xi|^{2s} + \i} g(x_1,\hat{\xi'},\hat{\tau}) e^{\i \tau t} d\tau \, dx_1 \\
 \end{split}
\]

That is
\[
\begin{split}
 \abs{f(\hat{\xi},t)}  \leq & \int_{\R}  \abs{\int_{\tau \in \R} \frac{1}{\tau + |\xi|^{2s} + \i} g(x_1,\hat{\xi'},\hat{\tau}) e^{\i \tau t} d\tau} dx_1 \\
 \end{split}
\]

Next, for fixed $\xi' \in \R^{n-1} \setminus \{0\}$ the map 
\[\begin{split}
{\mu}:& (0,\infty) \to (|\xi'|^{2s},\infty)\\
{\mu}(\xi_1) :=& |\xi|^{2s} 
\end{split}
\]
is a homeomorphism, and we have
\[
 d{\mu} = 2s |\xi|^{2s-2} \xi_1 d\xi_1
\]
In our case the assumption that $f \in Y_{k}^e$ that $|\xi'| \aleq {2^k}$ and $|\xi_1| \aeq {2^k}$ implies for any $\xi$ in with $\hat{f}(\xi,\tau) \neq 0$ we surely have (here we use $s \leq 1$)
\[
 d\xi_1 = \frac{1}{2s} |\xi|^{2-2s} \frac{1}{\xi_1}d{\mu} \leq {2^{k(1-2s)}} d{\mu}
\]
Then by Minkowski (still $\xi' \in \R^{n-1} \setminus \{0\}$ is fixed)
\[
\begin{split}
 \brac{\int_{{0}}^\infty |f(\hat{\xi},t)|^2 d\xi_1}^{\frac{1}{2}} 
 \leq& \int_{\R}   \brac{\int_{\R} \abs{\int_{\tau \in \R} \frac{1}{\tau + |\xi|^{2s} + \i} g(x_1,\hat{\xi'},\hat{\tau}) e^{\i \tau t} d\tau}^2 d\xi_1}^{\frac{1}{2}} dx_1  \\
 \aleq&{2^k}^{-\frac{2s-1}{2}} \int_{\R}   \brac{\int_{\R} \abs{\int_{\tau \in \R} \frac{1}{\tau + {\mu} + \i} g(x_1,\hat{\xi'},\hat{\tau}) e^{\i \tau t} d\tau}^2 d{\mu}}^{\frac{1}{2}} dx_1  \\
 \end{split}
\]
If we now set for fixed $x_1$, $\xi'$, $t$,
\[
 h(\tau) \equiv h_{x_1,\xi',t}(\tau) := g(x_1,\hat{\xi'},\hat{\tau}) e^{\i \tau t}
\]
and 
\[
 k(\tau) := \frac{1}{\tau+\i}
\]
we see that the interior two integrals become 
\[
 \brac{\int_{\R} \abs{\int_{\tau \in \R} \frac{1}{\tau + {\mu} + \i} g(x_1,\hat{\xi'},\hat{\tau}) e^{\i \tau t} d\tau}^2 d{\mu}}^{\frac{1}{2}} = \|k \ast h\|_{L^2(\R)}
\]
Since $k$ is a Calderon-Zygmund operator we have 
\[
\begin{split}
 \|k \ast h\|_{L^2(\R)} \aleq_k \|h\|_{L^2(\R)}
 \end{split}
\]

so that we arrive at
\[
\begin{split}
 &\brac{\int_{\R} \abs{\int_{\tau \in \R} \frac{1}{\tau + {\mu} + \i} g(x_1,\hat{\xi'},\hat{\tau}) e^{\i \tau t} d\tau}^2 d{2^j}}^{\frac{1}{2}}\\
 \aleq &\brac{\int_{\R} \abs{g(x_1,\hat{\xi'},\hat{\tau}) e^{\i \tau t}}^2 d\tau}^{\frac{1}{2}}\\
 \aleq&\brac{\int_{\R} \abs{g(x_1,\hat{\xi'},\hat{\tau})}^2 d\tau}^{\frac{1}{2}}
 \end{split}
\]
Thus we have shown
\[
\begin{split}
 \brac{\int_{{0}}^\infty |f(\hat{\xi},t)|^2 d\xi_1}^{\frac{1}{2}} 
 \leq& \int_{\R}   \brac{\int_{\R} \abs{\int_{\tau \in \R} \frac{1}{\tau + |\xi|^{2s} + \i} g(x_1,\hat{\xi'},\hat{\tau}) e^{\i \tau t} d\tau}^2 d\xi_1}^{\frac{1}{2}} dx_1  \\
 \aleq&{2^k}^{-\frac{2s-1}{2}} \int_{\R} \brac{\int_{\R} \abs{g(x_1,\hat{\xi'},\hat{\tau})}^2 d\tau}^{\frac{1}{2}} dx_1  \\
 \end{split}
\]
Essentially the same argument (noting that ${2^j}$ above is also a homeomorphism on $({-}\infty,0) \to (|\xi'|^{2s},\infty)$) leads to
\[
\begin{split}
 \brac{\int_{{-\infty}}^{0} |f(\hat{\xi},t)|^2 d\xi_1}^{\frac{1}{2}} 
 \aleq&{2^k}^{-\frac{2s-1}{2}} \int_{\R} \brac{\int_{\R} \abs{g(x_1,\hat{\xi'},\hat{\tau})}^2 d\tau}^{\frac{1}{2}} dx_1  \\
 \end{split}
\]
That is, we have for any $\xi' \in \R^{n-1} \setminus \{0\}$ and any $t \in \R$
\[
\begin{split}
 \brac{\int_{\R} |f(\hat{\xi},t)|^2 d\xi_1}^{\frac{1}{2}} 
 \aleq&{2^k}^{-\frac{2s-1}{2}} \int_{\R} \brac{\int_{\R} \abs{g(x_1,\hat{\xi'},\hat{\tau})}^2 d\tau}^{\frac{1}{2}} dx_1  \\
 \end{split}
\]
Applying once more Minkowski inequality we have
\[
\begin{split}
 \sup_{t \in \R}\| f(\hat{\xi},t)\|_{L^2(\R^n)}
 \aleq&{2^k}^{-\frac{2s-1}{2}} \int_{\R} \brac{\int_{\R^{n-1}}\int_{\R} \abs{g(x_1,\hat{\xi'},\hat{\tau})}^2 d\tau d\xi'}^{\frac{1}{2}} dx_1  \\
 \end{split}
\]
By Plancherel we thus have shown
\[
\begin{split}
 \| f\|_{L^\infty_t L^2_x(\R^n)}
 \aleq&{2^k}^{-\frac{2s-1}{2}} \|g\|_{L^1_e L^2_{e^\perp,t}}   \\
 =&{2^k}^{-\frac{2s-1}{2}} \|(\i \partial_t +\Dels{s} +\i) f\|_{L^1_e L^2_{t,e^\perp}}   \\
 =&\|f\|_{Y^e_{k}}\\
 \end{split}
\]
We can conclude.
\end{proof}

In preparation for the Local Smoothing and Maximal estimate we need the following. The following is the $s < 1$-analogue of \cite[Lemma 2.4]{IKCMP}  but using the spaces from \cite{IKDIE}.
\begin{lemma}\label{la:cmpla2.4}
Let $e \in \S^{n-1} $,  then for any $f \in Y_k^e$  there exists $h: \R \times e^\perp \times \R \to \C$ and $g$ such that
\begin{equation}\label{eq:cmpla24dec}
\begin{split}
  \hat{f}(\xi,\tau)= &  \hat{g}(\xi,\tau)\\
  &+\col{\ind_{\tau + |\xi_{e}'|^{2s} \leq -2^{2s(k-c')}}} \eta^{+}_{[k-\tilde{c},k+\tilde{c}]}(N(\xi_{e}',\tau)) \frac{  \eta_{ \leq k-c'}(\xi_{e,1}-N(\xi_{e}',\tau))}{(\xi_{e,1}-N(\xi_{e}',\tau)+\i/2^{k(2s-1)})}
\\
& \quad \times \int_{\mathbb{R}} e^{-\i y\xi_{e,1} } h(y_{1},\xi_{e}',\tau) dy,
\end{split}
\end{equation}
where we recall \eqref{eq:weirdNdef}.
Moreover $h$ and $g$ have the estimates
\[
 \|g\|_{X_k} \aleq \|f\|_{Y_k^e}
\]
and 
\[
 \|h\|_{L^1 (\R; L^2(e^\perp \times \R)) } \aleq 2^{-k\frac{2s-1}{2}} \|f\|_{Y_k^e}
\]
And $h $ is supported in the set 
\begin{equation}\label{eq:cmpla24decsupport}
\mathbb{R} \times \{(\xi_{e}',\tau) \in e^\perp \times \R:\  N(\xi_{e}',\tau) \aeq2^{k} , | \xi_{e}' | \leq C_{1} 2^{k} \}\end{equation}
\end{lemma}
\begin{proof}
If $k \leq 100$ and $f \in Y_k^e$ we set $g:=f$, $h=0$ and obtain the result by \eqref{eq:YkisXkforkneg}. So, from now on, we assume $k \geq 100$. 

By \eqref{eq:DIE:311} we may assume $ \hat{f}$ is supported in $ \{ | \tau + | \xi |^{2s} | \leq 2^{2sk-c'}\}$ where $ 2^{-c'} \ll 1$. Also, since $ f \in Y^{e}_{k}$ then $\hat{f}$ is supported in $ \{ \xi_{e,1} \aeq2^{k} \text{ and } | \xi | \leq C_{1} 2^{k} \}$. In particular, $ | \xi_{e}' | \leq C_{1} 2^{k}$. Therefore, we have equality 
$$
\hat{f}(\xi,\tau) = \eta_{| \xi_{e}' | \leq C_{1} 2^{k}} \eta^{+}_{[k-\tilde{c},k+\tilde{c}]}(\xi_{e,1}) \eta_{\leq 2sk-c'} ( \tau+| \xi |^{2s}) \hat{f}(\xi,\tau)
$$
We define $h: \R^{n+1} \to \C$ by the following
\begin{equation}\label{rr}
 \hat{f}(\xi,\tau)
 = \eta_{| \xi_{e}' | \leq C_{1} 2^{k}} \frac{\eta^{+}_{[k-\tilde{c},k+\tilde{c}]}(\xi_{e,1}) \eta_{\leq 2sk-c'} ( \tau+| \xi |^{2s})}{\tau+ | \xi |^{2s} +\i} \mathcal{F}_{\R^{n+1}} (h)(\xi,\tau)
\end{equation}
Using \Cref{Nprop2} we may write this as
\begin{equation}\label{rr3}
\begin{split}
 \hat{f}(\xi,\tau)
& = \col{\ind_{\tau + |\xi_{e}'|^{2s} \leq -2^{2s(k-c')}}}\eta_{| \xi_{e}' | \leq C_{1}2^{k}} \eta^{+}_{[k-\tilde{c},k+\tilde{c}]}(N(\xi_{e}',\tau)) \frac{ \eta_{ \leq k-c'}(\xi_{e,1} - N(\xi_{e}',\tau))}{ K(\xi_{e}',\tau) ( \xi_{e,1}- N(\xi_{e}',\tau) + \i 2^{-k\frac{2s-1}{2}})} \mathcal{F}_{\R^{n+1}}(h)
\\
&+ E(\xi,\tau)\mathcal{F}_{\R^{n+1}}(h)(\xi,\tau),
\end{split}
\end{equation}
where $ |  E |\aleq  2^{-2sk} + | \tau + | \xi |^{2} |^{-2}$.

We show that the last term in \eqref{rr3} can be absorbed into the $g$ term of \eqref{eq:cmpla2.4:goaldec}. More precisely, we are going to show
\begin{equation}\label{rr6}
\| \mathcal{F}^{-1}_{\mathbb{R}^{n+1}}( E\, \mathcal{F}_{\R^{n+1}}(h) \|_{X_{k}}
\aleq  \| f \|_{Y^{e}_{k}}.
\end{equation}
Indeed, by Plancherel we have
\begin{equation}\label{eq:cmpla2.4:goaldec}
\begin{split}
& 2^{\frac{j}{2}} \| Q_{j} \left( \mathcal{F}^{-1}_{\R^{n+1}}( E \mathcal{F}_{\R^{n+1}}(h)) \right) \|_{L^{2}(\R^{n+1})}
\\
& = 2^{\frac{j}{2}} \| \varphi((\tau+| \xi |^{2s})/2^{j}) E(\xi,\tau) \hat{h}(\xi,\tau) \|_{L^{2}}
\\
& = 2^{\frac{j}{2}}  \| \varphi((\tau+| \xi |^{2s})/2^{j}) E(\xi,\tau) \int_{\mathbb{R}} e^{-\i\xi_{e,1}y} \mathcal{F}_{x'_{e},t}(h) (ye+\xi_{e}',\tau) dy \|_{L^{2}(\R^{n+1})}
\\
& \leq 2^{\frac{j}{2}}\int_{\mathbb{R}} \| \varphi(\tau+| \xi |^{2s}/2^{j}) E(\xi,\tau) e^{-\i\xi_{e,1}y} \mathcal{F}_{x'_{e},t}(h)(ye+\xi_{e}',\tau) \|_{L^2(\R^{n+1})} dy
\\
&\aleq 2^{\frac{j}{2}} (2^{-2sk}+ 2^{-2j}) \int_{\mathbb{R}} \left( \| \mathcal{F}_{x'_{e},t}(h(ye+\cdot,\cdot)) \|_{L^{2}_{\xi_{e}',\tau}} \big[ \int_{\mathbb{R}} \ind_{| \tau+| \xi |^{2s}| \leq 2^{j}, \xi_{e,1} \aeq2^{k} } d\xi_{e,1}\big]^{1/2} \right)
\\
& \leq 2^{\frac{j}{2}} (2^{-2sk}+ 2^{-2j}) \int_{\mathbb{R}}  \| \mathcal{F}_{x'_{e},t}(h(ye + \cdot,\cdot) \|_{L^{2}_{\xi_{e}',\tau}} dy  | \{ \xi_{e,1} \aeq2^{k} : | \tau + | \xi |^{2s} | \leq 2^{j} \} |^{1/2}
\end{split}
\end{equation}
With the help of \eqref{eq:claimasdkljhfs} we find
\begin{equation}\label{rr4}
\begin{split}
&  2^{\frac{j}{2}} \| Q_{j} \left( \mathcal{F}^{-1}_{\R^{n+1}}( E \mathcal{F}_{\R^{n+1}}(h)) \right) \|_{L^{2}(\R^{n+1})}
\\
&\aleq  2^{\frac{j}{2}}( 2^{-2sk}+2^{-2j}) \min \{ 2^{\frac{k}{2}},2^{k\frac{1-2s}{2}}\, 2^{\frac{j}{2}}\} \| \mathcal{F}_{x'_{e},t}( h)(ye+\xi_{e}',\tau) \|_{L^{1}_{y},L^{2}_{\xi_{e}',\tau}}.
\end{split}
\end{equation}

Lastly, notice by Plancherel
\begin{equation}\label{rrM}
\begin{split}
& \| \mathcal{F}_{x'_{e},t}( h)(ye+\xi_{e}',\tau) \|_{L^{1}_{y},L^{2}_{\xi_{e}',\tau}}
\\
& = \| ( \i \partial_{t} + ( -\Delta^{s})+\i) f \|_{L^{1}_{e},L^{2}_{e^{\perp},t}} = 2^{k\frac{2s-1}{2}} \| f \|_{Y^{e}_{k}}.
\end{split}
\end{equation}
Plugging \eqref{rrM} into \eqref{rr4} we obtained 
\begin{equation*}
\begin{split}
& 2^{\frac{j}{2}} \| Q_{j} \left( \mathcal{F}^{-1}_{\R^{n+1}}( E \mathcal{F}_{\R^{n+1}}(h)) \right) \|_{L^{2}(\R^{n+1})}
\\
&\aleq  2^{\frac{j}{2}} ( 2^{-2sk}+2^{-2j}) \min\{ 2^{\frac{k}{2}},2^{k\frac{1-2s}{2}}2^{\frac{j}{2}}\} 2^{k\frac{2s-1}{2}} \| f \|_{Y^{e}_{k}}
\end{split}
\end{equation*}
Summing over $ j \in [0, 2sk+C]$ and using
\begin{equation}
\begin{split}
& \sum_{0 \leq j \leq 2sk+C} 2^{\frac{j}{2}} ( 2^{-2sk}+2^{-2j} ) \min\{ 2^{\frac{k}{2}},2^{k\frac{1-2s}{2}}2^{\frac{j}{2}}\} 2^{k\frac{2s-1}{2}}
\\
&\aleq  1
\end{split}
\end{equation}
This proves \eqref{rr6} and we arrive at
\begin{equation*}
\begin{split}
 \hat{f}(\xi,\tau) &= \hat{g}(\xi,\tau)\\
 &+\col{\ind_{\tau + |\xi_{e}'|^{2s} \leq -2^{2s(k-c')}}}\eta_{| \xi_{e}' | \leq C_{1} 2^{k}}\eta^{+}_{[k-\tilde{c},k+\tilde{c}]}(N(\xi_{e}',\tau))\times\\
 & \quad \times  \frac{ \eta_{ \leq k-c'}(\xi_{e,1} - N(\xi_{e}',\tau))}{ K(\xi_{e}',\tau) ( \xi_{e,1}- N(\xi_{e}',\tau) + \i 2^{-k\frac{2s-1}{2}})} \mathcal{F}_{\R^{n+1}}(h)(\xi,\tau)
\\
\end{split}
\end{equation*}
where $ \| g \|_{X_{k}}\aleq  \| f \|_{Y^{e}_{k}}$.

Next, write the Fourier transform of $h$ as
$$
\mathcal{F}_{\R^{n+1}}(h)(\xi,\tau) = \int_{\mathbb{R}} e^{-\i\xi_{e,1}y} \mathcal{F}_{x'_{e},t}(h)(ye+\xi_{e}',\tau) d \xi_{e,1}
$$
then define $h': \R \times e^\perp \times \R \to \C$ as
$$
h'(y,\xi_{e}',\tau) = \frac{1}{K(\xi_{e}',\tau)}  \mathcal{F}_{x'_{e},t}(h)(ye+\xi_{e}',\tau), \quad y \in \R,\, \xi_{e}' \in e^\perp,\, \tau \in \R
$$
Then we have succeeded obtaining \eqref{eq:cmpla24dec} (with $h'$ instead of $h$). Moreover, using \eqref{rrM} and \eqref{eq:Kdefest} we readily find
$$
\| h' \|_{L^{1}_{y}L^{2}_{\xi_{e}',\tau}}\aleq  2^{-k\frac{2s-1}{2}} \| f \|_{Y^{e}_{k}}.
$$
Lastly, \eqref{eq:cmpla24decsupport} holds because of the cuttoff functions.

\end{proof}

\subsection{Smoothing and Maximal estimates}

\begin{lemma}[Smoothing Estimate]\label{la:locsmooth}
Let $n \geq 2$, and $k \in \Z$.

Assume $f \in Z_k$ and $e \in \S^{n-1}$, then we have
$$
 \| \mathcal{F}_{\R^{n+1}}^{-1}\brac{\mathcal{F}_{\R^{n+1}}({\Tilde{f}}) \vartheta_{\xi \cdot e \ageq |\xi|} }\|_{L^\infty_e L^2_{e^\perp, t}} \aleq 2^{-k\frac{2s-1}{2}} \|f\|_{Z_k},
$$
where $ \Tilde{f} \in \{ f , \bar{f} \}$.
\end{lemma}

\begin{proof}

Notice that if we prove the lemma for $f$, then it also follows for $\bar{f}$, since we have 
\begin{equation}\label{eq:Fbarsmoothing}
 \| \mathcal{F}_{\R^{n+1}}^{-1}\brac{ \mathcal{F}_{\R^{n+1}}({\bar{f}}) \vartheta_{\xi \cdot e \ageq |\xi|}}\|_{L^\infty_e L^2_{e^\perp, t}} = \| \mathcal{F}_{\R^{n+1}}^{-1}\brac{ \mathcal{F}_{\R^{n+1}}({f}) \vartheta_{\xi \cdot (-e) \ageq |\xi|}}\|_{L^\infty_{-e} L^2_{(-e)^\perp, t}}.
\end{equation}
Indeed,
\[
 \mathcal{F} \bar{f}(\xi,\tau) \vartheta_{\xi \cdot e \ageq |\xi|} = \overline{\mathcal{F} f(-\xi,-\tau)
 \vartheta_{(-\xi) \cdot (-e) \ageq |\xi|}}
\]
so
\[
\mathcal{F}^{-1}\brac{ \mathcal{F} \bar{f}(\xi,\tau) \vartheta_{\xi \cdot e \ageq |\xi|}}(x,t) = \overline{\mathcal{F}^{-1}\brac{\mathcal{F} \brac{f(\xi,\tau)
 \vartheta_{(\xi) \cdot (-e) \ageq |-\xi|}}}(x,t)}
\]
and thus we have \eqref{eq:Fbarsmoothing}.

Therefore, we need to prove for any $f \in Z_{k}$, and any $e \in \S^{n-1}$ 
\begin{equation}\label{Reduction1}
 \| \mathcal{F}_{\R^{n+1}}^{-1}\brac{\hat{f} \vartheta_{\xi \cdot e \ageq |\xi|} }\|_{L^\infty_e L^2_{e^\perp, t}} \aleq 2^{-k\frac{2s-1}{2}} \|f\|_{Z_k}.
 \end{equation}
Assume \underline{first $f \in X_{k}$}, then we want to show
\begin{equation}\label{LocsmoothXpart}
\begin{split}
\Vert \mathcal{F}_{\R^{n+1}}^{-1} \left( \hat{f} \vartheta_{\xi \cdot e \ageq |\xi|} \right) \Vert_{L^{\infty}_{e}L^{2}_{e^{\perp},t}} \aleq 2^{-k\frac{2s-1}{2}} \Vert f \Vert_{X_{k}}.
\end{split}
\end{equation}
Summing over $j$, \eqref{LocsmoothXpart} follows once we show for any $j \in \Z_{+}$
\begin{equation}\label{ComponentsXpart}
\Vert \mathcal{F}_{\R^{n+1}}^{-1} \left( \mathcal{F}_{\R^{n+1}}(Q_{j}f) \vartheta_{\xi \cdot e \ageq |\xi|}\right) \Vert_{L^{\infty}_{e}L^{2}_{e^{\perp},t}} \aleq 2^{-k\frac{2s-1}{2}} 2^{\frac{j}{2}} \Vert Q_{j}f \Vert_{L^{2}(\R^{n+1})}.
\end{equation}

As in \eqref{eq:xieande1decomp}, we write $ \xi = \xi_{e,1} e+ \xi'_{e}$, and $x = x_{e,1}e+x'_{e}$. Then the left hand side of \eqref{ComponentsXpart} is equal to (if $j=0$ then $\varphi$ becomes $\eta$)
\begin{equation*}
\begin{split}
& \norm{ \int_{\R^{n+1}} e^{\i \xi \cdot x+ i  t \tau} \vartheta_{\xi \cdot e \ageq |\xi|} \varphi((\tau+| \xi |^{2s})/2^{j}) \hat{f}(\xi,\tau) d\xi d\tau }_{L^{\infty}_{e}L^{2}_{e^{\perp},t} }
\\
& =\norm{\int_{\R} e^{\i \xi_{e,1} x_{e,1}} \vartheta_{\xi \cdot e \ageq |\xi|} \varphi((\tau+| \xi |^{2s})/2^{j}) \hat{f}(\xi,\tau) d\xi_{e,1}  }_{L^{\infty}_{x_{e,1}}L^{2}_{\xi'_{e},\tau} },
\end{split}
\end{equation*}
where we used Plancherel. Notice that the inner integral is supported in the set $ \{ \xi_{e,1} \approx 2^{k} : | \tau+| \xi |^{2s} | \aleq 2^{j}, |\xi| \aeq 2^k \}$ which by \Cref{la:setestimate} has its $\mathcal{L}^1$-measure controlled by $2^{-k(2s-1)}2^j$. Therefore, by H\"{o}lder inequality we obtain
\begin{equation*}
\begin{split}
& \norm{\int_{\R} e^{\i \xi_{e,1} x_{e,1}} \vartheta_{\xi \cdot e \ageq |\xi|} \varphi((\tau+| \xi |^{2s})/2^{j}) \hat{f}(\xi,\tau) d\xi_{e,1}  }_{L^{\infty}_{x_{e,1}}L^{2}_{\xi'_{e},\tau} }
\\
& \aleq 2^{-k\frac{2s-1}{2}} 2^{\frac{j}{2}} \Vert Q_{j}f \Vert_{L^{2}(\R^{n+1})}.
\end{split}
\end{equation*}
Summing over $j \in \Z_{+}$ gives us \eqref{ComponentsXpart} and thus \eqref{LocsmoothXpart} is established.

Next, \underline{consider the $Y^{\tilde{e}}_{k}$-estimate} for some $ \tilde{e} \in \mathscr{E}$.

Observe that for suitably large $c > 0$
\[
 \vartheta_{\xi \cdot e \ageq |\xi|}  \eta^{+}_{[k-c,k+c]}( (\xi \cdot e)) \in C_c^\infty(\R^n)
\]
and by scaling we thus have
\[
 \|\mathcal{F} \brac{\vartheta_{\xi \cdot e \ageq |\xi|}  \eta^{+}_{[k-c,k+c]}( (\xi \cdot e)) }\|_{L^1(\R^n)} \aleq 1.
\]
Applying \Cref{multiplier}, it hence suffices to prove for any $f \in Y^{\tilde{e}}_k$,
\begin{equation}\label{LocsYpart}
\begin{split}
&  \| \mathcal{F}_{\R^{n+1}}^{-1}\brac{\hat{f} \eta^{+}_{[k-c,k+c]}( (\xi \cdot e)) }\|_{L^\infty_e L^2_{e^\perp, t}} \aleq 2^{-k\frac{2s-1}{2}} \Vert f \Vert_{Y^{\tilde{e}}_{k}}.
\end{split}
\end{equation}
Write $ \xi = \xi_{\tilde{e},1}\tilde{e}+ \xi'_{\tilde{e}}$, then carrying out the Fourier transform of the left hand side of \eqref{LocsYpart}  we obtain
\begin{equation*}
\begin{split}
& \| \mathcal{F}_{\R^{n+1}}^{-1}\brac{\hat{f} \eta^{+}_{[k-c,k+c]}( (\xi \cdot e)) }\|_{L^\infty_e L^2_{e^\perp, t}}
\\
& = \Vert \int_{\R^{n+1}} e^{\i \xi \cdot x + i t\tau} \hat{f}(\xi,\tau) \eta_{[k-c,k+c]}^{+}((\xi_{\tilde{e},1}\tilde{e}+\xi'_{\tilde{e}})\cdot e) d \xi d \tau \Vert_{L^{\infty}_{e}L^{2}_{e^{\perp},t}}
\end{split}
\end{equation*}
We would like to replace $ \xi_{\tilde{e},1} \tilde{ e}  +\xi'_{\tilde{e}} $ with $ N(\xi_{\tilde{e}},\tau) \tilde{e}+\xi'_{\tilde{e}}$. This can be done because the error term will be in the space $X_{k}$. Indeed, notice that by \eqref{eq:DIE:311}, using \eqref{LocsmoothXpart}, we may assume that $ \hat{f}$ is supported in $ \{ | \tau + | \xi |^{2s} | \leq 2^{2sk-c'} \text{ and } \xi_{\tilde{e},1} \approx 2^{k} \}$ for $ c' \gg 1$.

By Lemma \ref{la:Nproperties}, for any $ (\xi,\tau) \in {\rm supp}(\hat{f})$ the following holds
\begin{equation*}
\begin{split}
&| \eta^{+}_{[k-c,k+c]}(( \xi_{\tilde{e},1}\tilde{e}+\xi'_{\tilde{e}})\cdot e) - \eta^{+}_{[k-c,k+c]}((N(\xi'_{\tilde{e}},\tau)\tilde{e}+\xi'_{\tilde{e}})\cdot e)| \\
& \aleq \min\{1,2^{-k} | N(\xi'_{\tilde{e}},\tau) - \xi_{\tilde{e},1} | \}
\\
& \aleq \min\{1,2^{-2sk} | \tau + | \xi |^{2s} |\}
\end{split}
\end{equation*}

Thus, we can write
$$
\hat{f}(\xi,\tau) \eta^{+}_{[k-c,k+c]}(\xi \cdot e) = \hat{f}(\xi,\tau) \eta^{+}_{[k-c,k+c]}((N(\xi'_{\tilde{e}},\tau) \tilde{e}+\xi'_{\tilde{e}})\cdot e) + \hat{f}(\xi,\tau) E(\xi,\tau),
$$
with $ |E(\xi,\tau) | \aleq 2^{-2sk} | \tau + | \xi |^{2s} |$. Then using \Cref{la:simplessss} we obtain
$$
\Vert \mathcal{F}^{-1}_{\R^{n+1}}(E \hat{f}) \Vert_{X_{k}} \aleq \Vert f \Vert_{Y^{\tilde{e}}_{k}}.
$$
Then in light of \eqref{LocsmoothXpart}, to obtain \eqref{LocsYpart} it suffices to show
\begin{equation}\label{Reduc2}
\begin{split}
& \norm{\int_{\R^{n+1}} e^{\i \xi \cdot x + i t\tau} \hat{f}(\xi,\tau) \eta_{[k-c,k+c]}^{+}((N(\xi'_{\tilde{e}},\tau) \tilde{e}+\xi'_{\tilde{e}})\cdot e) d \xi d \tau }_{L^{\infty}_{e}L^{2}_{e^{\perp},t}}
\\
& \aleq 2^{-k\frac{2s-1}{2}} \norm{f }_{Y^{\tilde{e}}_{k}}
\end{split}
\end{equation}
where we implicitly assume that the support of $f$ is restricted to the $\tau$ and $\xi_{\tilde{e}}'$ such that $N(\xi'_{\tilde{e}},\tau)$ is well-defined.

To prove \eqref{Reduc2} we use Lemma \ref{la:cmpla2.4} to reduce matters to proving
\begin{equation}\label{Red3}
\begin{split}
& \Big \Vert \int_{\R^{n+1}} e^{\i \xi \cdot x+ \i t \tau}  \frac{ \eta_{\leq k -c'}(\xi_{\tilde{e},1}-N(\xi'_{\tilde{e}},\tau))}{\xi_{\tilde{e},1}-N(\xi'_{\tilde{e}},\tau) + \i/2^{k(2s-1)}} \eta_{[k-c,k+c]}^{+}(N(\xi'_{\tilde{e}},\tau))
\\
& \times  \eta_{[k-c,k+c]}^{+} ( ( N(\xi'_{\tilde{e}},\tau) \tilde{e}+\xi'_{\tilde{e}})\cdot e) \int_{\R} e^{\i \xi_{\tilde{e},1} y} h(y,\xi'_{\tilde{e}},\tau) dy d \xi d\tau \Big \Vert_{L^{\infty}_{e}L^{2}_{e^{\perp},t}}
\\
& \aleq \Vert h \Vert_{L^{1}_{y}L^{2}_{\xi'_{\tilde{e}},\tau}}
\end{split}
\end{equation}
Applying Minkowski inequality, we see that \eqref{Red3} follows once we prove
\begin{equation}\label{red4}
\begin{split}
&  \Big \Vert \int_{\R^{n+1}} e^{\i \xi \cdot x+ \i t \tau}  \frac{ \eta_{\leq k -c'}(\xi_{\tilde{e},1}-N(\xi'_{\tilde{e}},\tau))}{\xi_{\tilde{e},1}-N(\xi'_{\tilde{e}},\tau) + \i/2^{k(2s-1)}}
\\
& \times \eta_{[k-\tilde{c},k+\tilde{c}]}^{+}(N(\xi'_{\tilde{e}},\tau)) h(\xi'_{\tilde{e}},\tau) \eta^{+}_{[k-c,k+c]} (( N(\xi'_{\tilde{e}},\tau)\tilde{e}+\xi'_{\tilde{e}})\cdot e) d \xi d\tau \Big \Vert_{L^{\infty}_{e}L^{2}_{e^{\perp},t}}
\\
& \aleq \norm{h }_{L^{2}_{\xi'_{\tilde{e}},\tau}},
\end{split}
\end{equation}
for any $h \in L^{2}(\tilde{e}^{\perp} \times \R)$ and supported in $  \{ | \xi'_{\tilde{e}}| \aleq 2^{k} \text{ and } N(\xi'_{\tilde{e}},\tau) \approx 2^{k} \}$. To that end, we write
\[
\begin{split}
& \Vert \int_{\R^{n+1}} e^{\i \xi \cdot x+ \i t \tau}  \frac{ \eta_{\leq k -c'}(\xi_{\tilde{e},1}-N(\xi'_{\tilde{e}},\tau))}{\xi_{\tilde{e},1}-N(\xi'_{\tilde{e}},\tau) + \i/2^{k(2s-1)}}
\\
& \times  h(\xi'_{\tilde{e}},\tau)\eta_{[k-\tilde{c},k+\tilde{c}]}^{+}(N(\xi'_{\tilde{e}},\tau)) \eta^{+}_{[k-c,k+c]} (( N(\xi'_{\tilde{e}},\tau)\tilde{e}+\xi'_{\tilde{e}})\cdot e) d \xi d\tau \Vert_{L^{\infty}_{e}L^{2}_{e^{\perp},t}}
\\
& = \Vert \int_{\tilde{e}^{\perp} \times \R} e^{\i \xi'_{\tilde{e}} \cdot x'_{\tilde{e}}+ \i t \tau+\i N(\xi'_{\tilde{e}},\tau) x_{\tilde{e},1}}  h(\xi'_{\tilde{e}},\tau)\eta_{[k-\tilde{c},k+\tilde{c}]}^{+}(N(\xi'_{\tilde{e}},\tau)) \eta^{+}_{[k-c,k+c]} (( N(\xi'_{\tilde{e}},\tau)\tilde{e}+\xi'_{\tilde{e}})\cdot e)
\\
& \times \int_{\R} e^{ \i x_{\tilde{e},1}(\xi_{\tilde{e},1}-N(\xi'_{\tilde{e}},\tau))}  \frac{ \eta_{\leq k -c'}(\xi_{\tilde{e},1}-N(\xi'_{\tilde{e}},\tau))}{\xi_{\tilde{e},1}-N(\xi'_{\tilde{e}},\tau) + \i/2^{k(2s-1)}} d \xi_{\tilde{e},1} d\xi'_{\tilde{e}} d \tau \Vert_{L^{\infty}_{e} L^{2}_{e^{\perp},t}}
\\
& \aleq \Vert \int_{\tilde{e}^{\perp}\times \R} e^{\i \xi'_{\tilde{e}} \cdot x'_{\tilde{e}}+ \i t \tau+\i N(\xi'_{\tilde{e}},\tau) x_{\tilde{e},1}} h(\xi'_{\tilde{e}},\tau)\eta_{[k-\tilde{c},k+\tilde{c}]}^{+}(N(\xi'_{\tilde{e}},\tau)) \eta^{+}_{[k-c,k+c]} (( N(\xi'_{\tilde{e}},\tau)\tilde{e}+\xi'_{\tilde{e}})\cdot e) d \xi'_{\tilde{e}} d \tau \Vert_{L^{\infty}_{e} L^{2}_{e^{\perp},t}}.
\end{split}
\]
Where in the last inequality we used boundedness of the Fourier transform of the kernel $\frac{1}{\xi_{\tilde{e},1}}$. Therefore, we see that \eqref{red4} follows once we prove
\begin{equation}\label{red5}
\begin{split}
& \Vert\int_{\tilde{e}^{\perp}\times \R}e^{\i \xi'_{\tilde{e}} \cdot x'_{\tilde{e}}+ \i t \tau+\i N(\xi'_{\tilde{e}},\tau) x_{\tilde{e},1}} h(\xi'_{\tilde{e}},\tau)\eta_{[k-\tilde{c},k+\tilde{c}]}^{+}(N(\xi'_{\tilde{e}},\tau)) \eta^{+}_{[k-c,k+c]} (( N(\xi'_{\tilde{e}},\tau)\tilde{e}+\xi'_{\tilde{e}})\cdot e) d \xi'_{\tilde{e}} d \tau \Vert_{L^{\infty}_{e} L^{2}_{e^{\perp},t}}
\\
& \aleq \Vert h \Vert_{L^{2}_{\xi'_{\tilde{e}},\tau}}
\end{split}
\end{equation}
for any $ h \in L^{2}( \tilde{e}^{\perp} \times \R)$ supported in $\{ | \xi'_{\tilde{e}} | \aleq 2^{k} \text{ and } N(\xi'_{\tilde{e}},\tau) \approx 2^{k} \} $. To that end, introduce the change of variable $ \tau = - ( \mu^{2} + | \xi'_{\tilde{e}}|^{2})^{s}$, which is a valid change of variable because $\mu = N(\xi'_{\tilde{e}},\tau) \approx 2^{k}$. Therefore, after this change of variable for \eqref{red5} we need to show
\begin{equation}\label{red6}
\begin{split}
&  \Vert\int_{\tilde{e}^{\perp}\times \R}e^{\i \xi'_{\tilde{e}} \cdot x'_{\tilde{e}}+\i N x_{\tilde{e},1}- \i t ( \mu^{2} + | \xi'_{\tilde{e}}|^{2})^{s}} h(\xi'_{\tilde{e}},-( \mu^{2} + | \xi'_{\tilde{e}} |^{2})^{s}) \mu ( \mu^{2}+| \xi'_{\tilde{e}}|^{2})^{s-1}
\\
& \times \eta_{[k-\tilde{c},k+\tilde{c}]}^{+}(\mu) \eta^{+}_{[k-c,k+c]} (( \mu \tilde{e}+\xi'_{\tilde{e}})\cdot e) d \xi'_{\tilde{e}} d \mu \Vert_{L^{\infty}_{e} L^{2}_{e^{\perp},t}}
\\
& \aleq \Vert h \Vert_{L^{2}_{\xi'_{\tilde{e}},\tau}}
\end{split}
\end{equation}
If we define $ \tilde{h}(\xi'_{\tilde{e}},\mu) = h(\xi'_{\tilde{e}},-( \mu^{2} + | \xi'_{\tilde{e}} |^{2})^{s}) \mu ( \mu^{2}+| \xi'_{\tilde{e}}|^{2})^{s-1}$, then notice that
$$
\Vert \tilde{h} \Vert_{L^{2}_{\xi'_{\tilde{e}},\mu}} \approx 2^{k\frac{2s-1}{2}} \Vert h \Vert_{L^{2}_{\xi'_{\tilde{e}},\tau}}.
$$
Therefore, \eqref{red6} follows, once we prove
\begin{equation}\label{red7}
\begin{split}
& \Vert\int_{\tilde{e}^{\perp}\times \R}e^{\i \xi'_{\tilde{e}} \cdot x'_{\tilde{e}}+\i N(\xi'_{\tilde{e}},\tau) x_{\tilde{e},1}- it ( \mu^{2} + | \xi'_{\tilde{e}}|^{2})^{s}} h(\xi'_{\tilde{e}},\mu) \eta^{+}_{[k-c,k+c]} (( \mu \tilde{e}+\xi'_{\tilde{e}})\cdot e) d \xi'_{\tilde{e}} d \mu \Vert_{L^{\infty}_{e} L^{2}_{e^{\perp},t}}
\\
& \aleq  2^{-k\frac{2s-1}{2}} \Vert h \Vert_{L^{2}_{\xi'_{\tilde{e}},\mu}},
\end{split}
\end{equation}
for any $ h \in L^{2}(\tilde{e}^{\perp} \times \R)$ supported in $ \{ | (\xi'_{\tilde{e}},\mu) | \aleq 2^{k}\} $. Notice that we can identify the interior integral with an integral over $ \R^{n}$. Indeed, write any vector $ v \in \R^{n}$ as $ v = v_{\tilde{e},1} \tilde{e} + v'_{\tilde{e}}$. Then define $ g(v) = h(v'_{\tilde{e}},v_{\tilde{e},1})$. Then clearly, $ g$ is supported in $ | v | \aleq 2^{k}$ and $ \Vert g \Vert_{L^{2}_{v}} = \Vert h \Vert_{L^{2}_{\xi'_{\tilde{e}},\mu}}$. Further, by Fubini we have for any $ x \in \R^{n}$
\begin{equation*}
\begin{split}
& \int_{\R^{n}}e^{\i v \cdot x- it| v |^{2s}} g(v)  \eta^{+}_{[k-c,k+c]} (v \cdot e) d v
\\
& = \int_{\tilde{e}^{\perp} \times \R}e^{\i v'_{\tilde{e}} \cdot x'_{\tilde{e}}+iv_{\tilde{e},1} x_{\tilde{e},1}- \i t ( v_{\tilde{e},1}^{2} + | v'_{\tilde{e}}|^{2})^{s}} h(v'_{\tilde{e}},v_{\tilde{e},1}) \eta^{+}_{[k-c,k+c]} (( v_{\tilde{e},1} \tilde{e}+v'_{\tilde{e}})\cdot e) d v'_{\tilde{e}} d v_{\tilde{e},1}
\\
& = \int_{\tilde{e}^{\perp}\times \R}e^{\i \xi'_{\tilde{e}} \cdot x'_{\tilde{e}}+\i N(\xi'_{\tilde{e}},\tau) x_{\tilde{e},1}- \i t ( \mu^{2} + | \xi'_{\tilde{e}}|^{2})^{s}} h(\xi'_{\tilde{e}},\mu) \eta^{+}_{[k-c,k+c]} (( \mu \tilde{e}+\xi'_{\tilde{e}})\cdot e) d \xi'_{\tilde{e}} d \mu
\end{split}
\end{equation*}
Therefore, \eqref{red7} follows if we prove
\begin{equation}\label{red8}
\begin{split}
& \Vert\int_{\R^{n}}  e^{\i \xi \cdot x - \i t | \xi |^{2s}} h(\xi) \eta^{+}_{[k-c,k+c]} (\xi \cdot e) d \xi \Vert_{L^{\infty}_{e} L^{2}_{e^{\perp},t}}
\\
& \aleq 2^{-k\frac{2s-1}{2}}\Vert h \Vert_{L^{2}_{\xi}},
\end{split}
\end{equation}
for any $ h \in L^{2}(\R^{n})$ supported in $ | \xi | \aleq 2^{k}$. To that end, write $ \xi = \xi_{e,1} e + \xi'_{e}$ and $ x= x_{e,1}e+x'_{e}$ then using Plancherel we have
\begin{equation*}
\begin{split}
& \Vert\int_{\R^{n}}  e^{\i \xi \cdot x - \i t | \xi |^{2s}} h(\xi) \eta^{+}_{[k-c,k+c]} (\xi \cdot e) d \xi \Vert_{L^{\infty}_{e} L^{2}_{e^{\perp},t}}
\\
& = \Vert\int_{\R^{n}}  e^{\i \xi_{e,1}  x_{e,1} - \i t | \xi |^{2s}} h(\xi) \eta^{+}_{[k-c,k+c]} (\xi_{e,1}) d \xi_{e,1} \Vert_{L^{\infty}_{x_{e,1}}L^{2}_{\xi'_{e},t}}.
\end{split}
\end{equation*}
Fix $ x_{e,1}, \xi'_{e}$ then by duality we can find a Schwartz function $ \psi : \R \to \R $ so that $ \Vert \psi \Vert_{L^{2}} \leq 1$ and
\begin{equation*}
\begin{split}
& \Vert  \int_{\R} e^{\i \xi_{e,1} x_{e,1}-\i t | \xi |^{2s}} h(\xi) d\xi_{e,1} \eta^{+}_{[k-c,k+c]}(\xi_{e,1}) d\xi_{e,1} \Vert_{L^{2}_{t}}
\\
& \aleq |  \int_{\R} \psi(t)  \int_{\R} e^{\i \xi_{e,1} x_{e,1}-\i t | \xi |^{2s}} h(\xi) d\xi_{e,1} \eta^{+}_{[k-c,k+c]}(\xi_{e,1}) dt| .
\end{split}
\end{equation*}
By Fubini's theorem we can write the right hand side of the above as
\begin{equation}\label{Lastess}
\begin{split}
&  \abs{\int_{\R} e^{\i \xi_{e,1} x_{e,1}} h(\xi) d\xi_{e,1} \eta^{+}_{[k-c,k+c]}(\xi_{e,1})\hat{\psi}(| \xi |^{2s})d\xi_{e,1}}
\\
& \aleq \left( \int_{\R} | h(\xi) |^{2} d\xi_{e,1} \right)^{1/2} \left( \hat{\psi}(| \xi |^{2s}) \ind_{\xi_{e,1} \approx 2^{k}} d\xi_{e,1} \right)^{1/2}
\\
& \aleq 2^{-k(2s-1)/2} \left( \int_{\R} | h(\xi) |^{2} d\xi_{e,1} \right)^{1/2} \left(  \int_{\R} | \hat{\psi}(r)  |^{2} dr \right)^{1/2},
\end{split}
\end{equation}
where in the last step, for fixed $\xi_{e}'$ we used the change of variable $\xi_{e,1} \mapsto r:= | \xi|^{2s}$. Therefore, using \eqref{Lastess} we obtain for any fixed $x_{e,1}$
\begin{equation*}
\begin{split}
& \norm{\int_{\R^{n}}  e^{\i \xi \cdot x - \i t | \xi |^{2s}} h(\xi) \eta^{+}_{[k-c,k+c]} (\xi \cdot e) d \xi }_{ L^{2}_{\xi'_{e},t}}
\\
& \aleq 2^{-k\frac{2s-1}{2}} \Vert h \Vert_{L^{2}_{\xi}}.
\end{split}
\end{equation*}
This implies \eqref{red8} and thus \eqref{LocsYpart} is proven.

\end{proof}

We also need (a localized version) of the maximal estimate:

\begin{lemma}[Localized maximal estimate]\label{la:locmaxest}
Let $n \geq 3$, $e \in \S^{n-1}$, and $k \in Z$.

We have for $f \in X_k$
\begin{equation}\label{eq:max:Xkest}
  \|f\|_{L^2_e L^\infty_{e^\perp,t}} \aleq 2^{k\frac{n-1}{2}} \|f\|_{X_k}
\end{equation}
and for any $\tilde{e} \in \S^{n-1}$ and $f \in Y^{\tilde{e}}_k$
  \begin{equation}\label{eq:max:Ykest}
  \|f\|_{L^2_e L^\infty_{e^\perp,t}} \aleq 2^{k\frac{n-1}{2}} \|f\|_{Y^{\tilde{e}}_k}
 \end{equation}
In particular we have for any $f \in Z_k$
  \begin{equation}\label{eq:la:globmax}
  \|f\|_{L^2_e L^\infty_{e^\perp,t}} \aleq 2^{k\frac{n-1}{2}} \|f\|_{Z_k}
 \end{equation}

More generally let $k_1 \in \Z$ with $2^{k_1} \aleq 2^k$ with  and recall the definition in \eqref{eq:Pktell}. Then for any $f \in Z_k$
\begin{equation}\label{eq:max:Xkest:loc}
  \brac{\sum_{\tn \in 2^{k_1} \Z^n} \|P_{k_1,\tn}f\|_{L^{2}_e L^{\infty}_{e^\perp,t}}^2}^{\frac{1}{2}} \aleq
 2^{
 k\frac{(n-1)}{2}}\, 2^{-(k-k_1)\frac{n-2}{2}
 }\|f\|_{X_k}.
\end{equation}
and for any $e' \in \S^{n-1}$ and $f \in Y^{e}_k$ we have for any
\begin{equation}\label{eq:max:Ykest:loc}
  \brac{\sum_{\tn \in 2^{k_1} \Z^n} \|P_{k_1,\tn}f\|_{L^{2}_{e'} L^{\infty}_{e'^\perp,t}}^2}^{\frac{1}{2}} \aleq
 2^{
 k\frac{(n-1)}{2}}\, 2^{-(k-k_1)\frac{n-2}{2}
 } (1+|k-k_1|)\|f\|_{Y_k^e}.
\end{equation}

In particular we have for any $f \in Z_k$
\begin{equation}\label{eq:BIK:3.1:f}
 \brac{\sum_{\tn \in 2^{k_1} \Z^n} \|P_{k_1,\tn}f\|_{L^{2}_e L^{\infty}_{e^\perp,t}}^2}^{\frac{1}{2}} \aleq
 2^{
 k\frac{(n-1)}{2}}\, 2^{-(k-k_1)\frac{n-2}{2}
 }\, (1+|k-k_1|)\|f\|_{Z_k}.
 \end{equation}
All the above estimates also hold when replacing $f$ with the complex conjugated $\bar{f}$ on the left-hand side.
\end{lemma}
\begin{proof}
Clearly \eqref{eq:BIK:3.1:f} follows from \eqref{eq:max:Ykest:loc} and \eqref{eq:max:Xkest:loc}. Moreover, taking $k_1 := k$ and using \eqref{eq:weirdfinedecomp} we see that \eqref{eq:max:Ykest:loc} and \eqref{eq:max:Xkest:loc} implies \eqref{eq:max:Ykest} and \eqref{eq:max:Xkest}.

For \eqref{eq:max:Xkest}, \eqref{eq:max:Ykest}, \eqref{eq:la:globmax} it is clear that we can replace $f$ with $\bar{f}$ on the left-hand side, since the norms on the left-hand side do not change.

For \eqref{eq:max:Xkest:loc}, \eqref{eq:max:Ykest:loc}, and \eqref{eq:BIK:3.1:f} we observe 
\begin{equation*}
\begin{split}
& \|  P_{k_{1},\tn}  \bar{f}  \|_{L_{e'}L^{\infty}_{e'^{\perp},t}}
\\
& = \| \int_{\mathbb{R}^{n+1}} e^{ \i x\cdot \xi + i t.\tau} \mathcal{F}_{\R^{n+1}}(\bar{f})(\xi,\tau) \chi_{k_{1},\tn}(\xi) d\xi d\tau \|_{L^{2}_{e'},L^{\infty}_{e'^{\perp},t}}
\\
& = \| \int_{\mathbb{R}^{n+1}} e^{ \i x\cdot \xi + i t.\tau} \overline{\mathcal{F}_{\R^{n+1}}(f)(-\xi,-\tau)} \chi_{k_{1},\tn}(\xi) d\xi d\tau \|_{L^{2}_{e'},L^{\infty}_{e'^{\perp},t}}
\\
& = \| \overline{\int_{\mathbb{R}^{n+1}} e^{ -\i x\cdot \xi - i t.\tau} {\mathcal{F}_{\R^{n+1}}(f)(-\xi,-\tau)} \chi_{k_{1},\tn}(\xi) d\xi d\tau }\|_{L^{2}_{e'},L^{\infty}_{e'^{\perp},t}}
\\
& = \| \int_{\mathbb{R}^{n+1}} e^{ \i x\cdot \xi + i t.\tau} {\mathcal{F}_{\R^{n+1}}(f)(\xi,\tau)} \chi_{k_{1},-\tn}(\xi) d\xi d\tau \|_{L^{2}_{e'},L^{\infty}_{e'^{\perp},t}}
\end{split}
\end{equation*}
Therefore, summing over $\tn \in 2^{k_{1}}\mathbb{Z}^{n}$ gives us
$$
\left( \sum_{\tn \in 2^{k_{1}}\mathbb{Z}^{n}} \|  P_{k_{1},\tn}  \bar{f}  \|^{2}_{L^{2}_{e'}L^{\infty}_{e'^{\perp},t}} \right)^{1/2} = \left( \sum_{\tn \in 2^{k_{1}}\mathbb{Z}^{n}} \|  P_{k_{1},\tn}  f  \|^{2}_{L^{2}_{e'}L^{\infty}_{e'^{\perp},t}} \right)^{1/2}
$$
Hence, we only need to prove the result for $f$ and it readily follows for $\bar{f}$.

So we need to prove only \eqref{eq:max:Xkest:loc} and \eqref{eq:max:Ykest:loc}, which we do below.

\end{proof}

\begin{proof}[Proof of  \eqref{eq:locmaxfctXkgoal}]
\underline{Proof of \eqref{eq:max:Xkest:loc}}: Let $f \in X_k$. For any $j \geq 0$ we are going to prove
\begin{equation}\label{eq:locmaxfctXkgoal}
 \|P_{k_1,\tn}Q_j f\|_{L^{2}_e L^{\infty}_{e^\perp,t}} \aleq 2^{ k\frac{(n-1)}{2}}\, 2^{-(k-k_1)\frac{n-2}{2}
 }\, 2^{\frac{j}{2}}\|f\|_{L^2(\R^{n+1})}.
\end{equation}
Once we have \eqref{eq:locmaxfctXkgoal}, in view of \eqref{eq:weirdfinedecomp}, we also have 
\[
 \|P_{k_1,\tn}Q_j f\|_{L^{2}_e L^{\infty}_{e^\perp,t}}^2 \aleq \brac{2^{ k\frac{(n-1)}{2}}\, 2^{-(k-k_1)\frac{n-2}{2}
 }\, 2^{\frac{j}{2}}}^2 \sum_{\tilde{\tn} \in 2^{k_1}\Z: |\tn-\tilde{\tn}| \aleq 2^k} \sum_{\tilde{j} \in \Z: 2^{\tilde{j}} \aeq 2^j}\|P_{k_1,\tilde{\tn}}Q_{\tilde{j}} f\|_{L^2(\R^{n+1})}^2.
\]
and thus 
\[
\brac{\sum_{\tn \in 2^{k_1} \Z} \|P_{k_1,\tn}Q_j f\|_{L^{2}_e L^{\infty}_{e^\perp,t}}^2}^{\frac{1}{2}} \aleq 2^{ k\frac{(n-1)}{2}}\, 2^{-(k-k_1)\frac{n-2}{2}
 }\, 2^{\frac{j}{2}}\sum_{\tilde{j} \in \Z: 2^{\tilde{j}} \aeq 2^j}\|Q_{\tilde{j}} f\|_{L^2(\R^{n+1})}.
\]
In particular we conclude that \eqref{eq:locmaxfctXkgoal} implies \eqref{eq:max:Xkest:loc}.

So let us prove \eqref{eq:locmaxfctXkgoal}. Observe that since $2^{k_1} \aleq 2^k$ we have
\[
 P_{k_1,\tn} \Delta_k \equiv 0 \quad \text{unless $|\tn| \aleq 2^k$}
\]
We write for $g := Q_{j} f$
\[
\begin{split}
 P_{k_1,\tn}g(x,t) =& \int_{\R^{n+1}} e^{\i \brac{\langle \xi,x\rangle + \tau t}}\, \chi_{k_1,\tn}(\xi)\, \hat{g}(\xi,\tau) d\xi d\tau\\
 =& \int_{\R^{n+1}} e^{\i \brac{\langle \xi,x\rangle + |\xi|^{2s} t + \tau t}} \chi_{k_1,\tn}(\xi)\, \hat{g}(\xi,\tau-|\xi|^{2s}) d\xi d\tau\\
 =& \int_{\R^{n}} e^{\i \brac{\langle \xi,x\rangle + |\xi|^{2s} t}} \chi_{k_1,\tn}(\xi)\, \int_{\tau \aleq {2^j}} e^{\i \tau t} \hat{g}(\xi,\tau-|\xi|^{2s}) d\tau\, d\xi \\
 \end{split}
\]
Set $h(\xi) := \int_{\tau \aleq {2^j}} e^{\i \tau t} \hat{g}(\xi,\tau-|\xi|^{2s}) d\tau$, then 
\[
 |h(\xi)|^2 \aleq {2^j} \int_{\tau} |\hat{g}(\xi,\tau-\|\xi|^{2s})|^2 ={2^j} \int_{\tau} |\hat{g}(\xi,\tau)|^2
\]
and thus 
\[
\|h\|_{L^2(\R^n)} \aleq 2^{j\frac{1}{2}} \|g\|_{L^2(\R^{n+1})}.
\]
Thus, we need to show  for any $h \in C_c^\infty(B(0,2^k))$
\[
 \|\int_{\R^n} e^{\i \brac{\langle \xi,x\rangle + |\xi|^{2s} t}} \chi_{k_1,\tn}(\xi) h(\xi) d\xi \|_{L^2_e L^\infty_{e^\perp,t}} \aleq 2^{k\frac{n-1}{2}} 2^{-(k-k_1)\frac{n-2}{2}} \|h\|_{L^2(\R^n)}.
\]
Substituting $\zeta := \xi-\tn$ this is equivalent to showing 
\begin{equation}\label{oss}
\begin{split}
& \|\int_{\R^n} e^{\i \brac{\langle \zeta,x\rangle + |\zeta+\tn|^{2s} t}}  h(\zeta) d\zeta \|_{L^2_e L^\infty_{e^\perp,t}} 
\\
& \aleq 2^{k\frac{n-1}{2}} 2^{-(k-k_1)\frac{n-2}{2}} \|h\|_{L^2(\R^n)}.
\end{split}
\end{equation}
for any $h \in C_c^\infty(B(0,2^{k_1}))$ and $\tn \in 2^{k_1} \Z^n$ with $|\tn| \aleq 2^{k}$ amd $2^{k_1} \aleq 2^k$.

This follows from \Cref{th:L2eLinftyeperpvsL2est}. Thus \eqref{eq:locmaxfctXkgoal} is proven.
\end{proof}

\begin{proof}[Proof of \eqref{eq:max:Ykest:loc}]
Notice that we can assume $ f = Q_{ \leq 2sk -c} f $ where $c \gg 1$ is a fixed constant. Indeed, since by \eqref{eq:DIE:311} we have 
$$
\| Q_{ \geq 2sk-c} f \|_{X_{k}}\aleq  \| f \|_{Y^{e}_{k}}
$$
Therefore, applying \eqref{eq:max:Xkest:loc}, it suffices to prove (for a suitably large $c$)
\begin{equation*}
\begin{split}
& \left( \sum_{\tn \in 2^{k_{1}}\mathbb{Z}^{n}} \|  P_{k_{1} \tn} . Q_{\leq 2sk-c} f \|^{2}_{L^{2}_{\tilde{e}}L^{\infty}_{\tilde{e}^{\perp},t}} \right)^{1/2} \\
&\aleq  2^{k\frac{n-1}{2}} 2^{-(n-2)\frac{k-k_{1}}{2}} ( 1 + | k - k_{1} |)  \| f \|_{Y^{e}_{k}}
\end{split}
\end{equation*}

In particular we may assume w.l.o.g. that we can freely apply \Cref{la:Nproperties} for any $(\xi,\tau) \in \supp \hat{f}$.

We would like to replace $Q_{ \leq 2sk-c}$ with $ \mathcal{F}^{-1}_{\mathbb{R}^{n+1}}( \eta( \frac{ \xi_{e,1} - N(\xi_{e}',\tau)}{2^{k-c}}))$, and we can do this because the error term will be in the space $X_{k}$. Indeed, notice that by \eqref{eq:Xi1Nrel} we have
\begin{equation}\label{eq:ahmedsweirdQpsplit}
\alpha(\xi,\tau) \frac{ \tau + | \xi |^{2s}}{2^{2sk-c}} = \frac{ \xi_{e,1} - N(\xi_{e}',\tau)}{2^{k-c}}
\end{equation}
for some $ \alpha$ such that $ | \alpha | \aeq1$. Therefore, 
\begin{equation}\label{eq:ahmedsweirdQp}
\begin{split}
&|\eta( \frac{ \tau + | \xi |^{2s}}{2^{2sk-c}}) - \eta( \frac{ \xi_{e,1} - N(\xi_{e}',\tau)}{2^{k-c}}) |
\\
& \aleq 2^{-2sk} \vert \tau + \vert \xi \vert^{2s} \vec{}
\end{split}
\end{equation}
Hence, if we define $Q^{'}_{\leq k-c} f$ as 
$$
\mathcal{F}_{\mathbb{R}^{n+1}} (Q'_{\leq k-c} f)(\xi,\tau) := \eta ( \frac{ \xi_{e,1} -N(\xi_{e}',\tau)}{2^{k-c}} ) \hat{f}(\xi,\tau)
$$
then we write
$$
Q_{\leq 2sk-c} f(x,t)  =  Q'_{\leq k-c} f   + \underbrace{\brac{Q_{\leq 2sk-c} f-Q'_{\leq k-c} f}}_{Ef}.
$$
Then by Lemma \ref{la:simplessss} we have
\begin{equation}\label{Err}
\| Ef\|_{X_{k}}\aleq  \| f \|_{Y^{e}_{k}}
\end{equation}

Combining \eqref{eq:max:Xkest:loc} with \eqref{Err} it remains to prove
$$
\left( \sum_{\tn \in 2^{k_{1}}\mathbb{Z}^{n}} \|  P_{k_{1} \tn} . Q'_{\leq k-c} f \|^{2}_{L^{2}_{\tilde{e}}L^{\infty}_{\tilde{e}^{\perp},t}} \right)^{1/2}\aleq  2^{k\frac{n-1}{2}} 2^{-(n-2)\frac{k-k_{1}}{2}} ( 1 + | k - k_{1} |)  \| f \|_{Y^{e}_{k}}
$$

Next, we split 
\begin{equation}\label{v1}
\begin{split}
Q'_{\leq k -c} f
= Q'_{\leq k_{1}-c} f +  Q'_{[k_{1}-c,k-c]} \Tilde{f}
\end{split}
\end{equation}
where 
$$
\mathcal{F}_{\mathbb{R}^{n+1}} ( Q'_{[k_{1}-c,k-c]} f )(\xi,\tau) = \hat{f}(\xi,\tau) \sum_{m = k_{1}-c}^{k-c} \varphi\brac{ \frac{ \xi_{e,1}-N(\xi_{e}',\tau)}{2^{m}}}
$$
The second term of right hand side in \eqref{v1} above satisfies 
\begin{equation}\label{eq:yetanotherguy}
\| Q'_{[k_{1}-c,k-c]} f \|_{X_{k}}\aleq  (1+ |  k - k_{1} |) \| f \|_{Y^{e}_{k}}
\end{equation}
Indeed, if we have  $(\xi,\tau) \in \supp \hat{f}$ then whenever
$ | \xi_{e,1} - N(\xi_{e}',\tau) | \geq C 2^{k_{1}-c}$ we have, by \Cref{la:Nproperties},
\begin{equation*}
\begin{split}
&2^{k_{1}}\aleq  | \xi_{e,1} - N(\xi_{e}',\tau) |
\\
& \aeq2^{k(1-2s)} | \tau + | \xi |^{2s} |
\\
&\aleq  2^{k_{1}(1-2s)} | \tau + | \xi |^{2s} |,
\end{split}
\end{equation*}
i.e. $ | \tau + | \xi |^{2s} | \geq c 2^{2sk_{1}}$. Therefore, 
\begin{equation*}
\begin{split}
& \hat{f}(\xi,\tau) \sum_{m=k_{1}-c}^{k} \varphi (\frac{\xi_{e,1} -N(\xi_{e}',t)}{2^{m}})
\\
& = \eta_{[2sk_{1}-c,2sk+c']}(\tau+ | \xi |^{2s})\hat{f} (\xi,\tau)   \sum_{m=k_{1}-c}^{k} \varphi (\frac{\xi_{e,1} -N(\xi_{e}',\tau)}{2^{m}})
\end{split}
\end{equation*}
Thus, by \Cref{la:IKDIE:3.1},
\[
\begin{split}
&\|Q'_{[k_{1}-c,k-c]} f \|_{X_k}\\
\aleq&\sum_{j \in \Z_{+}: 2sk_{1}-c \leq j \leq 2sk+c} \|Q_jf\|_{L^2}\\
\aleq&\sum_{j \in \Z_{+}: 2sk_{1}-c \leq j \leq 2sk+c} \min\{2^{k s} 2^{-\frac{1}{2}j}, 1 \} \|f\|_{Y_{k}^e}\\
\aleq&\brac{1+|k-k_1|}\, \|f\|_{Y_{k}^e}.
\end{split}
\]
This establishes \eqref{eq:yetanotherguy}.

Therefore, we reduced matters into proving 
\begin{equation}\label{nn}
\begin{split}
\left( \sum_{\tn \in 2^{k_{1}}\mathbb{Z}^{n}} \| P_{k_{1},\tn} Q'_{\leq k_{1} -c} f \|^{2}_{L^{2}_{\tilde{e}}L^{\infty}_{\tilde{e}^{\perp},t}} \right)^{1/2} 
\aleq  2^{k\frac{n-1}{2}} 2^{-(n-2)\frac{k-k_{1}}{2}} ( 1 + | k - k_{1} |)  \| f \|_{Y^{e}_{k}}
\end{split}
\end{equation}
We would like to use the representation formula, \Cref{la:cmpla2.4}; however the term $P_{k_{1},\tn}(x)$ might cause trouble in the $L^1$-integral in the direction of $e$. To overcome this, we need to examine the support of the left hand side integral of \eqref{nn} in the frequency space. 

In what follows we assume implicitly that \[(\xi,\tau) \in \supp  \mathcal{F} P_{k_{1},\tn} Q'_{\leq k_{1} -c} f. \]
that is we have 
\begin{itemize}
 \item $\abs{\tau + |\xi|^{2s}} \leq 2^{2sk-c}$ (recall the assumption $f= Q_{ \leq 2sk -c} f$)
 \item $\xi_{e,1} \geq \bc |\xi|$ and $|\xi| \aeq 2^k$ (since $f \in Y_k^e$)
 \item $|\xi_{e,1} - N(\xi_{e}',\tau)| \ll 2^{k_1} \aleq 2^{k}$
 \item $|\xi_{e,1}| \aeq |N(\xi_{e}',\tau)| \aeq 2^k$
 \item $|\xi^i-\tn^i| \aeq 2^{k_1}$, for each component $i=1,\ldots,n$ 
\end{itemize}
We split $\xi = \xi_{e,1} e + \xi_{e}'$ and $ \tn = \tn_{e,1} e + \tn_{e}' $ (notice that $\tn_{e,1}$ might not be an integer) then we certainly have
\begin{itemize}
    \item $ | \xi_{e}' - \tn_{e}' |\aleq  2^{k_{1}}$ and
    \item $ | \xi_{e,1} -\tn_{e,1} |\aleq  2^{k_{1}}$
\end{itemize}
Consequently,
$$
| N(\xi_{e}',\tau)- \tn_{e,1} |\leq | N(\xi_{e}',\tau)- \xi_{e,1} | + | \xi_{e,1}- \tn_{e,1} | \aleq  2^{k_{1}}
$$
Therefore, we may write
\[
\begin{split}
 \mathcal{F}\brac{P_{k_{1},\tn} Q'_{\leq k_{1} -c} f}(\xi,\tau) = \pi_{k_{1},\tn}(\xi_{e}',N(\xi_{e}',\tau)) \mathcal{F}\brac{ P_{k_{1},\tn} Q'_{\leq k_{1} -c} f }(\xi,\tau)
 \end{split}
\]
where for $ ( \xi_{e}',R) \in e^{\perp} \times \R$. 
$$
\pi_{k_{1},\tn}( \xi_{e}',R) = \chi( \frac{R-\tn_{e,1}}{2^{k_{1}+C}}) \, \eta( \frac{ ( \xi_{e}' -\tn_{e}')}{2^{k_{1}+C}}),
$$
and $ \chi : \mathbb{R} \to \mathbb{R}$ is defined in \eqref{IntegFunc}.

We denote for $\tilde{f}$ with suitable support
\[
 A_{k_{1},\tn} \tilde{f}  := \mathcal{F}^{-1} \brac{ \pi_{k_{1},\tn}(\xi_{e}',N(\xi_{e}',\tau)) \mathcal{F} \tilde{f}}
\]
and arrive at 
\[
\begin{split}
 P_{k_{1},\tn} Q'_{\leq k_{1} -c} f= A_{k_{1},\tn} {P_{k_{1},\tn} Q'_{\leq k_{1} -c} f }.
 \end{split}
\]
Observe that by \Cref{simplestuff} we have 
\begin{equation}\label{eq:onlyfinite}
\begin{split}
& | \sum_{\tn \in 2^{k_{1}} \Z^{n}} \pi_{k_{1},\tn}(\xi_{e}',N(\xi_{e}',\tau)) |\aleq  1
\end{split}
\end{equation}

Then we reduced proving \eqref{nn} to showing
\begin{equation}\label{pa1}
\begin{split}
&\left( \sum_{\tn \in 2^{k_{1}}\mathbb{Z}^{n}} \|  A_{k_{1},\tn}  P_{k_{1},\tn} Q_{\leq k_{1}-c}f\|^{2}_{L^{2}_{\tilde{e}}L^{\infty}_{\tilde{e}^{\perp},t}} \right)^{1/2} 
\\
&\aleq  2^{k\frac{n-1}{2}} 2^{-(n-2)\frac{k-k_{1}}{2}}  \| f \|_{Y^{e}_{k}}
\end{split}
\end{equation}
that is we want to prove
\begin{equation}\label{paaa}
\begin{split}
&\left( \sum_{\tn \in 2^{k_{1}}\mathbb{Z}^{n}} \| \int_{\mathbb{R}^{n+1}} e^{\i (x\cdot \xi+t\tau)}  \pi_{k_{1},\tn}(\xi_{e}',N(\xi_{e}',\tau)) \mathcal{F}_{\mathbb{R}^{n+1}} \left( P_{k_{1},\tn} Q_{\leq k_{1}-c}f \right)d \xi d\tau \|^{2}_{L^{2}_{\tilde{e}}L^{\infty}_{\tilde{e}^{\perp},t}} \right)^{1/2}
\\
&\aleq  2^{k\frac{n-1}{2}} 2^{-(n-2)\frac{k-k_{1}}{2}}  \| f \|_{Y^{e}_{k}}.
\end{split}
\end{equation}

We now observe it is enough to prove 
\begin{equation}\label{pa2}
\begin{split}
& \| \int_{\mathbb{R}^{n+1}} e^{\i (x\cdot \xi+t\tau)}  \pi_{k_{1},\tn}(\xi_{e}',N(\xi_{e}',\tau)) \mathcal{F}_{\mathbb{R}^{n+1}} \left( P_{k_{1},\tn} Q_{\leq k_{1}-c}f \right)d \xi d\tau \|_{L^{2}_{\tilde{e}}L^{\infty}_{\tilde{e}^{\perp},t}}
\\
&\aleq  2^{k\frac{n-1}{2}} 2^{-(n-2)\frac{k-k_{1}}{2}}  \| f \|_{Y^{e}_{k}}
\end{split}
\end{equation}
or equivalently
\[
\begin{split}
& \| \int_{\mathbb{R}^{n+1}} e^{\i (x\cdot \xi+t\tau)}  \pi_{k_{1},\tn}(\xi_{e}',N(\xi_{e}',\tau)) \mathcal{F}_{\mathbb{R}^{n+1}} \left( P_{k_{1},\tn} Q_{\leq k_{1}-c}f \right)d \xi d\tau \|_{L^{2}_{\tilde{e}}L^{\infty}_{\tilde{e}^{\perp},t}}
\\
&\aleq  2^{k\frac{n-1}{2}} 2^{-(n-2)\frac{k-k_{1}}{2}} \| \int_{\mathbb{R}}e^{\i \xi_{1}x_{1}} \pi_{k_{1},\tn} (\xi_{e}',N(\xi_{e}',\tau)) (\tau + | \xi |^{2s} +\i) \hat{f}(\xi,\tau) d\xi_{e,1} \|_{L^{1}_{x_{e,1}}L^{2}_{\xi_{e}',\tau}}.
\end{split}
\]
Indeed, having \eqref{pa2} summing over $ \tn \in 2^{k_{1}} \mathbb{Z}^{n}$ we obtain
\begin{equation}\label{md1}
\begin{split}
& \left( \sum_{\tn \in 2^{k_{1}}\mathbb{Z}^{n}} \| \int_{\mathbb{R}^{n+1}} e^{\i (x\cdot \xi+t\tau)}  \pi_{k_{1},\tn}(\xi_{e}',N(\xi_{e}',\tau)) \mathcal{F}_{\mathbb{R}^{n+1}} \left( P_{k_{1},\tn} Q_{\leq k_{1}-c}f \right)d \xi d\tau \|^{2}_{L^{2}_{\tilde{e}}L^{\infty}_{\tilde{e}^{\perp},t}} \right)^{1/2}
\\
&\aleq  2^{k\frac{n-1}{2}} 2^{-(n-2)\frac{k-k_{1}}{2}}
\\
& \times \left( \sum_{\tn \in 2^{k_{1}} \mathbb{Z}^{n}} \| \int_{\mathbb{R}}e^{\i \xi_{e,1}x_{1}} \pi_{k_{1},\tn} (\xi_{e}',N(\xi_{e}',\tau)) (\tau + | \xi |^{2s} +\i) \hat{f}(\xi,\tau) d\xi_{e,1} \|^{2}_{L^{1}_{x_{1}}L^{2}_{\xi',\tau}}\right)^{1/2}
\end{split}
\end{equation}
Using Minkowski we obtain 
\begin{equation}\label{md2}
\begin{split}
&  \left( \sum_{\tn \in 2^{k_{1}} \mathbb{Z}^{n}} \| \int_{\mathbb{R}}e^{\i \xi_{e,1}x_{1}} \pi_{k_{1},\tn} (\xi_{e}',N(\xi_{e}',\tau)) (\tau + | \xi |^{2s} +\i) \hat{f}(\xi,\tau) d\xi_{e,1} \|^{2}_{L^{1}_{x_{e,1}}L^{2}_{\xi_{e}',\tau}}\right)^{1/2}
\\
& \leq \int_{x_{e,1}} \left( \int_{\xi_{e}',\tau} \sum_{ \tn \in 2^{k_{1}} \mathbb{Z}^{n}}  |\pi_{k_{1},\tn}(\xi_{e}',N(\xi_{e}',\tau))|^{2} \abs{\int_{\mathbb{R}} e^{\i \xi_{e,1}.x_{e,1}} ( \tau + | \xi |^{2s}+\i) \hat{f}(\xi,\tau) d\xi_{e,1} }^{2}  \right)^{1/2} dx_{e,1}
\\
&\overset{\eqref{eq:onlyfinite}}{\aleq}   \int_{\mathbb{R}} \left( \norm{\int_{\mathbb{R}} e^{\i \xi_{e,1}.x_{e,1}} ( \tau + | \xi |^{2s}+\i) \hat{f}(\xi,\tau) d\xi_{e,1} }^{2}_{L^{2}_{\xi_{e}',\tau}} \right)^{1/2}
\\
\end{split}
\end{equation}
Plugging \eqref{md2} into \eqref{md1} gives us \eqref{paaa}. 

Therefore, we only need to prove \eqref{pa2}. 

We notice that by \Cref{multiplier} we have 
$$
\|  P_{k_{1},\tn}f \|_{Y^{e}_{k}}\aleq  \| f \|_{Y^{e}_{k}}
$$
With constant independent of $k_{1},\tn$. 

Therefore, we reduced matters into proving 
\begin{equation}\label{pa3}
\begin{split}
&\| \int_{\mathbb{R}^{n+1}} e^{\i \xi.x + \i t\tau}\pi_{k_{1},\tn}(\xi_{e}',N(\xi_{e}',\tau)) \hat{f}(\xi,\tau)\, \eta_{\leq k_{1} -c}( \xi_{e,1} -N(\xi_{e}',\tau))) \|_{L^{2}_{\tilde{e}}L^{\infty}_{\tilde{e}^{\perp},t}} 
\\
&\aleq  2^{k\frac{n-1}{2}} 2^{-(n-2)\frac{k-k_{1}}{2}}  \| f \|_{Y^{e}_{k}} 
\end{split}
\end{equation}
for any $\tn \in 2^{k_{1}} \mathbb{Z}^{n}$. 

Now we can use the representation formula, \Cref{la:cmpla2.4} to reduce matters into proving

\begin{equation}\label{pa4}
\begin{split}
& 
2^{-(2s-1)k/2} \Big \| \int_{e^{\perp} \times \mathbb{R}} \int_{\mathbb{R}} e^{\i x\cdot \xi+\i t\tau} \pi_{k_{1},\tn} (\xi_{e}',N(\xi_{e}',\tau))  \ind_{N \aeq2^{k}} 
\\
& \times \frac{\eta_{ \leq k_{1}-c}(\xi_{e,1}-N(\xi_{e}',\tau))}{ (\xi_{e,1}-N(\xi_{e}',\tau)+\i/2^{k(2s-1)})} h(\xi_{e}',\tau) d\xi_{e,1} d \xi_{e}' d\tau \Big \|_{L^{2}_{\tilde{e}}L^{\infty}_{\tilde{e}^{\perp},t}} 
\\
&\aleq   2^{k\frac{n-1}{2}} 2^{-(n-2)\frac{k-k_{1}}{2}} \| h \|_{L^{2}_{\xi_{e}',\tau}}
\end{split}
\end{equation}
For any $h \in L^{2}(e^{\perp} \times \R)$ and supported in $ | \xi_{e}' |\aleq  2^{k}$.

Indeed, assume we have \eqref{pa4} then by \Cref{la:cmpla2.4} and  \eqref{eq:max:Xkest} we have
\begin{equation*}
\begin{split}
& \| \int_{\mathbb{R}^{n+1}} e^{\i \xi.x + \i t\tau}\pi_{k_{1},\tn}(\xi_{e}',N(\xi_{e}',\tau)) \hat{f}(\xi,\tau). \eta_{\leq k_{1} -c}( \xi_{e,1} -N)) \|_{L^{2}_{\tilde{e}}L^{\infty}_{\tilde{e}^{\perp},t}} 
\\
&=  \| \int_{e^{\perp} \times \mathbb{R}} \int_{\mathbb{R}} e^{\i x\cdot \xi+\i t\tau} \pi_{k_{1},\tn} (\xi_{e}',N(\xi_{e}',\tau))  \chi_{N \aeq2^{k}}
\\
& \times \frac{\eta_{ \leq k_{1}-c}(\xi_{e,1}-N(\xi_{e}',\tau))}{ (\xi_{e,1}-N(\xi_{e}',\tau)+\i/2^{k(2s-1)})}\int_{\R} e^{i\xi_{e,1}y} h(y,\xi_{e}',\tau) d\xi_{e,1} d \xi_{e}' d\tau \|_{L^{2}_{\tilde{e}}L^{\infty}_{\tilde{e}^{\perp},t}} 
\\
&+\|f\|_{Y_k^e}\\
& \leq  \int_{\R} \| \int_{e^{\perp} \times \mathbb{R}} \int_{\mathbb{R}} e^{\i x\cdot \xi+\i t\tau} \pi_{k_{1},\tn} (\xi_{e}',N(\xi_{e}',\tau))  \chi_{N \aeq2^{k}} 
\\
& \times \frac{\eta_{ \leq k_{1}-c}(\xi_{e,1}-N(\xi_{e}',\tau))}{ (\xi_{e,1}-N(\xi_{e}',\tau)+\i/2^{k(2s-1)})} h(y,\xi_{e}',\tau) d\xi_{e,1} d \xi_{e}' d\tau \|_{L^{2}_{\tilde{e}}L^{\infty}_{\tilde{e}^{\perp},t}} dy\\
&+\|f\|_{Y_k^e}\\
\end{split}
\end{equation*}
Here we used Minkowski inequality in the last line. Therefore, using \eqref{pa4} and the fact that $ 2^{k\frac{2s-1}{2}} \| h \|_{L^{1}_{y}L^{2}_{\xi_{e}',\tau}}\aleq  \| f \|_{Y^{e}_{k}}$ we see that \eqref{pa3} follows.

Thus, it suffices to only prove \eqref{pa4}.

Carrying out the integral in $ \xi_{e,1}$ first, then using boundedness of the Fourier transform of the kernel $ \frac{1}{\xi_{e,1}}$ we are left to prove 
\begin{equation}\label{pa5}
\begin{split}
& 
2^{-(2s-1)k/2} \| \int_{e^{\perp}} \int_{\mathbb{R}} e^{x_{e,1} N(\xi_{e}',\tau)\i} e^{\i \xi_{e}'.x'_{e}} e^{\i t\tau} \pi_{k_{1},\tn} (\xi_{e}',N(\xi_{e}',\tau))
\\
& \times \ind_{N \aeq2^{k}}(N) h(\xi_{e}',\tau)  d \xi_{e}' d\tau \|_{L^{2}_{\tilde{e}}L^{\infty}_{\tilde{e}^{\perp},t}} 
\\
&\aleq   2^{k\frac{n-1}{2}} 2^{-(n-2)\frac{k-k_{1}}{2}} \| h \|_{L^{2}_{\xi_{e}',\tau}}
\end{split}
\end{equation}
Introduce the change of variable $\tau = - ( \mu^{2} + | \xi_{e}' |^{2} )^{s}$ then this means $\mu =N(\xi_{e}',\tau)$  and this is valid change of variable because $ \mu \aeq2^{k}$. Then what we want to prove becomes 
\begin{equation}\label{pa6}
\begin{split}
& 2^{-(2s-1)k/2} \| \int_{e^{\perp}} \int_{\mathbb{R}} e^{x_{e,1} \mu \i} e^{\i \xi_{e}'.x_{e}'} e^{-\i t( \mu^{2} + | \xi_{e}' |^{2})^{s}} \pi_{k_{1},\tn} (\mu,\xi_{e}')  \chi_{\mu \aeq2^{k}}
\\
& \times h(\xi_{e}',-(\mu^{2}+ | \xi_{e}' |^{2})^{s}) \alpha(\xi_{e}',\mu)  d \xi' d\tau \|_{L^{2}_{\tilde{e}}L^{\infty}_{\tilde{e}^{\perp},t}} 
\\
&\aleq   2^{k\frac{n-1}{2}} 2^{-(n-2)\frac{k-k_{1}}{2}} \| h \|_{L^{2}_{\xi_{e}',\tau}}
\end{split}
\end{equation}
Where $ \alpha(\xi_{e}',\mu) = -2s \mu ( \mu^{2} + | \xi_{e}' |^{2})^{s-1}$.
If we define 
$$
h'(\xi_{e}',\mu) = 2^{-(2s-1)k/2} h(\xi_{e}', - ( \mu^{2}+| \xi_{e}' |^{2})^{s}) \alpha( \xi_{e}',\mu)
$$ then  notice that we have
$$
\| h' \|_{L^{2}_{\xi_{e}',\mu}} \aeq\| h \|_{L^{2}_{\xi_{e}',\tau}}
$$
Therefore, \eqref{pa6} follows if we prove 
\begin{equation*}
\begin{split}
&  \| \int_{e^{\perp}} \int_{\mathbb{R}} e^{x_{e,1} \mu \i} e^{\i \xi_{e}'.x'_{e}} e^{-\i t( \mu^{2} + | \xi' |^{2})^{s}} \pi_{k_{1},n} (\mu,\xi_{e}') h'(\xi_{e}',\mu)  d \xi_{e}' d\mu \|_{L^{2}_{\tilde{e}}L^{\infty}_{\tilde{e}^{\perp},t}}
\\
&\aleq   2^{k\frac{n-1}{2}} 2^{-(n-2)\frac{k-k_{1}}{2}} \| h'\|_{L^{2}_{\xi_{e}',\mu}}
\end{split}
\end{equation*}
For any $h'$ supported in $ | ( \xi_{e}',\mu) |\aleq  2^{k}$.
Introduce the change of variable $(\mu,\xi_{e}') \to (\mu+\tn_{e,1},\xi_{e}'+\tn_{e}')$ to reduce matters to proving 
\begin{equation}
\begin{split}
&  \| \int_{e^{\perp}} \int_{\mathbb{R}} e^{x_{e,1} \mu \i} e^{\i \xi_{e}'.x_{e}'} e^{-\i t( (\mu+\tn_{e,1})^{2} + | \xi_{e}'+\tn_{e}' |^{2})^{s}} h'(\xi_{e}',\mu)  d \xi' d\mu \|_{L^{2}_{\tilde{e}}L^{\infty}_{\tilde{e}^{\perp},t}}
\\
&\aleq   2^{k\frac{n-1}{2}} 2^{-(n-2)\frac{k-k_{1}}{2}} \| h'\|_{L^{2}_{\xi_{e}',\mu}}
\end{split}
\end{equation}
For any $h'$ supported in $ | (  \xi_{e}',\mu) |\aleq  2^{k_{1}}$.
We identify the inner integral above with the an integral on $\mathbb{R}^{n}$. More precisely, write any vector $v \in \mathbb{R}^{n}$ as $ v = v_{e,1} e + v'_{e}$. Then define $g$ as
$$
g(v) = h'( v_{e}',v_{e,1}) 
$$
Then clearly $g$ is supported in $ | v |\aleq  2^{k_{1}}$ and we have by Fubini
\begin{equation*}
\begin{split}
& \int_{\mathbb{R}^{n}} e^{\i v .x}  e^{-it | v + \tn |^{2s}} g(v) dv
\\
& \int_{e^{\perp}} \int_{\mathbb{R}} e^{x_{e,1} v_{e,1} \i} e^{\i v_{e}'.x_{e}'} e^{-\i t( (v_{e,1}+\tn_{e,1})^{2} + | v_{e}'+\tn_{e}' |^{2})^{s}} h'(v_{e}',v_{e,1})  d v_{e}' dv_{e,1}
\\
& = \int_{e^{\perp}} \int_{\mathbb{R}} e^{x_{e,1} \mu \i} e^{\i \xi_{e}'.x_{e}'} e^{-\i t( (\mu+\tn_{e,1})^{2} + | \xi_{e}'+\tn_{e}' |^{2})^{s}} h'(\xi_{e}',\mu)  d \xi_{e}' d\mu
\end{split}
\end{equation*}
Therefore, we only need to prove 
\begin{equation}
\begin{split}
&  \| \int_{\mathbb{R}^{n}} e^{\i v .x}  e^{-\i t | v + \tn |^{2s}} g(v) dv \|_{L^{2}_{\tilde{e}}L^{\infty}_{\tilde{e}^{\perp},t}}
\\
&\aleq  2^{k\frac{n-1}{2}} 2^{-(n-2)\frac{k-k_{1}}{2}} \| g \|_{L^{2}(\R^{n})}
\end{split}
\end{equation}

For any $ g$ supported in $ | v |\aleq  2^{k_{1}}$. But this is precisely \Cref{th:L2eLinftyeperpvsL2est}.

\end{proof}

\section{Linear estimates for the Schr\"odinger equation}\label{s:linearest}
\begin{lemma}[homogeneous estimates]\label{la:homogeneous}
We have for any $\sigma  \in \R$
\[
 \|e^{\i \Dels{s}t} u_0\|_{F^{\sigma}} \aleq_{\psi} \|u_0\|_{\dot{H}^{\sigma}}
\]
\end{lemma}
\begin{proof}
Observe that 
\[
 \mathcal{F} \brac{e^{\i \Dels{s}t} \Delta_{k}u_0} = \delta_{|\xi|^{2s} = -\tau} \Delta_{k} u_0.
\]
In particular,
\[
 Q_j e^{\i \Dels{s}t} \Delta_{k}u_0 = 0 \quad \forall j \geq 1.
\]
Thus,
\[
\begin{split}
 \|\Delta_{k} e^{\i \Dels{s}t}u_0\|_{Z_k}\leq&\|\Delta_{k} e^{\i \Dels{s}t}u_0\|_{X_k}\\
 =& \|Q_0 e^{\i \Dels{s}t}\Delta_k u_0\|_{L^2}\\
 \aleq& \brac{\int_{\R^{n+1}} \abs{\delta_{|\xi|^{2s}=-\tau} \mathcal{F}\brac{\Delta_{k} u_0}(\xi)| \chi_{\abs{|\xi|^{2s}+\tau} \aleq 1}}^2}^{\frac{1}{2}}\\
 =&\brac{\int_{\R^{n}} \abs{\mathcal{F}\brac{\Delta_{k} u_0}(\xi)| }^2}^{\frac{1}{2}}\\
 =&\|\Delta_{k} u_0\|_{L^2(\R^n)}\\
 \end{split}
\]
In  particular, by the definition of $F^\sigma$, \eqref{eq:Fsigma}, we have
\[
 \|e^{\i \Dels{s}t}u_0\|_{F^\sigma} \aleq \|u_0\|_{\dot{H}^\sigma(\R^d)}.
\]

We can conclude.
\end{proof}

For the inhomogeneous estimate we need a cutoff function.
\begin{lemma}[Inhomogeneous estimate]\label{la:inhomogeneous}
Let $\psi \in C_c^\infty(\R)$. Set
\[
 u(x,t) := -\psi(t) \i \int_0^t e^{\i (t-\tilde{t}) \Dels{s} } F(x,\tilde{t})\, d\tilde{t}
\]

Then for any $\sigma \in \R$
\[
\|u\|_{F^\sigma} \aleq_{\psi} \|F\|_{N^{\sigma}}.
\]
\end{lemma}

\begin{proof}
We are interested in computing $\|\Delta_{k} u\|_{Z_k}$ in terms of $\|(\i \partial_t +\Dels{s} +\i)^{{-1}} \Delta_{k} F\|_{Z_k}$, so we write
\[
 \mathcal{F}_{x,t} \Delta_{k}u(\xi,\tau) = -\i\int_{\R} \frac{\mathcal{F}_{\R^{n+1}}\brac{(\i \partial_t +\Dels{s} +\i)^{{-1}} \Delta_{k}F}(\xi,\tilde{\tau})\,  (\tilde{\tau}+|\xi|^{2s}+\i)}{\tilde{\tau}+|\xi|^{2s}} \brac{\hat{\psi}(\tau-\tilde{\tau})-\hat{\psi}(\tau+|\xi|^{2s}) }d\tilde{\tau}
\]
Thus, setting 
\[
 \begin{split}
 T h (\xi,\tau):=& \mathcal{F}^{-1} \brac{\int_{\R} \frac{\hat{h}(\xi,\tilde{\tau})\,  (\tilde{\tau}+|\xi|^{2s}+\i)}{\tilde{\tau}+|\xi|^{2s}} \brac{\hat{\psi}(\tau-\tilde{\tau})-\hat{\psi}(\tau+|\xi|^{2s}) }d\tilde{\tau}}.
 \end{split}
\]
it suffices to prove
\begin{equation}\label{eq:ThZkest:goal}
  \left \|Th \right \|_{Z_k} \aleq  \|h\|_{Z_k},
\end{equation}
and the claim follows from taking 
\[
 h= (\i \partial_t +\Dels{s} +\i)^{{-1}} \Delta_{k}F.
\]
So let us prove \eqref{eq:ThZkest:goal}.

First we observe that \underline{if $h \in X_k$}, then $\supp \mathcal{F}(Th) \subset \{(\xi,\tau): |\xi| \aeq 2^k\}$, so we can try to estimate the $X_k$-norm.

We can write $h = \sum_{j=0}^\infty Q_j h$, and thus we consider
\[
 \sum_{\ell =0}^\infty 2^{\ell \frac{1}{2}} \|Q_{\ell} T Q_j h\|_{L^2}
\]
We estimate
\[
\begin{split}
&|\mathcal{F}_{n+1}(T Q_j h)(\xi,\tau)|\\
= &\abs{\int_{\tilde{\tau} \in \R:\, \abs{\tilde{\tau}+|\xi|^{2s}} \aeq 2^j} \frac{\widehat{Q_{j} h}(\xi,\tilde{\tau})  (\tilde{\tau}+|\xi|^{2s}+\i)}{\tilde{\tau}+|\xi|^{2s}} \brac{\hat{\psi}(\tau-\tilde{\tau})-\hat{\psi}(\tau+|\xi|^{2s}) }d\tilde{\tau}}\\
 \leq&\int_{\tilde{\tau} \in \R:\, \abs{\tilde{\tau}+|\xi|^{2s}} \aeq 2^j} \abs{\widehat{Q_{j} h}(\xi,\tilde{\tau}) } \brac{\abs{\hat{\psi}(\tau-\tilde{\tau})}+\abs{\hat{\psi}(\tau+|\xi|^{2s})} }d\tilde{\tau}\\
&+\int_{\tilde{\tau} \in \R:\, \abs{\tilde{\tau}+|\xi|^{2s}} \aeq 2^j} \abs{\widehat{Q_{j} h}(\xi,\tilde{\tau}) }\frac{\abs{\hat{\psi}(\tau-\tilde{\tau})-\hat{\psi}(\tau+|\xi|^{2s}) }}{\abs{\tilde{\tau}+|\xi|^{2s}}} d\tilde{\tau}\\
 \aleq_{\psi,N}&\int_{\tilde{\tau} \in \R:\, \abs{\tilde{\tau}+|\xi|^{2s}} \aeq 2^j} \abs{\widehat{Q_{j} h}(\xi,\tilde{\tau}) } \brac{\frac{1}{1+|\tau-\tilde{\tau}|^N}+\frac{1}{1+|\tau+|\xi|^{2s}|^N} }d\tilde{\tau}\\
 \end{split}
\]
In the last step we used that $\hat{\psi}$ is a Schwartz function and thus for any $N \in \N$ we have 
\[
 \frac{\abs{\hat{\psi}(b)-\hat{\psi}(a)}}{|b-a|} \aleq_{\psi,N} \frac{1}{1+|a|^N} + \frac{1}{1+|b|^N}.
\]
Next, if $2^{\ell } \aeq 2^j$ we observe that (since $|\tilde{\tau} + |\xi|^{2s}| \aeq 2^j \aeq 2^{\ell } \aeq |\tau + |\xi|^{2s}|$,
\[
 |\tau - \tilde{\tau}| \leq |\tau + |\xi|^{2s}|+ |\tilde{\tau} + |\xi|^{2s}| \aeq 2^{\ell } \aeq |\tau+|\xi|^{2s}|
\]
and thus 
\[
 \frac{1}{1+|\tau+|\xi|^{2s}|^N}  \aleq \frac{1}{1+|\tau-\tilde{\tau}|^N} 
\]
So in this case we estimate 
\[
 |\mathcal{F}_{n+1}(T Q_j h)(\xi,\tau)| \aleq_{\psi,N}\int_{\R} \abs{\widehat{Q_{j} h}(\xi,\tilde{\tau}) } \frac{1}{1+|\tau-\tilde{\tau}|^N} d\tilde{\tau}.
\]
Then convolution rules yield
\[
 \|\mathcal{F}_{n+1}(T Q_j h)(\xi,\tau)\|_{L^2(d\tau)} \aleq \|\widehat{Q_{j} h}\|_{L^2(d\tau)}.
\]
Consequently we have shown
\[
 \|Q_{\ell} T Q_j h\|_{L^2(\R^{n+1})} \aleq \|Q_{j} h\|_{L^2(\R^{n+1})} \quad \text{whenever $2^{\ell } \aeq 2^j$}.
\]

If on the other hand $2^{\ell } \not \aeq 2^j$ we have \[|\tau-\tilde{\tau}| \geq \max\{2^j,2^{\ell }\} - \min\{2^j,2^{\ell }\} \aeq  \max\{2^j,2^{\ell }\} \geq 2^{\ell }.\]
In this case we have
\[
\begin{split}
 \abs{\mathcal{F}_{n+1}(T Q_j h)(\xi,\tau)}
 \aleq_{\psi,N}&\brac{\frac{1}{1+2^{\ell  N}} } \int_{\tilde{\tau} \in \R:\, \abs{\tilde{\tau}+|\xi|^{2s}} \aeq 2^j} \abs{\widehat{Q_{j} h}(\xi,\tilde{\tau}) } d\tilde{\tau}\\
 \aleq&2^{j\frac{1}{2}} \brac{\frac{1}{1+2^{\ell  N}} } \brac{\int \abs{\widehat{Q_{j} h}(\xi,\tilde{\tau}) }^2 d\tilde{\tau}}^{\frac{1}{2}}
 \end{split}
\]
so that 
\[
 \sum_{2^{\ell } \not \aeq 2^j}  \|Q_{\ell} T Q_j h\|_{L^2(\R^{n+1})} \aleq 2^{j\frac{1}{2}}\|Q_j h\|_{L^2(\R^{n+1})}
\]
Summing over the $2^j$ we see that 
\begin{equation}\label{eq:inhomo:ThXlambda}
 \|Th\|_{X_k} \aleq \|h\|_{X_k}
\end{equation}
which is the first contribution towards \eqref{eq:ThZkest:goal}.

\underline{Assume now $h \in Y_k^e$}. We need to show
\begin{equation}\label{eq:inhomo:ThYlambdagoal}
 \|Th\|_{Z_k} \aleq \|h\|_{Y_k^e}.
\end{equation}

We write 
\[
 \hat{h} = \underbrace{\frac{\partial_t + \Dels{s}}{\partial_t + \Dels{s}+\i} \hat{h}}_{=:\hat{h}_1} + \underbrace{\frac{\i}{\partial_t + \Dels{s}+\i} \hat{h}}_{=:\hat{h}_2}
\]
From \Cref{la:IKDIE:3.1} we have 
\[
 \|Q_j h_2 \|_{X_k} \aleq \frac{1}{1+2^j} \|Q_j h_2 \|_{X_k} \aleq \frac{1}{1+2^j} \|Q_j h \|_{X_k} \aleq \frac{1}{1+2^j} \|h \|_{Y_k^e}
\]

so that 
\[
 \|h_2\|_{X_k} \aleq \sum_{j=0}^\infty \|Q_j h_2\|_{X_k} \aleq  \|h\|_{Y_k^e}.
\]
With the help of \eqref{eq:inhomo:ThXlambda} we find
\[
 \|Th_2\|_{X_k} \aleq \|h_2 \|_{X_k} \aleq  \|h\|_{Y_k^e}.
\]
So in order to establish \eqref{eq:inhomo:ThYlambdagoal}, we need to show 
\[
 \|Th_{{1}}\|_{Z_k} \aleq \|h\|_{Y_k^e}.
\]
Set 
\[
 \begin{split}
 T_1 h :=& \mathcal{F}^{-1} \brac{\int_{\R} \frac{\hat{h}(\xi,\tilde{\tau})  (\tilde{\tau}+|\xi|^{2s}+\i)}{\tilde{\tau}+|\xi|^{2s}} \brac{\hat{\psi}(\tau-\tilde{\tau})}d\tilde{\tau}}\\
 \end{split}
\]
and
\[
 \begin{split}
 T_2 h :=& \mathcal{F}^{-1} \brac{\hat{\psi}(\tau+|\xi|^{2s})\, \int_{\R} \frac{\hat{h}(\xi,\tilde{\tau})  (\tilde{\tau}+|\xi|^{2s}+\i)}{\tilde{\tau}+|\xi|^{2s}} d\tilde{\tau} }\\
 \end{split}
\]
We have $Th_1 = T_1h_1 + T_2h_1$, and by the definition of $h_1$ we find
\[
 \begin{split}
 T_1 h_1 :=& \mathcal{F}^{-1} \brac{\int_{\R} \hat{h}(\xi,\tilde{\tau})  \hat{\psi}(\tau-\tilde{\tau})d\tilde{\tau}}\\
 \end{split}
\]
and
\[
 \begin{split}
 T_2 h_1 :=& \mathcal{F}^{-1} \brac{\hat{\psi}(\tau+|\xi|^{2s})\, \int_{\R} \hat{h}(\xi,\tilde{\tau}) d\tilde{\tau} }\\
 \end{split}
\]

We claim 
\begin{equation}\label{eq:T1h1Ylambda}
 \|T_1 h_1 \|_{Z_k} + \|T_2 h_1 \|_{X_k} \aleq \|h\|_{Y_k^e}
\end{equation}
For the second term we see from the Schwartz property of $\hat{\psi}$
\[
\begin{split}
 &\|Q_j \brac{T_2  h_1}\|_{L^2(\R^{n+1})}^2\\
 \aleq &\int_{\R^{n+1}} \chi_{\abs{|\xi|^{2s}-\tau} \aeq 2^j} \frac{1}{1+\abs{|\xi|^{2s}-\tau}^N} \brac{\int_{\R} \hat{h}(\xi,\tilde{\tau}) d\tilde{\tau} }^2 d\xi d\tau\\
 \aleq &\int_{\R^n} \brac{\int_{\R} \hat{h}(\xi,\tilde{\tau}) d\tilde{\tau} }^2 d\xi\ \int_{\tau} \chi_{\abs{|\xi|^{2s}-\tau} \aeq 2^j} \frac{1}{1+\abs{|\xi|^{2s}-\tau}^N} d\tau\\
  \aleq &\int_{\R^n} \brac{\int_{\R} \hat{h}(\xi,\tilde{\tau}) d\tilde{\tau} }^2 d\xi\ \frac{2^j}{1+2^{jN}}\\
  =&\int_{\R^n} \brac{\int_{\R} h(x,\hat{\tilde{\tau}}) d\tilde{\tau} }^2 dx\ \frac{2^j}{1+2^{jN}}\\
  =&\int_{\R^n} \brac{h(x,0)}^2 dx\ \frac{2^j}{1+2^{jN}}\\
  \leq& \|h\|_{L^\infty_t L^2_x}\ \frac{2^j}{1+2^{jN}}\\
  \aleq& \|h\|_{Y_k^e} \frac{2^j}{1+2^{jN}}.
 \end{split}
\]
In the last step we used \Cref{la:IKDIE:3.33}.

In particular we have 
\[
 \|T_2 h_1\|_{X_k} = \sum_{j=0}^\infty 2^{j\frac{1}{2}} \|Q_j \brac{T_2  h_1}\|_{L^2(\R^{n+1})} \aleq \|h\|_{Y_k^e}.
\]
This takes care of the second part of \eqref{eq:T1h1Ylambda}, and it remains to show
\begin{equation}\label{eq:T1h1Ylambdav2}
 \|T_1 h_1 \|_{Z_k} \aleq \|h\|_{Y_k^e}.
\end{equation}
We write 
\[
 \mathcal{F}(h_1)(\xi,\tau) = \hat{h}_1(\xi,\tau)\, \frac{\tilde{\tau}+|\xi|^{2s}+\i}{\tau + |\xi|^{2s} + \i} + \hat{h}_1(\xi,\tau)\, \frac{\tau-\tilde{\tau}}{\tau + |\xi|^{2s}+\i}
\]
and thus we can begin the estimate \eqref{eq:T1h1Ylambdav2} by writing
\[
\begin{split}
 \|T_1 h_1 \|_{Z_k}
 \aleq&\norm{\mathcal{F}^{-1} \brac{\frac{1}{\tau + |\xi|^{2s}+\i} \int_{\R} (\tilde{\tau}+|\xi|^{2s}+\i) \hat{h}(\xi,\tilde{\tau})  \hat{\psi}(\tau-\tilde{\tau})d\tilde{\tau}}}_{Y_k^e}\\
 &+\norm{\mathcal{F}^{-1} \brac{\frac{1}{\tau + |\xi|^{2s}+\i} \int_{\R} \hat{h}(\xi,\tilde{\tau})  \hat{\psi}(\tau-\tilde{\tau}) (\tau-\tilde{\tau}) d\tilde{\tau}}}_{X_k}.
 \end{split}
\]
For the first term
\[
\begin{split}
 &\norm{\mathcal{F}^{-1} \brac{\frac{1}{\tau + |\xi|^{2s}+\i} \int_{\R} (\tilde{\tau}+|\xi|^{2s}+\i) \hat{h}(\xi,\tilde{\tau})  \hat{\psi}(\tau-\tilde{\tau})d\tilde{\tau}}}_{Y_k^e}\\
  =&2^{-k\frac{2s-1}{2}}\norm{\mathcal{F}^{-1} \brac{\int_{\R} \mathcal{F}\brac{(\i \partial_t +\Dels{s} + \i)h}(\xi,\tilde{\tau})  \hat{\psi}(\tau-\tilde{\tau})d\tilde{\tau}}}_{L^1_e L^2_{t,e^\perp}}\\
  =&2^{-k\frac{2s-1}{2}}\norm{ \psi \brac{(\i \partial_t +\Dels{s} + \i)h}}_{L^1_e L^2_{t,e^\perp}}\\
  \aleq&2^{-k\frac{2s-1}{2}}\|(\i \partial_t +\Dels{s} + \i)h\|_{L^1_e L^2_{t,e^\perp}}\\
  =&\|h\|_{Y^e_k}.
 \end{split}
\]
For the second term (and last) term
\[
\begin{split}
 &\norm{\mathcal{F}^{-1} \brac{\frac{1}{\tau + |\xi|^{2s}+\i} \int_{\R} \hat{h}(\xi,\tilde{\tau})  \hat{\psi}(\tau-\tilde{\tau}) (\tau-\tilde{\tau}) d\tilde{\tau}}}_{X_k}\\
 =&\sum_{{\ell }=0}^\infty \sum_{j=0}^\infty 2^{\ell \frac{1}{2}} \norm{\brac{\frac{1}{\tau + |\xi|^{2s}+\i} q_{2^{\ell }}(|\xi|-\tau) \int_{\R} \widehat{Q_j h}(\xi,\tilde{\tau})  \hat{\psi}(\tau-\tilde{\tau}) (\tau-\tilde{\tau}) d\tilde{\tau}}}_{L^2(\R^{n+1})}\\
 \aleq&\sum_{2^{\ell } \aeq 2^j} 2^{\ell {-\frac{1}{2}}}\, \norm{\int_{\R} \hat{h}(\xi,\tilde{\tau})  \hat{\psi}(\tau-\tilde{\tau}) (\tau-\tilde{\tau}) d\tilde{\tau}}_{L^2(\R^{n+1})} \\ 
 &+\sum_{2^{\ell } \not \aeq 2^j} 2^{\ell {-\frac{1}{2}}} \norm{q_{2^{\ell }}(|\xi|-\tau) \int_{\R} \widehat{Q_j h}(\xi,\tilde{\tau})  \hat{\psi}(\tau-\tilde{\tau}) (\tau-\tilde{\tau}) d\tilde{\tau}}_{L^2(\R^{n+1})}.
 \end{split}
\]
Observe again that if $|\tau + |\xi|^{2s}| \aeq 2^{\ell }$ and $|\tilde{\tau} + |\xi|^{2s}| \aeq 2^j$ and $2^j \not \aeq 2^{\ell }$ then
\[
 |\tau - \tilde{\tau}| \aeq \max\{2^j,2^{\ell }\}
\]
Using that $\hat{\psi}$ is a Schwartz function, we arrive at 
\[
\begin{split}
 &\norm{\mathcal{F}^{-1} \brac{\frac{1}{\tau + |\xi|^{2s}+\i} \int_{\R} \widehat{Q_j h}(\xi,\tilde{\tau})  \hat{\psi}(\tau-\tilde{\tau}) (\tau-\tilde{\tau}) d\tilde{\tau}}}_{X_k}\\
 \aleq&\sum_{2^{\ell } \aeq 2^j} 2^{\ell {-\frac{1}{2}}}\, \norm{\int_{\R} \widehat{Q_j h}(\xi,\tilde{\tau})  \hat{\psi}(\tau-\tilde{\tau}) (\tau-\tilde{\tau}) d\tilde{\tau}}_{L^2(\R^{n+1})} \\ 
 &+\sum_{2^{\ell } \not \aeq 2^j} 2^{\ell {-\frac{1}{2}}} \frac{1}{1+\max\{2^j,2^{\ell }\}^N}\norm{q_{2^{\ell }}(|\xi|-\tau) \int_{\R} \abs{\widehat{Q_j h}} \chi_{|\tilde{\tau}-|\xi|^{2s} \aeq 2^j} d\tilde{\tau}}_{L^2(\R^{n+1})}\\
 \aleq&\sum_{2^{\ell } \aeq 2^j} 2^{\ell {-\frac{1}{2}}}\, \norm{\widehat{Q_j h}}_{L^2(\R^{n+1})} \\ 
 &+\sum_{{\ell }=0}^\infty \sum_{j=0}^\infty 2^{\ell {-\frac{1}{2}}} \frac{2^{j\frac{1}{2}} 2^{\ell \frac{1}{2}}} {1+\max\{2^j,2^{\ell }\}^N}\norm{\widehat{Q_j h}}_{L^2(\R^{n+1})}\\
 \overset{\text{\Cref{la:IKDIE:3.1}}}{\aleq}
 &\sum_{2^{\ell } \aeq 2^j} 2^{\ell {-\frac{1}{2}}}\, \|h\|_{Y_k^e} \\
 &+\sum_{{\ell }=0}^\infty \sum_{j=0}^\infty 2^{\ell {-\frac{1}{2}}} \frac{2^{j\frac{1}{2}} 2^{\ell \frac{1}{2}}} {1+\max\{2^j,2^{\ell }\}^N}\|h\|_{Y_k^e}\\
 \aleq&\|h\|_{Y_k^e}
 \end{split}
\]
This establishes \eqref{eq:inhomo:ThYlambdagoal}, and combining this with  \eqref{eq:inhomo:ThXlambda}
  we have proven \eqref{eq:ThZkest:goal}. We can conclude.
\end{proof}

\section{Trilinear estimates}\label{s:trilinearest}
The main goal of this section is the following trilinear estimate 
\begin{theorem}\label{th:trilinear}
Let $s \in (1/2,1)$ and $n \geq 4$. For any $-\frac{2s-1}{2} \leq \beta \leq 2s-1$ and $\sigma \geq \frac{n-2s}{2}$ we have
Then
\begin{equation}\label{eq:BIK2007:Nsigmaest1}
\begin{split}
 \|\brac{\Ds{-\beta} \brac{\tilde{f}_1\, \tilde{f}_2}} |\, \Ds{\beta} \tilde{f}_3\|_{N^{\sigma}}  \aleq& \|f_1\|_{F^{\sigma}}\, \|f_2\|_{F^{\frac{n-2s}{2}}}\, \|f_{3}\|_{F^{\frac{n-2s}{2}}}\\
 &+ \|f_1\|_{F^{\frac{n-2s}{2}}}\, \|f_2\|_{F^{\sigma}}\, \|f_{3}\|_{F^{\frac{n-2s}{2}}}\\
 &+ \|f_1\|_{F^{\frac{n-2s}{2}}}\, \|f_2\|_{F^{\frac{n-2s}{2}}}\, \|f_{3}\|_{F^{\sigma}}
 \end{split}
\end{equation}
where $\tilde{f}_i$ is either $f_i$ or the complex conjugated $\overline{f_i}$, $i=1,2,3$.
\end{theorem}

\begin{lemma}\label{la:BIK5.4}
Let $s \in (1/2,1]$, $\beta \leq  2s-1$ and $n \geq 4$. Then
\[
 \|\Delta_{\col{\aleq} 2^k} \Ds{-\beta} \brac{\tilde{f} \tilde{g}}\|_{L^1_e L^{\infty}_{e^\perp,t}} \aleq 2^{k(2s-1-\beta)} \|f\|_{F^{\frac{n-2s}{2}}}\, \|g\|_{F^{\frac{n-2s}{2}}}
 \]
 where $\tilde{f} \in \{f,\bar{f}\}$, $\tilde{g} \in \{g,\bar{g}\}$.
\end{lemma}
\begin{proof}
Since $\beta \leq 2s-1$, it suffices to show 
\begin{equation}\label{eq:BIK5.4:goal}
 \sum_{j \in \Z} 2^{-j(2s-1)}\|\Delta_{j} \brac{\tilde{f} \tilde{g}}\|_{L^1_e L^{\infty}_{e^\perp,t}} \aleq \|f\|_{F^{\frac{n-2s}{2}}}\, \|g\|_{F^{\frac{n-2s}{2}}}
\end{equation}
Indeed, once we have \eqref{eq:BIK5.4:goal} we obtain the claim, since
\[
\begin{split}
&\|\Delta_{\aleq 2^k} \Ds{-\beta} \brac{\tilde{f}\tilde{g}}\|_{L^1_e L^{\infty}_{e^\perp,t}}\\
\aleq&\sum_{j: 2^j \aleq 2^k} 2^{-j\beta}\|\Delta_{j} \brac{\tilde{f}\tilde{g}}\|_{L^1_e L^{\infty}_{e^\perp,t}}\\
\overset{\beta \leq 2s-1}{\aleq}&2^{k(2s-1-\beta)} \sum_{j: 2^j \aleq 2^k} 2^{-j(2s-1)}\|\Delta_{j} \brac{\tilde{f}\tilde{g}}\|_{L^1_e L^{\infty}_{e^\perp,t}}\\
\end{split}
\]
So let us prove \eqref{eq:BIK5.4:goal}.

We have
\[
\begin{split}
\|\Delta_{j} \brac{\tilde{f}\tilde{g}}\|_{L^1_e L^{\infty}_{e^\perp,t}} \aleq& \sum_{2^\ell \aleq 2^j}\| \Delta_{\ell} \tilde{f}\, \Delta_{\leq j+C}\tilde{g}\|_{L^1_e L^{\infty}_{e^\perp,t}} +\sum_{2^\ell \aleq 2^j}\| \Delta_{\leq j+C}\tilde{f}\, \Delta_{\ell} \tilde{g}\|_{L^1_e L^{\infty}_{e^\perp,t}} \\
&+\sum_{2^{\ell_1} \aeq 2^{\ell_2} \ageq 2^j}\| \Delta_{j} \brac{\Delta_{\ell_2}\tilde{f}\, \Delta_{\ell_2} \tilde{g}}\|_{L^1_e L^{\infty}_{e^\perp,t}}.
\end{split}
\]
With H\"older's inequality (here the complex conjugation vanishes), and the local maximal estimate, \Cref{la:locmaxest} \eqref{eq:la:globmax},
 \[
\begin{split}
 \| \Delta_{\ell} \tilde{f}\, \Delta_{\leq j+C}\tilde{g}\|_{L^1_e L^{\infty}_{e^\perp,t}} \aleq&\| \Delta_{\ell} f\|_{L^{2}_e L^\infty_{e^\perp,t}} \|\Delta_{\leq j+C}g\|_{L^2_e L^{\infty}_{e^\perp,t}}\\
\aleq&2^{\ell \frac{n-1}{2}} \| f\|_{Z_\ell} 2^{j\frac{n-1}{2}} \|g\|_{Z_j}\\
\aleq&2^{\ell \frac{2s-1}{2}} \, 2^{j\frac{2s-1}{2}}  2^{\ell\frac{n-2s}{2}}\| f\|_{Z_\ell}  (2^j)^{\frac{n-2s}{2}} \|g\|_{Z_j}\\
 \end{split}
 \]
Set $\gamma := \frac{2s-1}{2}$ and 
\begin{equation}\label{eq:trilinear:alskdj1}
a_\ell := 2^{\ell\frac{n-2s}{2}}\| f\|_{Z_\ell} \text{ and } b_j := (2^j)^{\frac{n-2s}{2}} \|g\|_{Z_j}\end{equation}

Then, since $\gamma > 0$ we have 
\[
\begin{split}
 \sum_{j} 2^{-(2s-1)j} \sum_{\ell \leq j} 2^{\gamma \ell}\, a_\ell 2^{\gamma j} b_j =&\sum_{j} 2^{-2\gamma j} \sum_{\ell \leq j} 2^{\gamma \ell}\, a_\ell 2^{\gamma j} b_j\\
 =&\sum_{j} \sum_{\ell \leq j} 2^{\gamma (\ell-j)}\, a_\ell b_j\\
 =&\sum_{j} \sum_{\tilde{\ell} \leq 0} 2^{\gamma \tilde{\ell}}\, a_{\tilde{\ell}+j} b_j\\
 \aeq&\sum_{j} a_{\tilde{\ell}+j} b_j\\
 \leq& [a_j]_{\ell^2} [b_j]_{\ell^2}
 \end{split}
\]
Applying this to the situation above we find 
\[
 \sum_{2^\ell \aleq 2^j}\| \Delta_{\ell} \tilde{f}\, \Delta_{\leq j+C}\tilde{g}\|_{L^1_e L^{\infty}_{e^\perp,t}} \aleq \|f\|_{F^{\frac{n-2s}{2}}}\, \|g\|_{F^{\frac{n-2s}{2}}}.
\]
Interchanging the roles of $f$ and $g$, 
\[
\sum_{2^\ell \aleq 2^j}\| \Delta_{\leq j+C}\tilde{f}\, \Delta_{\ell} \tilde{g}\|_{L^1_e L^{\infty}_{e^\perp,t}} \aleq \|f\|_{F^{\frac{n-2s}{2}}}\, \|g\|_{F^{\frac{n-2s}{2}}}.
\]
It remains to estimate 
\[
\sum_{2^{\ell_1} \aeq 2^{\ell_2} \ageq 2^j}\| \Delta_{j}\brac{\Delta_{\ell_2}\tilde{f}\, \Delta_{\ell_2} \tilde{g}}\|_{L^1_e L^{\infty}_{e^\perp,t}},
\]
and the usual maximal estimate does not suffice.

Instead use the decomposition from \eqref{eq:weirdfinedecomp}, and have (for any $j \in \Z$)
\[
\begin{split}
 &\brac{\Delta_{\ell_2}\tilde{f}\, \Delta_{\ell_2} \tilde{g}}\\
 =&\sum_{\tell_1,\tell_2 \in 2^k \Z} \brac{\Delta_{\ell_2} P_{j,\tell_2}\tilde{f}\, \Delta_{\ell_2} P_{j,\tell_2}\tilde{g}}
 \end{split}
\]
Observing that 
\[
\begin{split}
 &|\xi| \aeq 2^j,\ |\xi-\eta - \tell_1| \aleq 2^j,\ |\eta - \tell_2| \aleq 2^j\\
 \Rightarrow & |\tell_1 + \tell_2| =|-\eta - \tell_1 + (\eta - \tell_2)| \aleq  |\xi| + 2^j \aleq 2^j.
 \end{split}
\]
Moreover, if $\tell_1, \tell_2 \in 2^j \Z$ and $|\tell_1 + \tell_2| \aleq 2^j$ then $\tell_1$ and $\tell_2$ are essentially the same in the sense that 
\[
 \#\{\tell_2 \in 2^j \Z: |\tell_2+\tell_1| \aleq 2^j\} \aleq 1 \quad \forall \tell_1 \in 2^j \Z.
\]
Thus,
\[
\begin{split}
 &\| \Delta_{j}\brac{\Delta_{\ell_1}\tilde{f}\, \Delta_{\ell_2} g}\|_{L^1_e L^{\infty}_{e^\perp,t}} \\
 \aleq& \sum_{\tell_1 \in 2^j \Z^n}\sum_{\tell_2 \in 2^j \Z^n, |\tell_2+\tell_1| \aeq 2^k} \|\Delta_{\ell_1}P_{j,\tell_1}\tilde{f}\|_{L^2_e L^\infty_{e^\perp,t}} \| \Delta_{\ell_2} P_{j,\tell_2}g \|_{L^2_e L^\infty_{e^\perp,t}}\\
 \aleq& \brac{\sum_{\tell \in 2^j \Z^n}\|\Delta_{\ell_1}P_{j,\tell}\tilde{f}\|_{L^2_e L^\infty_{e^\perp,t}}^2}^{\frac{1}{2}} \brac{\sum_{\tell \in 2^j \Z^n}\| \Delta_{\ell_2} P_{j,\tell}\tilde{g} \|_{L^2_e L^\infty_{e^\perp,t}}^2 }^{\frac{1}{2}}
 \end{split}
\]
In view of the localized maximal estimate \Cref{la:locmaxest}, \eqref{eq:BIK:3.1:f} (which also applicable to $\bar{f}$ and $\bar{g}$) since $j \aleq \ell_1$ for $\sigma = \frac{n-2}{2}$,
\[
 \brac{\sum_{\tell \in 2^j \Z}\|\Delta_{\ell_1}P_{j,\tell}\tilde{f}\|_{L^2_e L^\infty_{e^\perp,t}}^2}^{\frac{1}{2}} \aleq 2^{\ell_1\frac{(n-1)}{2}}\, 2^{-(\ell_1-j)\sigma
 }\, (1+|\ell_1-j|)\|f\|_{Z_{\ell_1}}.
\]
and since $2^{\ell_1} \aeq 2^{\ell_2}$ we also have
\[
\brac{\sum_{\tell \in 2^j \Z^n}\| \Delta_{\ell_2} P_{j,\tell}\tilde{g} \|_{L^2_e L^\infty_{e^\perp,t}}^2 }^{\frac{1}{2}}\aleq 2^{\ell_1\frac{(n-1)}{2}}\, 2^{-(\ell_1-j)\sigma
 }\, (1+|\ell_1-j|)\|g\|_{Z_{\ell_2}}.
\]
Thus, whenever $2\sigma -2s+1  > 0$ (which for our $\sigma$ works if $n \geq 4$)
\[
\begin{split}
 &\sum_{j \in \Z} 2^{-j(2s-1)} \sum_{2^{\ell_1} \aeq 2^{\ell_2} \ageq 2^j} \| \Delta_{j}\brac{\Delta_{\ell_2}\tilde{f}\, \Delta_{\ell_2} \tilde{g}}\|_{L^1_e L^{\infty}_{e^\perp,t}}\\
\aleq&  \sum_{j \in \Z} 2^{-j(2s-1)} \sum_{2^{\ell_1} \aeq 2^{\ell_2} \ageq 2^j} 2^{\ell_1(n-1)}\, 2^{-(\ell_1-j)2\sigma}
 (1+|\ell_1-j|)^2\,  \|f\|_{Z_{\ell_1}} \, \|g\|_{Z_{\ell_2}}\\
 \aeq&   \sum_{2^{\ell_1} \aeq 2^{\ell_2}} 2^{\ell_1(n-1)}\,  \|f\|_{Z_{\ell_1}} \, \|g\|_{Z_{\ell_2}} \sum_{j: 2^j \aleq 2^{\ell_1}} 2^{-j(2s-1)} 2^{-(\ell_1-j)2\sigma}
 (1+|\ell_1-j|)^2\\
 \aeq&   \sum_{\ell_1, \ell_2: 2^{\ell_1} \aeq 2^{\ell_2}} 2^{\ell_1(n-2s)} \|f\|_{Z_{\ell_1}} \, \|g\|_{Z_{\ell_2}}  \sum_{j: 2^j \aleq 2^{\ell_1}}  2^{-(\ell_1-j)(2\sigma-2s+1)}
 (1+|\ell_1-j|)^2  \\
 \overset{2\sigma -2s+1  > 0}{\aeq}&\sum_{2^{\ell_1} \aeq 2^{\ell_2}} 2^{\ell_1(n-2s)} \|f\|_{Z_{\ell_1}} \, \|g\|_{Z_{\ell_2}}  \\
 \aleq&\|f\|_{F^{\frac{n-2s}{2}}}\, \|g\|_{F^{\frac{n-2s}{2}}}
 \end{split}
\]
We can conclude.
\end{proof}

\begin{lemma}\label{la:FL2estv2}
Let $n \geq 3$ and $k \in \Z$. 
\[
\begin{split}
\|\Delta_{k} (\tilde{f} \tilde{g})\|_{L^2}  \aleq& \|f\|_{F^{\frac{n-2s}{2}}}\, \sum_{\ell: 2^\ell \aeq2^k}  \|\Delta_{\ell} g\|_{Z_\ell}\\
&+\sum_{\ell: 2^\ell \aeq2^k}  \|\Delta_{\ell} f\|_{Z_\ell} \|g\|_{F^{\frac{n-2s}{2}}}\\
&+\sum_{j_1,j_2: 2^{j_1} \aeq 2^{j_2} \ageq 2^k} 2^{j_1 \frac{n-2s}{2}} \|\Delta_{j_1} f\|_{Z_{j_1}}  \|\Delta_{j_2} g\|_{Z_{j_2}}
\end{split}
\]
where $\tilde{f} \in \{f,\bar{f}\}$, $\tilde{g} \in \{g,\bar{g}\}$.
\end{lemma}
\begin{proof}
The usual paraproduct frequency decomposition gives estimate  
\[
\begin{split}
 &\|\Delta_{k} (\tilde{f} \tilde{g})\|_{L^2} \\
 \aleq& \sum_{\ell: 2^\ell \aeq 2^k}\|\Delta_{k} (\Delta_{\aleq 2^k}\tilde{f} \Delta_{\ell} g)\|_{L^2} + \sum_{\ell: 2^\ell \aeq 2^k}\|\Delta_{k} (\Delta_{\ell}\tilde{f} \Delta_{\aleq 2^k} \tilde{g})\|_{L^2} \\
 &+ \sum_{j_1,j_2: 2^{j_1} \aeq 2^{j_2} \ageq 2^k } \|\Delta_{j_1}\tilde{f} \Delta_{j_2} \tilde{g}\|_{L^2} \\
 \end{split}
\]
For the first term we estimate we use \Cref{la:geometricobserv}
\[
\begin{split}
 & \|\Delta_{k} (\Delta_{\aleq 2^k}\tilde{f} \Delta_{\ell} \tilde{g})\|_{L^2} \\
 \aleq& \max_{e \in \mathscr{E}} \|\Delta_{k} (\Delta_{\aleq 2^k}\tilde{f} \mathcal{F}^{-1} \brac{\vartheta_{\langle \xi,e\rangle \geq \bc{} |\xi|} \widehat{\Delta_{\ell}\tilde{g}}} )\|_{L^2} \\
 \end{split}
\]
For any fixed $e \in \mathscr{E}$ we estimate 
\[
 \begin{split}
&  \|\Delta_{k} (\Delta_{\aleq 2^k}\tilde{f} \mathcal{F}^{-1} \brac{\vartheta_{\langle \xi,e\rangle \geq \bc{} |\xi|} \widehat{\Delta_{\ell}\tilde{g}}} )\|_{L^2}\\
\aleq&\|\Delta_{\aleq 2^k}\tilde{f}\|_{L^2_e L^\infty_{e^\perp,t}} \| \mathcal{F}^{-1} \brac{\vartheta_{\langle \xi,e\rangle \geq \bc{} |\xi|} \widehat{\Delta_{\ell}\tilde{g}}} )\|_{L^{\infty}_e L^2_{e^\perp,t}}\\
\overset{\text{\eqref{eq:la:globmax}, \Cref{la:locsmooth}}}{\aleq}& \sum_{i: 2^i \aleq 2^k} 2^{i\frac{n-1}{2}} \|\Delta_{i} f\|_{Z_i}  2^{-\ell \frac{2s-1}{2}} \|\Delta_{\ell} g\|_{Z_\ell}.
 \end{split}
\]

Thus,
\[
 \begin{split}
&  \sum_{\ell: 2^\ell \aeq2^k} \|\Delta_{k} (\Delta_{\aleq 2^k}f \mathcal{F}^{-1} \brac{\vartheta_{\langle \xi,e\rangle \geq \bc{} |\xi|} \widehat{\Delta_{\ell}g}} )\|_{L^2}\\
\aleq& \sum_{\ell: 2^\ell \aeq2^k}\sum_{i: 2^i \aleq 2^\ell} 2^{-(\ell-i)\frac{2s-1}{2}} 
2^{i\frac{n-2s}{2}}  \|\Delta_{i} f\|_{Z_i}  \|\Delta_{\ell} g\|_{Z_\ell}\\
\aleq&\|f\|_{F^{\frac{n-2s}{2}}}\, \sum_{\ell: 2^\ell \aeq2^k}  \|\Delta_{\ell} g\|_{Z_\ell}.
 \end{split}
\]
Similarly we estimate 
\[
 \sum_{\ell: 2^\ell \aeq 2^k}\|\Delta_{k} (\Delta_{\ell}\tilde{f} \Delta_{\aleq 2^k} \tilde{g})\|_{L^2}
 \aleq \|g\|_{F^{\frac{n-2s}{2}}}\, \sum_{\ell: 2^\ell \aeq2^k}  \|\Delta_{\ell} f\|_{Z_\ell}.
\]
For the last term to estimate we observe again with  \Cref{la:geometricobserv}
\[
\begin{split}
 &\|\Delta_{j_1}\tilde{f} \Delta_{j_2} \tilde{g}\|_{L^2} \\
 \aleq&\max_{e \in \mathscr{E}} \|
  \Delta_{j_1}\tilde{f} \, \mathcal{F}^{-1} \brac{\vartheta_{\langle \xi,e\rangle \geq \bc{} |\xi|} \widehat{\Delta_{j_2}\tilde{g}}}\|_{L^2} \\
  \aleq&\max_{e \in \mathscr{E}} \|
  \Delta_{j_1}\tilde{f}\|_{L^{2}_e L^\infty_{e^\perp,t}} \|\mathcal{F}^{-1} \brac{\vartheta_{\langle \xi,e\rangle \geq \bc{} |\xi|} \widehat{\Delta_{j_2}\tilde{g}}}\|_{L^\infty_e L^2_{e^\perp,t}} \\
  \overset{\text{\eqref{eq:la:globmax}, \Cref{la:locsmooth}}}{\aleq}& 2^{j_1 \frac{n-1}{2}} \|\Delta_{j_1} f\|_{Z_{j_1}} 2^{-j_2\frac{2s-1}{2}} \|\Delta_{j_2} g\|_{Z_{j_2}}\\
 \end{split}
\]
Thus,
\[
\begin{split} 
&\sum_{j_1,j_2: 2^{j_1} \aeq 2^{j_2} \ageq 2^k} \|\Delta_{j_1}\tilde{f}\, \Delta_{j_2} \tilde{g}\|_{L^2}\\
\aleq& \sum_{j_1,j_2: 2^{j_1} \aeq 2^{j_2} \ageq 2^k} 2^{j_1 \frac{n-2s}{2}} \|\Delta_{j_1} f\|_{Z_{j_1}}  \|\Delta_{j_2} g\|_{Z_{j_2}}
\end{split}
\]
We can conclude.
\end{proof}

\begin{proof}[Proof of \Cref{th:trilinear}]
Fix $k \in \Z$. Setting $F = \Ds{-\beta} (\tilde{f}_1 \tilde{f}_2)$, we estimate
\begin{equation}\label{eq:tril:split}
\begin{split}
 &\|\brac{\i \partial_t+\Dels{s} + \i}^{-1} \Delta_k \brac{F\, \Ds{\beta} \tilde{f}_3}\|_{Z_k}\\
 \aleq&\|\brac{\i \partial_t+\Dels{s} + \i}^{-1} \Delta_k \brac{\Delta_{\aleq 2^k}F\, \Ds{\beta} \Delta_{\aeq k}\tilde{f}_3}\|_{Z_k}\\
 &+\sum_{j: 2^j\ageq 2^k}\|\brac{\i \partial_t+\Dels{s} + \i}^{-1} \Delta_k \brac{\Delta_{j}F\, \Ds{\beta} \Delta_{\aleq j}\tilde{f}_3}\|_{Z_k}\\
 \end{split}
 \end{equation}

We are going to prove 
\begin{equation}\label{eq:trilineargoal1}
 \|\brac{\i \partial_t+\Dels{s} + \i}^{-1} \Delta_k \brac{\Delta_{\aleq 2^k}F\, \Ds{\beta} \Delta_{\aeq k}\tilde{f}_3}\|_{Z_k} \aleq \|f_1\|_{F^{\frac{n-2s}{2}}}\, \|f_2\|_{F^{\frac{n-2s}{2}}}   \sum_{\ell: 2^\ell \aeq 2^k}  \|\Delta_{\ell}f_3\|_{Z_\ell}
\end{equation}
and 
\begin{equation}\label{eq:trilineargoal2}
\begin{split}
&\sum_{j: 2^j\ageq 2^k}\|\brac{\i \partial_t+\Dels{s} + \i}^{-1} \Delta_k \brac{\Delta_{j}F\, \Ds{\beta} \Delta_{\aleq j}\tilde{f}_3}\|_{Z_k}\\
\aleq& \|{f}_1\|_{F^{\frac{n-2s}{2}}} \sum_{\ell: 2^\ell\ageq 2^k}  2^{(\ell-k)\frac{2s-1}{2}} \|\Delta_{\ell} {f}_2\|_{Z_\ell}   \|{f}_3\|_{F^{\frac{n-2s}{2}}}\\
&+ \sum_{\ell: 2^\ell\ageq 2^k} 2^{(\ell-k)\frac{2s-1}{2}} \|\Delta_{\ell} {f}_1\|_{Z_\ell}\, \|{f}_2\|_{F^{\frac{n-2s}{2}}} \|{f}_3\|_{F^{\frac{n-2s}{2}}}\\
\end{split}
\end{equation}
Together, \eqref{eq:tril:split}, \eqref{eq:trilineargoal1}, and \eqref{eq:trilineargoal2} imply the claim. Indeed they give for any $\sigma > \frac{2s-1}{2}$

\[
\begin{split}
 &\sum_{k \in \Z}2^{k \sigma}\|\brac{\i \partial_t+\Dels{s} + \i}^{-1} \Delta_k \brac{F\, \Ds{\beta} \tilde{f}_3}\|_{Z_k}\\
 \aleq& \|{f}_1\|_{F^{\frac{n-2s}{2}}}\, \|{f}_2\|_{F^{\frac{n-2s}{2}}}   \sum_{k \in \Z}2^{k \sigma} \|\Delta_{k}{f}_3\|_{Z_k}\\
 &+\|{f}_1\|_{F^{\frac{n-2s}{2}}} \sum_{\ell} 2^{\ell\frac{2s-1}{2}} \|\Delta_{\ell} {f}_2\|_{Z_\ell} \sum_{k : 2^\ell\ageq 2^k} 2^{k \brac{\sigma-\frac{2s-1}{2}}}  \|{f}_3\|_{F^{\frac{n-2s}{2}}}\\
&+ \sum_{\ell \in \Z} 2^{\ell\frac{2s-1}{2}} \|\Delta_{\ell} {f}_1\|_{Z_\ell} \sum_{k : 2^\ell\ageq 2^k}   2^{k \brac{\sigma-\frac{2s-1}{2}}}  \|{f}_2\|_{F^{\frac{n-2s}{2}}} \|{f}_3\|_{F^{\frac{n-2s}{2}}}\\
\overset{\sigma > \frac{2s-1}{2}}{\aleq}& \|{f}_1\|_{F^{\frac{n-2s}{2}}}\, \|{f}_2\|_{F^{\frac{n-2s}{2}}}   \sum_{k \in \Z}2^{k \sigma} \|\Delta_{k}{f}_3\|_{Z_k}\\
 &+\|{f}_1\|_{F^{\frac{n-2s}{2}}} \sum_{\ell} 2^{\ell\sigma} \|\Delta_{\ell} {f}_2\|_{Z_\ell}  \|{f}_3\|_{F^{\frac{n-2s}{2}}}\\
&+ \sum_{\ell \in \Z} 2^{\ell \sigma} \|\Delta_{\ell} {f}_1\|_{Z_\ell}  \|{f}_2\|_{F^{\frac{n-2s}{2}}} \|{f}_3\|_{F^{\frac{n-2s}{2}}}\\
\aleq & \|{f}_1\|_{F^{\frac{n-2s}{2}}}  \|{f}_2\|_{F^{\frac{n-2s}{2}}}  \|{f}_3\|_{F^{\sigma}} \\
&+ \|{f}_1\|_{F^{\frac{n-2s}{2}}}  \|{f}_2\|_{F^{\sigma}}  \|{f}_3\|_{F^{\frac{n-2s}{2}}}  \\
&+ \|{f}_1\|_{F^{\sigma}} \|{f}_2\|_{F^{\frac{n-2s}{2}}}  \|{f}_3\|_{F^{\frac{n-2s}{2}}}   
  \end{split}
\]
We can conclude taking $\sigma \geq \frac{n-2s}{2}$ (and observe that $\sigma \geq \frac{n-2s}{2} > \frac{2s-1}{2}$ if and only if $n > 4s-1$, so in particular if $n \geq 4$).

It remains to prove \eqref{eq:trilineargoal1} and \eqref{eq:trilineargoal2}.

\underline{First we treat \eqref{eq:trilineargoal1}}: we use the geometric observation \Cref{la:geometricobserv}
\[
\begin{split}
 &\|\brac{\i \partial_t+\Dels{s} + \i}^{-1} \Delta_k \brac{\Delta_{\aleq 2^k}F\, \Ds{\beta} \Delta_{\aeq k}\tilde{f}_3}\|_{Z_k}\\
 \aleq&\sum_{\ell: 2^\ell \aeq 2^k}\max_{e \in \mathscr{E}} \|\brac{\i \partial_t+\Dels{s} + \i}^{-1} \Delta_k \brac{\Delta_{\aleq k}F\, \Ds{\beta} \mathcal{F}^{-1} \brac{\vartheta_{\langle \xi,e\rangle \geq \bc{} |\xi|}\widehat{\Delta_{\ell}\tilde{f}_3}}}\|_{Z_k}\\
 \aleq&\sum_{\ell: 2^\ell \aeq 2^k}\max_{e \in \mathscr{E}} \|\brac{\i \partial_t+\Dels{s} + \i}^{-1} \Delta_k \brac{\Delta_{\aleq k}F\, \Ds{\beta} \mathcal{F}^{-1} \brac{\vartheta_{\langle \xi,e\rangle \geq \bc{} |\xi|}\widehat{\Delta_{\ell}\tilde{f}_3}}}\|_{Y_k^e}\\
 =&\sum_{\ell: 2^\ell \aeq 2^k}\max_{e \in \mathscr{E}} 2^{-k\frac{2s-1}{2}}\|\Delta_k \brac{\Delta_{\aleq k}F\, \Ds{\beta} \mathcal{F}^{-1} \brac{\vartheta_{\langle \xi,e\rangle \geq \bc{} |\xi|}\widehat{\Delta_{\ell}\tilde{f}_3}}}\|_{L^1_e L^2_{e^\perp,t}}\\
 \aleq&\sum_{\ell: 2^\ell \aeq 2^k}\max_{e \in \mathscr{E}} 2^{-k\frac{2s-1}{2}}\|\Delta_{\aleq k}F \|_{L^1_e L^\infty_{e^\perp,t}} \|\Ds{\beta} \mathcal{F}^{-1} \brac{\vartheta_{\langle \xi,e\rangle \geq \bc{} |\xi|}\widehat{\Delta_{\ell}\tilde{f}_3}}\|_{L^\infty_e L^2_{e^\perp,t}}\\
 \overset{\text{\Cref{la:dsbetaest}}}{\aleq}&\sum_{\ell: 2^\ell \aeq 2^k}\max_{e \in \mathscr{E}} 2^{-k\frac{2s-1}{2}}\|\Delta_{\aleq k}F \|_{L^1_e L^\infty_{e^\perp,t}} 2^{\beta \ell} \|\mathcal{F}^{-1} \brac{\vartheta_{\langle \xi,e\rangle \geq \bc{} |\xi|}\widehat{\Delta_{\ell}\tilde{f}_3}}\|_{L^\infty_e L^2_{e^\perp,t}}\\
 \end{split}
\]
With local smoothening \Cref{la:locsmooth},then using that $F = \Ds{-\beta}(\tilde{f}_1\tilde{f}_2)$ and \Cref{la:BIK5.4}
\[
\begin{split}
 &\|\brac{\i \partial_t+\Dels{s} + \i}^{-1} \Delta_k \brac{\Delta_{\aleq k}F\, \Ds{\beta} \Delta_{\aeq k}\tilde{f}_3}\|_{Z_k}\\
\aleq&\sum_{\ell: 2^\ell \aeq 2^k}  2^{-k\frac{2s-1}{2}}  2^{k(2s-1-\beta)} \|{f}_1\|_{F^{\frac{n-2s}{2}}}\, \|{f}_2\|_{F^{\frac{n-2s}{2}}}  2^{\beta \ell} 2^{-\ell\frac{2s-1}{2}} \|{f}_3\|_{Z_\ell}\\
\aeq&\|{f}_1\|_{F^{\frac{n-2s}{2}}}\, \|{f}_2\|_{F^{\frac{n-2s}{2}}}   \sum_{\ell: 2^\ell \aeq 2^k}  \|{f}_3\|_{Z_\ell}\\
 \end{split}
\]
This proves \eqref{eq:trilineargoal1}.

We still need to \underline{prove \eqref{eq:trilineargoal2}}.

Fix any $j$ with and $2^j \ageq 2^k$ and any $e \in \mathscr{E}$ 
\[
\begin{split}
&\|\brac{\i \partial_t+\Dels{s} + \i}^{-1} \Delta_k \brac{\Delta_{j}F\, \Ds{\beta} \Delta_{\aleq j}\tilde{f}_3}\|_{Z_k}\\
\leq&\|\brac{\i \partial_t+\Dels{s} + \i}^{-1} \Delta_k \brac{\Delta_{j}F\, \Ds{\beta} \Delta_{\aleq j}\tilde{f}_3}\|_{Y_k^e}\\
 \aleq&2^{-k\frac{2s-1}{2}}\|\Delta_{j}F\, \Ds{\beta} \Delta_{\aleq j}\tilde{f}_3\|_{L^1_eL^2_{e^\perp, t}}\\
\aleq&2^{-k\frac{2s-1}{2}}\|\Delta_{j}F\|_{L^2} \|\Ds{\beta} \Delta_{\aleq j}\tilde{f}_3\|_{L^2_eL^\infty_{e^\perp, t}}\\ 
\end{split}
\]
With the maximal estimate, \eqref{eq:la:globmax},
\[
\begin{split}
\|\Ds{\beta} \Delta_{\aleq j}\tilde{f}_3\|_{L^2_eL^\infty_{e^\perp, t}}
\overset{\text{\Cref{la:dsbetaest}}}{\aleq}&\sum_{\ell: 2^\ell \aleq 2^j} 2^{j\beta} \| \Delta_{\ell}\tilde{f}_3\|_{L^2_eL^\infty_{e^\perp, t}}\\
\overset{\text{\eqref{eq:la:globmax}}}{\aleq}&\sum_{\ell: 2^\ell \aleq 2^j} 2^{\ell\beta} 2^{\ell\frac{n-1}{2}}\| \Delta_{\ell}{f}_3\|_{Z_\ell}\\
\overset{\beta \geq -\frac{2s-1}{2}}{\aleq}&2^{j(\beta + \frac{2s-1}{2})}\sum_{\ell: 2^\ell \aleq 2^j}  2^{\ell\frac{n-2s}{2}}\| \Delta_{\ell}{f}_3\|_{Z_\ell}\\
=&2^{j(\beta + \frac{2s-1}{2})} \|{f}_3\|_{F^{\frac{n-2s}{2}}}.
\end{split}
\]
Since $F = \Ds{-\beta} (\tilde{f}_1 \tilde{f}_2)$ we have 
\[
\begin{split}
 \|\Delta_{j}F\|_{L^2} \aleq& 2^{-j\beta} \sum_{{j}: 2^j \aeq 2^{{j}}}\|\Delta_{{j}} \brac{{f}_1 {f}_2}\|_{L^2}\\
  \overset{\text{\Cref{la:FL2estv2} }}{\aleq}&   2^{-j\beta} \|{f}_1\|_{F^{\frac{n-2s}{2}}}\, \sum_{\ell: 2^\ell \aeq2^j}  \|\Delta_{\ell} {f}_2\|_{Z_\ell}\\
&+2^{-j\beta} \sum_{\ell: 2^\ell \aeq2^j}  \|\Delta_{\ell} {f}_1\|_{Z_\ell} \|{f}_2\|_{F^{\frac{n-2s}{2}}}\\
&+2^{-j\beta} \sum_{j_1,j_2: 2^{j_1} \aeq 2^{j_2} \ageq 2^j} 2^{j_1 \frac{n-2s}{2}} \|\Delta_{j_1} {f}_1\|_{Z_{j_1}}  \|\Delta_{j_2} {f}_2\|_{Z_{j_2}}
\\
 \end{split}
\]
Thus,
\[
\begin{split}
&\|\brac{\i \partial_t+\Dels{s} + \i}^{-1} \Delta_k \brac{\Delta_{j}F\, \Ds{\beta} \Delta_{\aleq j}\tilde{f}_3}\|_{Z_k}\\
\aleq&
\|{f}_1\|_{F^{\frac{n-2s}{2}}}\, 2^{j\frac{2s-1}{2}} \sum_{\ell: 2^\ell \aeq2^j}  \|\Delta_{\ell} {f}_2\|_{Z_\ell} \|{f}_3\|_{F^{\frac{n-2s}{2}}}\\
&+ 2^{j( \frac{2s-1}{2})}\sum_{\ell: 2^\ell \aeq2^j}   \|\Delta_{\ell} {f}_1\|_{Z_\ell}\, \|{f}_2\|_{F^{\frac{n-2s}{2}}} \|{f}_3\|_{F^{\frac{n-2s}{2}}}\\
&+ 2^{j(\frac{2s-1}{2})} \sum_{j_1,j_2: 2^{j_1} \aeq 2^{j_2} \ageq 2^j} 2^{j_1 \frac{n-2s}{2}} \|\Delta_{j_1} {f}_1\|_{Z_{j_1}}  \|\Delta_{j_2} {f}_2\|_{Z_{j_2}} \|{f}_3\|_{F^{\frac{n-2s}{2}}}
\\ 
\end{split}
\]
And this implies 
\[
\begin{split}
&\sum_{j: 2^j\ageq 2^k}\|\brac{\i \partial_t+\Dels{s} + \i}^{-1} \Delta_k \brac{\Delta_{j}F\, \Ds{\beta} \Delta_{\aleq j}\tilde{f}_3}\|_{Z_k}\\
\aleq&
\sum_{j: 2^j\ageq 2^k}2^{(j-k)\frac{2s-1}{2}} \|{f}_1\|_{F^{\frac{n-2s}{2}}}\, \sum_{\ell: 2^\ell \aeq2^j}   \|\Delta_{\ell} {f}_2\|_{Z_\ell} \|{f}_3\|_{F^{\frac{n-2s}{2}}}\\
&+ \sum_{j: 2^j\ageq 2^k} 2^{(j-k)\frac{2s-1}{2}}\sum_{\ell: 2^\ell \aeq2^j}   \|\Delta_{\ell} {f}_1\|_{Z_\ell}\, \|{f}_2\|_{F^{\frac{n-2s}{2}}} \|{f}_3\|_{F^{\frac{n-2s}{2}}}\\
&+ \sum_{j: 2^j\ageq 2^k}\sum_{j_1,j_2: 2^{j_1} \aeq 2^{j_2} \ageq 2^j} 2^{(j-k)\frac{2s-1}{2}}2^{j_1 \frac{n-2s}{2}} \|\Delta_{j_1} {f}_1\|_{Z_{j_1}}  \|\Delta_{j_2} {f}_2\|_{Z_{j_2}}  \|{f}_3\|_{F^{\frac{n-2s}{2}}}
\\ 
=&\|{f}_1\|_{F^{\frac{n-2s}{2}}} \sum_{\ell: 2^\ell\ageq 2^k}  2^{(\ell-k)\frac{2s-1}{2}} \|\Delta_{\ell} {f}_2\|_{Z_\ell}   \|{f}_3\|_{F^{\frac{n-2s}{2}}}\\
&+ \sum_{\ell: 2^\ell\ageq 2^k} 2^{(\ell-k)\frac{2s-1}{2}} \|\Delta_{\ell} {f}_1\|_{Z_\ell}\, \|{f}_2\|_{F^{\frac{n-2s}{2}}} \|{f}_3\|_{F^{\frac{n-2s}{2}}}\\
&+ \sum_{j_1: 2^{j_1} \ageq 2^k} \sum_{j_2: 2^{j_1} \aeq 2^{j_2}} \sum_{j: 2^j\aleq 2^{j_1}} 2^{(j-k)\frac{2s-1}{2}}2^{j_1 \frac{n-2s}{2}} \|\Delta_{j_1} {f}_1\|_{Z_{j_1}}  \|\Delta_{j_2} {f}_2\|_{Z_{j_2}}  \|{f}_3\|_{F^{\frac{n-2s}{2}}}
\\ 
\overset{s >1/2}{\aeq}&\|{f}_1\|_{F^{\frac{n-2s}{2}}} \sum_{\ell: 2^\ell\ageq 2^k}  2^{(\ell-k)\frac{2s-1}{2}} \|\Delta_{\ell} {f}_2\|_{Z_\ell}   \|{f}_3\|_{F^{\frac{n-2s}{2}}}\\
&+ \sum_{\ell: 2^\ell\ageq 2^k} 2^{(\ell-k)\frac{2s-1}{2}} \|\Delta_{\ell} {f}_1\|_{Z_\ell}\, \|{f}_2\|_{F^{\frac{n-2s}{2}}} \|{f}_3\|_{F^{\frac{n-2s}{2}}}\\
&+ \sum_{j_1: 2^{j_1} \ageq 2^k} \sum_{j_2: 2^{j_1} \aeq 2^{j_2}}  2^{(j_1-k)\frac{2s-1}{2}}2^{j_1 \frac{n-2s}{2}} \|\Delta_{j_1} {f}_1\|_{Z_{j_1}}  \|\Delta_{j_2} {f}_2\|_{Z_{j_2}}  \|{f}_3\|_{F^{\frac{n-2s}{2}}}\\
\aleq&\|{f}_1\|_{F^{\frac{n-2s}{2}}} \sum_{\ell: 2^\ell\ageq 2^k}  2^{(\ell-k)\frac{2s-1}{2}} \|\Delta_{\ell} {f}_2\|_{Z_\ell}   \|{f}_3\|_{F^{\frac{n-2s}{2}}}\\
&+ \sum_{\ell: 2^\ell\ageq 2^k} 2^{(\ell-k)\frac{2s-1}{2}} \|\Delta_{\ell} {f}_1\|_{Z_\ell}\, \|{f}_2\|_{F^{\frac{n-2s}{2}}} \|{f}_3\|_{F^{\frac{n-2s}{2}}}\\
&+ \sum_{j_1: 2^{j_1} \ageq 2^k}  2^{(j_1-k)\frac{2s-1}{2}} \|\Delta_{j_1} {f}_1\|_{Z_{j_1}}  \|{f}_2\|_{F^{\frac{n-2s}{2}}}  \|{f}_3\|_{F^{\frac{n-2s}{2}}}\\
\end{split}
\]
This implies \eqref{eq:trilineargoal2}.

\end{proof}

\section{Short-time existence}\label{s:existence}
In this section we combine the estimates of the previous sections to prove the following

\begin{theorem}\label{th:mainshort}
Let $s \in (\frac{1}{2},1]$ and let $n \geq 4$. For any $N \in \N$ and $(\beta_i)_{i=1}^N \in [-\frac{2s-1}{2},2s-1]$ and sufficiently small $\eps > 0$ such that $u_0 \in \dot{H}^{\frac{n-2s}{2}}(\R^n)$ and
\[
\lVert u_0\rVert_{\dot{H}^{\frac{n-2s}{2}}(\R^n)} \leq \eps,
\]
there exists $u \in L^\infty\brac{(-1,1),\dot{H}^{\frac{n-2s}{2}}} \cap F^{\frac{n-2s}{2}}$ such that
\[
\begin{cases} \brac{ \i \partial_t + \Dels{s} } u = \sum_{i=1}^N \brac{\Ds{-\beta_i} u_{i,1} u_{i,2}}\ \Ds{\beta_i} u_{i,3} \quad &\text{in }\R^d \times (-1,1)\\
u(0) = u_0
\end{cases}.
\]
where $u_{i,k} \in \{u,\bar{u}\}$ for $i \in \{1,\ldots,N\}$, $k \in \{1,2,3\}$.

Here we mean solution in the sense that 
\[
u(x,t) =e^{it\Dels{s}} u_0 - \i \int_0^t e^{\i (t-\tilde{t}) \Dels{s} } F(u)(x,\tilde{t})\, d\tilde{t} \quad t \in (-1,1), x \in \R^n
\]
where $F(u)$ is the nonlinearity on the right-hand side of the PDE.

Moreover 
\begin{itemize}\item (A priori estimates) $u$ satisfies the estimate
\begin{equation}\label{eq:mainpdeshortest}
\|u\|_{L^\infty_t \dot{H}^{\frac{n-2s}{2}}}  + \|u\|_{F^{\frac{n-2s}{2}}} \aleq C \|u_0\|_{\dot{H}^{\frac{n-2s}{2}}(\R^n)}
\end{equation}

\item (conditional uniqueness) The solution is unique in the sense that any two solutions in the above sense satisfying the estimate \eqref{eq:mainpdeshortest} must coincide.

\item (continuous dependence on data) If we consider $u$ the solution to the data $u_0$ and $v$ the solution to the data $v_0$, both satisfying the smallness condition and \eqref{eq:mainpdeshortest}, then 
\[
 \|u-v\|_{L^\infty_t \dot{H}^{\frac{n-2s}{2}}}  + \|u-v\|_{F^{\frac{n-2s}{2}}} \aleq C \|u_0-v_0\|_{\dot{H}^{\frac{n-2s}{2}}(\R^n)}
\]

\item If moreover $u_0 \in \dot{H}^{\sigma}$ for $\sigma > \frac{n-2s}{2}$ we also can assume (adjusting $\eps$) 
\begin{equation}\label{eq:mainpdeshortestsigma}
\|u\|_{L^\infty_t \dot{H}^{\sigma}}  + \|u\|_{F^{\sigma}} \aleq C \|u_0\|_{\dot{H}^{\sigma}(\R^n)}
\end{equation}

\item If moreover $u_0 \in \dot{H}^{\sigma}$ for $\sigma > \frac{n}{2}$ we have in addition
\[
 u \in C^0\brac{(-1,1),\dot{H}^{\frac{n-2s}{2}}}
\]

\end{itemize}

\end{theorem}
\begin{proof}
Below we assume for simplicity of representation that we are interested in 
\begin{equation}\label{eq:mainpdeshort}
\begin{cases} \brac{ \i \partial_t + \Dels{s} } u = \brac{\Ds{-\beta} |u|^2}\ \Ds{\beta} u \quad &\text{in }\R^d \times (-1,1)\\
u(0) = u_0
\end{cases}.
\end{equation}
The more general result follows almost verbatim.

Existence follows via a Banach fixed-point argument: 
Fix $\psi \in C_c^\infty(-2,2)$ with $\psi \equiv 1$ in $(-1,1)$ and define $F(u) = \brac{\Ds{-\beta} |u|^2}\ \Ds{\beta} u$.

Fix $u_0 \in \dot{H}^{\frac{n-2s}{2}}(\R^n)$, and define for some $\gamma > 0$ yet to be chosen 
\[
 X_\gamma := \left \{v \in F^{\frac{n-2s}{2}}: \quad \|v\|_{F^{\frac{n-2s}{2}}} \leq \gamma \right \}.
\]
which by definition is a closed and complete metric space when equipped with the $F^{\frac{n-2s}{2}}$-norm.

For $v \in X_\gamma$ we set 
\begin{equation}\label{eq:fixedpointeq}
Tv(x,t) :=e^{it\Dels{s}} u_0 -\psi(t) \i \int_0^t e^{\i (t-\tilde{t}) \Dels{s} } F(v)(x,\tilde{t})\, d\tilde{t}
\end{equation}
From \Cref{la:homogeneous} and \Cref{la:inhomogeneous}, combined with \Cref{th:trilinear} we have for some constant $C > 0$ (depending in particular on the choice of $\psi$)
\[
 \|Tv\|_{F^{\frac{n-2s}{2}}} \leq C \|u_0\|_{\dot{H}^{\frac{n-2s}{2}}} + C \|v\|_{F^{\frac{n-2s}{2}}}^3
\]
and, using also the trilinearity of $F$,
\[
 \|Tv_1-Tv_2\|_{F^{\frac{n-2s}{2}}} \leq C \brac{\|v_1\|_{F^{\frac{n-2s}{2}}}^2+\|v_2\|_{F^{\frac{n-2s}{2}}}^2}\, \|v_1-v_2\|_{F^{\frac{n-2s}{2}}}
\]
In particular if we choose $\eps$ and $\gamma$ suitably small so that 
\[
 2C \gamma^2 \leq \frac{1}{2}, \quad \text{and} \quad C \eps \leq \frac{1}{2}\gamma
\]
then for all $v \in X_\gamma$
\begin{equation}\label{eq:fp:Tvest}
\begin{split}
 \|Tv\|_{F^{\frac{n-2s}{2}}} \leq& C \|u_0\|_{\dot{H}^{\frac{n-2s}{2}}} + C \gamma^2 \|v\|_{F^{\frac{n-2s}{2}}}\\
 \leq& C \|u_0\|_{\dot{H}^{\frac{n-2s}{2}}} + \frac{1}{2} \|v\|_{F^{\frac{n-2s}{2}}}\\
 \leq& C \eps + \frac{1}{2} \gamma\\
 \leq&\gamma.
 \end{split}
\end{equation}
and 
\[
 \|Tv_1-Tv_2\|_{F^{\frac{n-2s}{2}}} \leq 2C \gamma^2 \|v_1-v_2\|_{F^{\frac{n-2s}{2}}} \leq \frac{1}{2} \|v_1-v_2\|_{F^{\frac{n-2s}{2}}},
\]
i.e. $T: X_\gamma \to X_\gamma$ is a contracting self-map and thus has a fixed point $u \in X_\gamma$,
\[
 Tu = u.
\]
Since $\psi \equiv 1$ in $(-1,1)$ we see that $u$ satisfies \eqref{eq:mainpdeshort} by Duhamel's principle.

From the second line of \eqref{eq:fp:Tvest} we find the estimate
\[
 \|u\|_{F^{\frac{n-2s}{2}}} \leq C \|u_0\|_{\dot{H}^{\frac{n-2s}{2}}} + \frac{1}{2} \|u\|_{F^{\frac{n-2s}{2}}}
\]
which readily implies
\[
 \|u\|_{F^{\frac{n-2s}{2}}} \leq 2C \|u_0\|_{\dot{H}^{\frac{n-2s}{2}}}.
\]
This combined with \Cref{la:IKDIE:3.33} implies \eqref{eq:mainpdeshortest}.

The conditional uniqueness is immediate from the fixed point argument.

As for continuous dependence on data: Fix $u_0,v_0 : \R^n \to \C$ such that $\|u_0\|_{\dot{H}^{\frac{n-2s}{2}}}$, $\|v_0\|_{\dot{H}^{\frac{n-2s}{2}}} < \eps$. Denote their respective solutions to the PDE $u$ and $v$ (i.e. $u$ and $v$ are fixed points of the operator in \eqref{eq:fixedpointeq}. That is, we have 
\[
u(x,t)-v(x,t) =e^{it\Dels{s}} (u_0 -v_0)(x)-\psi(t) \i \int_0^t e^{\i (t-\tilde{t}) \Dels{s} } \brac{F(u)-F(v)}(x,\tilde{t})\, d\tilde{t}
\]

Arguing as before, using the trilinearity of $F$ we find
\[
 \begin{split}
\lVert u - v \rVert_{F^{\frac{n-2s}{2}}}\le C \lVert u_0 - v_0 \rVert_{\dot{H}^{\frac{n-2s}{2}}(\R^n)}+ 2C\gamma^2\lVert u - v \rVert_{F^{\frac{n-2s}{2}}}. 
 \end{split}
\]
Since $2C\gamma^2 \leq \frac{1}{2}$ we readily conclude
\[
\lVert u - v \rVert_{F^{\frac{n-2s}{2}}} \aleq \lVert u_0 - v_0 \rVert_{\dot{H}^{\frac{n-2s}{2}}(\R^n)},
\]
which gives us continuous dependence on data as claimed.

Assume now that $u_0 \in \dot{H}^{\sigma}$ for $\sigma > \frac{n-2s}{2}$. In that case, we just apply \Cref{th:trilinear} to the fix point equation and find 
\[
\begin{split}
 \|u\|_{F^\sigma} \leq& C_\sigma \|u_0\|_{\dot{H}^\sigma} + C_\sigma \|u\|_{F^{\frac{n-2s}{2}}}^2 \|u\|_{F^{\sigma}}^2\\
 \aleq& C_\sigma \|u_0\|_{\dot{H}^\sigma} + C_\sigma \gamma^2 \|u\|_{F^{\sigma}}^2\\
 \end{split}
\]
Choosing $\gamma$ (and thus $\eps$) even smaller, and in view of \Cref{la:IKDIE:3.33}\eqref{eq:LinftyL2F}, we find 
\begin{equation}\label{eq:Fsigmaestex}
 \|u\|_{L^\infty_t \dot{H}^\sigma_x}  \aleq \|u\|_{F^\sigma} \aleq \|u_0\|_{\dot{H}^\sigma}.
\end{equation}
This establishes in \eqref{eq:mainpdeshortestsigma}.

Let us now show continuity, i.e. $u \in C^0\brac{(-1,1), \dot{H}^{\frac{n-2s}{2}}(\R^n)}$ under the assumption $\sigma > \frac{n}{2}$. In this case we claim that \eqref{eq:Fsigmaestex} implies for some small $\theta \in (0,s)$ 

\begin{equation}\label{eq:FuHdestimate}
 \|F(u)\|_{L^\infty_t H_x^{\frac{n-2s}{2}}(\R^n)} + \|F(u)\|_{L^\infty_t H_x^{\frac{n-2s}{2}+\theta}(\R^n)} \aleq \brac{\|u_0\|_{\dot{H}^{\frac{n-2s}{2}}} + \|u_0\|_{\dot{H}^\sigma}}^3.
\end{equation}
We prove \eqref{eq:FuHdestimate}: By the fractional Leibniz rule we have 
\[
\begin{split}
 &\|\Ds{\frac{n-2s}{2}+\theta} \Ds{\beta} |u|^2 \Ds{-\beta} u\|_{L^2(\R^n)}\\
 \aleq& \max_{\alpha_1+\alpha_2+\alpha_3 = \frac{n-2s}{2}+\theta, -\beta \leq \alpha_i \leq \frac{n-2s}{2} + \beta+\theta}\| \Ds{\alpha_1} u \|_{L^{p_1}(\R^n)} \|\Ds{\alpha_2} u\|_{L^{p_2}(\R^n)} \|\Ds{\alpha_3} u\|_{L^{p_3}(\R^n)}\\
 \aleq& \max_{\alpha_1+\alpha_2+\alpha_3 = \frac{n-2s}{2}+\theta, -\frac{2s-1}{2} \leq \alpha_i \leq \frac{n-1}{2}+\theta}\| \Ds{\alpha_1} u \|_{L^{p_1}(\R^n)} \|\Ds{\alpha_2} u\|_{L^{p_2}(\R^n)} \|\Ds{\alpha_3} u\|_{L^{p_3}(\R^n)}\\
 \end{split}
\]
where we can choose $p_i \in (1,\infty)$ depending on the $\alpha_i$ such that \[\sum_{i=1}^3 \frac{1}{p_i} = \frac{1}{2}.\]

We want to find for each $i$ some $\sigma_i \in [\frac{n-2s}{2},\sigma]$ and $\sigma_i \in (\alpha_i,\alpha_i + \frac{n}{2})$ (to guarantee $p_i \in (2,\infty)$) such that
\begin{equation}\label{eq:ex:pisigmaialphai}
 \frac{1}{p_i} +\frac{\sigma_i -\alpha_i}{n} = \frac{1}{2},
\end{equation}
because then 
\[
 \| \Ds{\alpha_1} u \|_{L^{p_1}(\R^n)} \aleq \| \Ds{\sigma_i} u \|_{L^{2}(\R^n)} = \|u\|_{\dot{H}^{\sigma_i}}.
\]
and thus 
\[
\begin{split}
 &\| \Ds{\alpha_1} u \|_{L^{p_1}(\R^n)} \|\Ds{\alpha_2} u\|_{L^{p_2}(\R^n)} \|\Ds{\alpha_3} u\|_{L^{p_3}(\R^n)}\\
 \aleq &\|u\|_{\dot{H}^{\sigma_1}}\, \|u\|_{\dot{H}^{\sigma_2}} \|u\|_{\dot{H}^{\sigma_3}}\\
 \aleq &\brac{\|u\|_{\dot{H}^{\frac{n-2s}{2}}} +\|u\|_{\dot{H}^{\sigma}}}^3.
 \end{split}
\]
This, in turn, gives  in view of \eqref{eq:Fsigmaestex}, \eqref{eq:mainpdeshortest} 
\[
\begin{split}
\|\Ds{\beta} |u|^2 \Ds{-\beta} u\|_{L^\infty_t \dot{H}^{\frac{n-2s}{2}}}
 \aleq& \brac{\|u\|_{L^\infty_t \dot{H}^{\frac{n-2s}{2}}} +\|u\|_{L^\infty_t \dot{H}^{\sigma}}}^3\\
 \aleq& \brac{\|u\|_{F^{\frac{n-2s}{2}}} +\|u\|_{F^{\sigma}}}^3\\
 \aleq& \brac{\|u_0\|_{\dot{H}^{\frac{n-2s}{2}}} +\|u_0\|_{\dot{H}^{\sigma}}}^3\\
 \end{split}
\]
i.e. we have \eqref{eq:FuHdestimate} -- if we can find $\sigma_i$ as described above.

In view of \eqref{eq:ex:pisigmaialphai} the condition $\sum_{i=1}^3 \frac{1}{p_i} = \frac{1}{2}$ becomes 
\[
\begin{split}
 &\sum_{i=1}^3 \brac{\frac{1}{2}-\frac{\sigma_i -\alpha_i}{n} }= \frac{1}{2}\\
 \Leftrightarrow&\sum_{i=1}^3 \frac{\sigma_i -\alpha_i}{n} = 1\\
 \Leftrightarrow& \sum_{i=1}^3\sigma_i -\frac{n-2s}{2}-\theta = n.
 \end{split}
\]
The easiest way to solve the above is choosing $\sigma_1=\sigma_2 = \frac{n}{2}$ and $\sigma_3 = \frac{n-2s}{2}+\theta$, but this does not necessarily satisfy $\sigma_i \in (\alpha_i,\alpha_i + \frac{n}{2})$, we also need to ensure $\sigma_i \in [\frac{n-2s}{2},\sigma]$

Observe that since $-\frac{2s-1}{2}\leq \alpha_i \leq \frac{n-1}{2}$ we have $\frac{n-2s}{2} < \alpha_i + \frac{n}{2}$. If we take 
\[
 \underline{\sigma}_i := \max\{\alpha_i,\frac{n-2s}{2}\} < \overline{\sigma}_i := \min\{\sigma,\alpha_i+\frac{n}{2}\}
\]
Then $\underline{\sigma}_i \in [\alpha_i,\alpha_i+\frac{n}{2}) \cap [\frac{n-2s}{2},\sigma]$ and $\overline{\sigma}_i \in (\alpha_i,\alpha_i+\frac{n}{2}] \cap [\frac{n-2s}{2},\sigma]$. Moreover we have
\[
 \sum_{i=1}^3 \underline{\sigma}_i - \frac{n-2s}{2} < n < \sum_{i=1}^3 \overline{\sigma}_i - \frac{n-2s}{2}
\]
By the intermediate value theorem there must be some choice of $\sigma_i \in (\underline{\sigma}_i,\overline{\sigma}_i)$ such that 
\[
 \sum_{i=1}^3 \sigma_i - \frac{n-2s}{2} =n+\theta,
\]
assuming that $\theta$ is suitably small. These $\sigma_i$ satisfy the required conditions and we have established \eqref{eq:FuHdestimate}.

With this in mind we can establish $u \in C^0((-1,1),\dot{H}^{\frac{n-2s}{2}}(\R^n))$. The fixed point equation reads as
\[
u(x,t) =e^{it\Dels{s}} u_0(x) -\psi(t) \i \int_0^t e^{\i (t-\tilde{t}) \Dels{s} } F(u)(x,\tilde{t})\, d\tilde{t}
\]
Take the Fourier transform in $x$ then this becomes
\[
 u(\hat{\xi},t) =e^{it|\xi|^{2s}} \hat{u}_0(\xi) -\psi(t) \i \int_0^t e^{\i (t-\tilde{t}) |\xi|^{2s} } F(u)(\hat{\xi},\tilde{t})\, d\tilde{t}.
\]
Fix $t,t_2 \in (-1,1)$, in particular $\psi(t), \psi(t_2) = 1$. Then we have
\[
\begin{split}
 \abs{u(\hat{\xi},t)-u(\hat{\xi},t_2)} \leq& \abs{e^{it|\xi|^{2s}}-e^{it_2|\xi|^{2s}}} \abs{ \hat{u}_0(\xi)}\\
 &+ \abs{\int_t^{t_2} e^{\i (t_2-\tilde{t}) |\xi|^{2s} } F(u)(\hat{\xi},\tilde{t})\, d\tilde{t}}\\
 &+ \int_0^t \abs{e^{\i (t-\tilde{t}) |\xi|^{2s} } -e^{\i (t_2-\tilde{t}) |\xi|^{2s} }} \abs{F(u)(\hat{\xi},\tilde{t})}\, d\tilde{t}\\
 \aleq&\abs{t-t_2}^{\frac{\theta}{2s}}  |\xi|^{\theta} \abs{ \hat{u}_0(\xi)}+ \int_t^{t_2} \abs{F(u)(\hat{\xi},\tilde{t})}\, d\tilde{t}+ \abs{t-t_2}^{\frac{\theta}{2s}} \int_{-1}^1 |\xi|^{\theta} \abs{F(u)(\hat{\xi},\tilde{t})}\, d\tilde{t}\\
 \end{split}
\]
Consequently,
\[
\begin{split}
 &\|u(\cdot,t)-u(\cdot,t_2)\|_{\dot{H}^{\frac{n-2s}{2}}}^2\\
 \aeq&\int_{\R^n} |\xi|^{2s-1} \abs{u(\hat{\xi},t)-u(\hat{\xi},t_2)}^2 d\xi\\
 \aleq&\abs{t-t_2}^{\frac{\theta}{s}} \int_{\R^n} \abs{ |\xi|^{\frac{2s-1}{2}+\theta} \hat{u}_0(\xi)}^2 d\xi\\
 &+ |t-t_2|^2 \sup_{\tilde{t} \in (-1,1)} \int |\xi|^{2s-1} \abs{F(u)(\hat{\xi},\tilde{t})}^2 d\xi\\
 &+ \abs{t-t_2}^{\frac{\theta}{s}} \sup_{\tilde{t} \in (-1,1)}\int_{\R^n} \abs{|\xi|^{\frac{2s-1}{2}+\theta} F(u)(\hat{\xi},\tilde{t})}^2 d\xi\\
 \aleq&|t-t_2|^{\frac{\theta}{2}} \|u_0\|_{\dot{H}^{\frac{n-2s}{2}+\theta}}
+|t-t_2|^2 \|F(u)\|_{L^\infty_t \dot{H}_x^{\frac{n-2s}{2}+\theta}}
+|t-t_2|^{\frac{\theta}{s}} \|F(u)\|_{L^\infty_t \dot{H}_x^{\frac{n-2s}{2}+\theta}}\\
\aleq&|t-t_2|^{\frac{\theta}{2}} \brac{\|u_0\|_{\dot{H}^{\frac{n-2s}{2}}}+\|u_0\|_{\dot{H}^{\sigma}}}
+|t-t_2|^2 \|F(u)\|_{L^\infty_t \dot{H}_x^{\frac{n-2s}{2}+\theta}}
+|t-t_2|^{\frac{\theta}{s}} \|F(u)\|_{L^\infty_t \dot{H}_x^{\frac{n-2s}{2}+\theta}}\\
 \end{split}
\]
In view of \eqref{eq:FuHdestimate} we have proven (even H\"older-)continuity in time.
We can conclude.
 \end{proof}

\appendix

\section{Simple estimates}\label{s:simpleest}
In this section we collect elementary estimates that we use in various places of the article.

\begin{lemma}\label{la:setestimate}
Fix $e \in \S^{n-1}$, $s \in (0,1]$. For $\xi' \in e^\perp$ and $\tau \in \R$, $k \in \Z$, $j \in \Z$, consider the set 
\[
 \Sigma_{k,j}(\xi',\tau) := \left \{\xi_1 \in \R: \quad  \text{for $\xi := \xi_1 e + \xi'$ we have} \begin{array}{l}
|\xi_1| \aeq 2^k \text{ and}\\
\abs{\tau + |\xi|^{2s}}\aleq {2^j}\text{ and}\\
|\xi| \aeq 2^k
\end{array}
\right \} \subset \R
\]

Then for any $j, k \in \Z$ we have 
\begin{equation}\label{eq:claimasdkljhfs}
 \sup_{\xi' \in e^\perp, \tau \in \R} \mathcal{L}^1 \brac{\Sigma_{k,j}(\xi',\tau)} \aleq \min \{{2^k},{2^{-k(2s-1)}} {2^j}\}.
\end{equation}
\end{lemma}
\begin{proof}

Fix $\xi' \in e^\perp$ and $\tau \in \R$. Clearly
\[
 \Sigma_{k,j}(\xi',\tau)  \subset \{\xi_1 \in \R: |\xi_1| \aeq 2^k\}
\]
so the estimate 
\[
 \mathcal{L}^1 \brac{\Sigma_{k,j}(\xi',\tau)} \aleq  2^k 
\]
is obvious.

In particular, whenever $j \geq 2sk-c$ for some given $c$, we have 
\[
 \mathcal{L}^1 \brac{\Sigma_{k,j}(\xi',\tau)}\aleq \min \{{2^k},{2^{-k(2s-1)}} {2^j}\}.
\]

Thus, we only need to prove
\begin{equation}\label{eq:measureestgoal}
 \mathcal{L}^1 \brac{\Sigma_{k,j}(\xi',\tau)}\aleq 2^{-k(2s-1)} 2^j \quad \forall j \leq 2sk-c.
\end{equation}

In this case, taking $c$ suitably large, we can apply \Cref{la:Nproperties} and thus
\[
 \Sigma_{k,j}(\xi',\tau) \subset \{\xi_1: \quad |\xi_{1} - N(\xi',\tau)| \aleq 2^{-k(2s-1)} 2^{j}\}
\]
which implies \eqref{eq:measureestgoal}. We can conclude.

\end{proof}

\begin{lemma}\label{simplestuff}
Let $k_{1} \in \mathbb{Z}$. Take $e \in \S^{n-1}$.

As in \eqref{eq:xieande1decomp}, for $ \tn \in 2^{k_{1}} \mathbb{Z}^{n}$ we write 
\[
 \tn= \tn_{e,1} e+ \tn_{e}'
\]
where $\tn_{e}' \in e^\perp$.

Define $\pi: e^{\perp} \times \R  \to \R$ by
$$\pi_{k_{1},\tn}(\xi_{e}',N) := \chi((N -\tn_{e,1})/2^{k_{1}+C}) \eta( \frac{\xi_{e}'-\tn_{e}'}{2^{k_{1}+C}}), $$
where $\chi$ is from \eqref{IntegFunc}.

Then we have (with a constant independent of $N$, $\xi_{e}' \in e^\perp$)
\begin{equation}\label{pa}
\begin{split}
& \abs{\sum_{\tn \in 2^{k_{1}} \Z^{n}} \pi_{k_{1},\tn}(N,\xi_{e}') }\aleq 1
\end{split}
\end{equation}
\end{lemma}
\begin{proof}
There are only finitely many $\tn \in 2^{k_1} \Z^n$ such that 
\[
 \pi_{k_{1},\tn}(\xi_{e}',N) \neq 0.
\]
Once we have established this fact, we can conclude easily, since $|\pi_{k_1,\tn}| \aleq 1$.

Fix two $\tn,\tm \in 2^{k_1} \Z^n$ such that 
\[
 |N -\tn_{e,1}| \aleq 2^{k_1} \quad \text{ and } |\xi_{e}'-\tn_{e}'| \aleq 2^{k_1}.
\]
and 
\[
 |N -\tm_{e,1}| \aleq 2^{k_1} \quad \text{ and } |\xi_{e}'-\tm_{e}'| \aleq 2^{k_1}.
\]
Then 
\[
 |\tn-\tm| \aleq |\tn_{e,1}-\tm_{e,1}| +  |\tn_{e}'-\tm_{e}'| \aleq 2^{k_1}
\]
and clearly 
\[
 \#\{\tm \in 2^{k_1} \Z^n: \quad |\tn-\tm| \aleq 2^{k_1}\} \aleq 1.
\]
\end{proof}

\section{Stationary phase estimates}
The main point of this section is to prove
\begin{theorem}\label{th:L2eLinftyeperpvsL2est}
Let $n \geq 3$, $s \in (1/2,1)$.
Then we have for any $e \in \S^{n-1}$ and for any $\ell \in \Z$ and for any $h \in C_c^\infty(B(0,2^\ell))$

\begin{equation}\label{eq:L2eLinftyeperpvsL2estv1}
 \norm{\int_{\R^n} e^{\i \brac{\langle \xi,x\rangle + |\xi|^{2s} t}} h(\xi) d\xi }_{L^2_e L^\infty_{e^\perp,t}} \aleq 2^{\ell\frac{n-1}{2}} \|h\|_{L^2(\R^n)}
\end{equation}
and more generally for any $\tn \in \R^n$ such that $|\tn| \aleq 2^k$ for some $k \in \Z$ with $2^k \ageq 2^\ell$,
\begin{equation}\label{eq:L2eLinftyeperpvsL2estv2}
 \norm{\int_{\R^n} e^{\i \brac{\langle \xi,x\rangle + |\xi-\tn|^{2s} t}} h(\xi) d\xi }_{L^2_e L^\infty_{e^\perp,t}} \aleq 2^{k\frac{n-1}{2}} 2^{-(k-\ell)\frac{n-2}{2}} \|h\|_{L^2(\R^n)}.
\end{equation}

\end{theorem}

\Cref{th:L2eLinftyeperpvsL2est} can be proven with stationary phase estimate, cf. \cite[\textsection 2.3 Proposition 6]{Stein} and \cite[\textsection 2.1, Proposition 4]{Stein}. For the convenience of the reader we will prove this statements using one-dimensional arguments. 

In \Cref{th:L2eLinftyeperpvsL2est} we will focus on the case $s \in (1/2,1)$. The case $s=1$ is also true, but the argument is different (e.g. based on the explicit representation of the Fourier transform of $e^{\i t|\xi|^2}$). The argument in \Cref{pr:statphaseestschrod} requires $s < 1$.

We begin by what is essentially an application of the $TT^\ast$-method.
\begin{lemma}\label{pr:TTastmethod}
Fix some real function $\phi(\xi)$ and assume we have for some $\lambda, \Lambda > 0$ the estimate
\begin{equation}\label{eq:Ttestneeded}
 \left \|\sup_{t,x'} \abs{\int_{\R^n} e^{\i (t\phi(\xi)-\langle x,\xi\rangle)} \eta_{|\xi| \aleq \lambda} d\xi}\right \|_{L^1_{x_1}} \leq \Lambda^2.
\end{equation}
Set 
\[
 T_tf(x) := \int_{\R^n} e^{\i (t\phi(\xi)-\langle x,\xi\rangle)} \eta_{|\xi| \aleq \lambda} \hat{f}(\xi)\, d\xi.
\]
Then we have for any $f \in C_c^\infty(B(0,\lambda))$,
\[
 \|T_t f\|_{L^2_{x_1} L^\infty_{x',t}} \aleq \Lambda \|f\|_{L^2(\R^n)}.
\]
\end{lemma}
\begin{proof}
By duality, for some $g : \R^{n+1} \to \R$, $\|g\|_{L^2_{x_1} L^1_{x',t}} \aleq 1$,
\[
\begin{split}
  &\|T_t f\|_{L^2_{x_1} L^\infty_{x',t}}\\
  \aleq&\int_{\R^{n+1}} T_t f(x_1,x',t) g(x_1,x',t)\, dx' dt dx_1\\
  =&\int_{\R^n} \hat{f}(\xi)\, \brac{\int_{\R^{n+1}} e^{\i (t\phi(\xi)-\langle x,\xi\rangle)} \eta_{|\xi| \aleq \lambda} g(x_1,x',t) dx'dt dx_1} d\xi \\
 \end{split}
\]
So the claim follows if we can show
\begin{equation}\label{eq:Ttastgoal}
 \int_{\R^n} \brac{\int_{\R^n\times \R} e^{\i (t\phi(\xi)-\langle x,\xi\rangle)} \eta_{|\xi| \aleq \lambda} g(x_1,x',t) dx'dt dx_1}^2 d\xi \aleq \Lambda\, \|g\|_{L^2_{x_1} L^\infty_{x',t}}^2.
\end{equation}
In order to obtain \eqref{eq:Ttastgoal} we write 
\[
\begin{split}
 &\int_{\R^n} \brac{\int_{\R^n\times \R} e^{\i (t\phi(\xi)-\langle x,\xi\rangle)} \eta_{|\xi| \aleq \lambda} g(x_1,x',t) dx'dt dx_1}^2 d\xi\\
 =&\int_{\R^n} \int_{\R^n\times \R} \int_{\R^n\times \R}  e^{\i (t\phi(\xi)-\langle x,\xi\rangle)} \eta_{|\xi| \aleq \lambda} g(x_1,x',t) e^{-\i (s\phi(\xi)-\langle y,\xi\rangle)} \eta_{|\xi| \aleq \lambda} \overline{g}(y_1,y',s) dx'dt dx_1 dy'ds dy_1  d\xi\\
 =&\int_{\R^n} \int_{\R^n\times \R} \int_{\R^n\times \R}  e^{\i ((t-s)\phi(\xi)-\langle x-y,\xi\rangle)} \eta_{|\xi| \aleq \lambda} g(x_1,x',t) \overline{g}(y_1,y',s) dx'dt dx_1 dy'ds dy_1  d\xi\\
 =& \int_{\R^n\times \R} \int_{\R^n\times \R}  \brac{\int_{\R^n} e^{\i ((t-s)\phi(\xi)-\langle x-y,\xi\rangle)} \eta_{|\xi| \aleq \lambda} d\xi} g(x_1,x',t) \overline{g}(y_1,y',s) dx'dt dx_1 dy'ds dy_1  \\
 =&\int_{\R^n\times \R} \int_{\R^n\times \R}  k(x_1-y_1,x'-y',s-t) g(x_1,x',t) \overline{g}(y_1,y',s) dx'dt dx_1 dy'ds dy_1  \\
 \end{split}
\]
where we have set
\[
 k(z_1,z',r) := \brac{\int_{\R^n} e^{\i (r\phi(\xi)-\langle z,\xi\rangle)} \eta_{|\xi| \aleq \lambda} d\xi}
\]
If we set 
\[
 \tilde{k}(z_1) := \sup_{z',r} |k(z_1,z',r)|
\]
then we have shown
\[
\begin{split}
 &\int_{\R^n} \brac{\int_{\R^n\times \R} e^{\i (t\phi(\xi)-\langle x,\xi\rangle)} \eta_{|\xi| \aleq \lambda} g(x_1,x',t) dx'dt dx_1}^2 d\xi\\
 \leq &\int_{\R^n\times \R} \int_{\R^n\times \R}  \tilde{k}(x_1-y_1) \|g(x_1,\cdot,\cdot)\|_{L^1_{x',t}} \|g(y_1,\cdot,\cdot)\|_{L^1_{x',t}} dx_1 dy_1  \\
 \aleq& \|\tilde{k} \ast \|g(x_1,\cdot,\cdot)\|_{L^1_{x',t}}\|_{L^2(\R)} \|g\|_{L^2_{x_1} L^1_{x',t}}\\
 \aleq &\|\tilde{k}\|_{L^1(\R)} \|g\|_{L^2_{x_1} L^1_{x',t}} \|g\|_{L^2_{x_1} L^1_{x',t}} \\
 \overset{\eqref{eq:Ttestneeded}}{\aleq}& \Lambda^2 \|g\|_{L^2_{x_1} L^1_{x',t}}^2.
 \end{split}
\]
The second to last inequality is simply Young's inequality,
\[
 \|\tilde{k} \ast \|g(x_1,\cdot,\cdot)\|_{L^1_{x',t}}\|_{L^2(\R)}\aleq \|\tilde{k}\|_{L^1(\R)} \|g\|_{L^2_{x_1} L^1_{x',t}}
\]
Taking the square root we obtain \eqref{eq:Ttastgoal}.
\end{proof}

\begin{lemma}\label{la:aest}
Let $n \geq 2$ and $\rho > 1$. Then there exists two smooth functions $ a_{1},a_{2}$ such that the following holds
$$
\int_{\S^{n-1}} e^{i\rho \theta_{1}} d\theta = e^{i\rho} a_{1}(\rho) + e^{-i\rho} a_{2} (\rho).
$$
Here $a_{1},a_{2}$ satisfy for any $l =0,1,....$ the following 
\begin{equation}\label{eq:aestclaim}
\abs{\frac{d^{l}}{d\rho^{l}} a_{i}(\rho) } \aleq \rho^{-\frac{n-1}{2} -l}
\end{equation}
for $i \in \{1,2\}$ and $\rho > 1$
\end{lemma}
\begin{proof}
The integral under consideration is related to the Bessel potential, see \cite[p. 347]{Stein},
\begin{equation}\label{Bess1}
\int_{\S^{n-1}} e^{i\theta_{1}\rho} d\theta= \rho^{-(n-2)/2} J_{(n-2)/2}(\rho)
\end{equation}
where $ J_{(n-2)/2}$ is the Bessel function. For $\rho >1$ it has the explicit expansion, \cite[Page 338]{Stein}
\begin{equation}\label{Bess2}
J_{(n-2)/2}(\rho)= e^{i\rho} \rho^{-1/2} \sum_{j=0}^{K} a_{j} \rho^{-j} + e^{-i\rho} \rho^{-1/2}\sum_{j=0}^{K} b_{j} \rho^{-j}. \end{equation}
Here $K$ is some integer greater than $(n-1)/2$.

Thus,
 \begin{equation}
\begin{split}
& \int_{\S^{n-1}} e^{i\rho \theta_{1}} d\theta 
\\
& = \rho^{-(n-2)/2} J_{(n-2)/2} (\rho)
\\
& = \rho^{-(n-2)/2} \left( e^{i\rho} \rho^{-1/2} \sum_{j=0}^{K} a_{j} \rho^{-j} + e^{-i\rho} \rho^{-1/2}\sum_{j=0}^{K} b_{j} \rho^{-j} \right)
\\
& = e^{i\rho} \big[ \rho^{-\frac{n-1}{2}} \sum_{j=0}^{K} a_{j} \rho^{-j}\big]  + e^{-i\rho} \big[\rho^{-1/2}\sum_{j=0}^{K} b_{j} \rho^{-j} \big]
\end{split}
\end{equation}
Set $ a_{1}(\rho) = \rho^{-\frac{n-1}{2}}\sum_{j=0}^{K} a_{j} \rho^{-j}$ and $ a_{2}(\rho) = \rho^{-\frac{n-1}{2}}\sum_{j=0}^{K} b_{j} \rho^{-j}$.

From the representation it is clear that we have \eqref{eq:aestclaim}. We can conclude.
\end{proof}

\begin{lemma}\label{la:eq:xi2sestgoal}
Let $\frac{1}{2} < s \leq 1$. Let $\eta \in C_c^\infty(B(0,1))$, $\eta \equiv 1$ in $B(0,1/2)$, and $\eta(\xi) = \eta(|\xi|)$. Denote $\psi(\xi) := \brac{\eta(\xi)-\eta(\xi/2)}$. 

Then we have for any $k \geq 0$
\begin{equation}\label{eq:xi2sestgoal}
 \sup_{x \in \R^n} \abs{\int_{\R^n} e^{\i \brac{t |\xi|^{2s} - \langle x,\xi\rangle)}} \psi(\xi/2^k) d\xi } \aleq_{\eta} 2^{kn(1-s)}  |t|^{-\frac{n}{2}}.
\end{equation}

Also, we have for $\Lambda \geq 10$
\begin{equation}\label{eq:xi2sestgoalwithn}
 \abs{\int_{\R^n} e^{\i \brac{ t |\xi|^{2s} - \langle x,\xi\rangle}} \eta (|\xi|-\Lambda) d\xi }\aleq_{\eta} \Lambda^{n(1-s)}\, |t|^{-\frac{n}{2}}.
\end{equation}
\end{lemma}
\begin{proof}
We first consider \eqref{eq:xi2sestgoal} and by scaling we may assume $k=0$. We write the integrals in polar coordinates,
\[
\begin{split}
 &\int_{\R^n} e^{\i (t |\xi |^{2s} - \langle x,\xi\rangle)} \psi(\xi) d\xi \\
 =&\int_{r=\frac{1}{2}}^1 \int_{\S^{k-1}}e^{\i (t r^{2s} - r\langle x,\theta\rangle)} \psi(r) r^{k-1} dr d\theta \\
 =&\int_{r=\frac{1}{2}}^1 \int_{\S^{k-1}} e^{\i \brac{t r^{2s}  - r|x|\brac{\langle \frac{x}{|x|},\theta\rangle} }} \psi(r) r^{k-1} dr d\theta \\
  =&\int_{r=\frac{1}{2}}^1 e^{\i \brac{t r^{2s} }} \int_{\S^{k-1}}  e^{-\i r|x| \theta_1}  d\theta\, \psi(r) r^{k-1} dr \\
 \end{split}
\]

We consider two cases: If $|x| \aleq 1$ (and thus $r|x| \aleq 1$ for all $r \in (1/2,1)$) we write this as 
\[
  \int_{r=\frac{1}{2}}^1 e^{\i \brac{t r^{2s} +r|x|}} \underbrace{\int_{\S^{k-1}}  e^{-\i r|x| (\theta_1-1)}  d\theta}_{=:a_0(r|x|)}\, \psi(r) r^{k-1} dr \\
\]
and observe that we have 
\[
 \abs{\frac{d^l}{dr^l} a(r|x|)} \aleq |x|^l |a_0^{(l)}(r|x|)| \aleq |x|^l \aleq |x|^l (1+|rx|)^{\frac{n-1}{2}-l} \quad \forall r \in [\frac{1}{2},1],
\]
If on the other hand we have $|x| \geq 10$ (and thus $r|x| \geq 1$ for all $r \in [1/2,1]$, then we use \Cref{la:aest} and write 
\[
 \begin{split}
&\int_{r=\frac{1}{2}}^1 e^{\i \brac{t r^{2s} }} \int_{\S^{k-1}}  e^{-\i r|x| \theta_1}  d\theta\, \psi(r) r^{k-1} dr\\
=&\int_{r=\frac{1}{2}}^1 e^{\i \brac{t r^{2s} +r|x|}} a_1(r|x|)\, \psi(r) r^{k-1} dr\\  
&+\int_{r=\frac{1}{2}}^1 e^{\i \brac{t r^{2s} -r|x|}} a_2(r|x|)\, \psi(r) r^{k-1} dr\\  
 \end{split}
\]
where again we have for $i=1,2$
\[
 \abs{\frac{d^l}{dr^l} a_i(r|x|)} \aleq |x|^l (r|x|)^{-\frac{n-1}{2}} \aleq |x|^l (1+|rx|)^{\frac{n-1}{2}-l} \quad \forall r \in [\frac{1}{2},1],
\]
Thus, we need to estimate 
\[
 \int_{r=\frac{1}{2}}^1 e^{\i t \phi(r)} a(r|x|)  \psi(r) r^{k-1} dr.
\]
where \[\phi(r) :=\brac{r^{2s} \pm \frac{r}{t}|x|}.\]
and 
\begin{equation}\label{eq:applla:aest}
 \abs{\frac{d^l}{dr^l} a(r|x|)} \aleq |x|^l (1+|rx|)^{\frac{n-1}{2}-l} \quad \forall r \in [\frac{1}{2},1],
\end{equation}

We will carry out the proof for 
\[
\phi(r) :=\brac{r^{2s} - \frac{r}{t}|x|}
\]
the case of the flipped sign is analogous.

Observe that we have 
\[
\phi'(r) = 2s(r^2)^{s-1} r - \frac{|x|}{t}.
\]
and 
\[
|\phi^{(\ell)}(r)| \aleq_\ell 1 \quad \forall \ell \geq 2, \forall r \in [\frac{1}{2},1]
\]

\underline{Assume that $\frac{|x|}{t} \aeq 1$.}
%
We compute 
\[
\begin{split}
 \phi''(r) =& 2sr^{2s-2} +2s2(s-1)r^{2s-4}r^2\\
 =& 2sr^{2s-4} \brac{(r^2 )+ 2(s-1)r^2}\\
 =& 2sr^{2s-4} \brac{(2s-1)r^2 }\\
 \end{split}
\]
and thus, 
\[
 \phi''(r) \geq 2s r^{2s-4} (2s-1) r^2 = 2s(2s-1) r^{2s-2}\ageq 1 \quad \text{for $r \in (1/2,1)$}.
\]
By Van der Corput lemma, \cite[Corollary p.334]{Stein},
\[
\begin{split}
 &\int_{r=\frac{1}{2}}^1 e^{t \phi(r)} \col{a}(r|x|)  \psi(r) r^{k-1} dr \\
 \aleq&t^{-1/2} \brac{\|a(r|x|)\|_{L^\infty((1/2,1))} + |x| \|a'(r|x|)\|_{L^1((1/2,1))} }\\
 \aleq&t^{-1/2} \brac{|x|^{\frac{1-k}{2}} + |x| \brac{|x|}^{\frac{1-k}{2} -1}}\\
 \overset{|x|  \aeq t}{\aeq}&t^{-\frac{k}{2}}
 \end{split}
\]
where we have used \eqref{eq:applla:aest}.

\underline{Assume now $|x| \gg |t|$}.

Assuming that $|x| \gg t$ means suitably large, using that $s \leq 1$,
\[
 |\phi'(r)| \aeq \frac{|x|}{t} \quad \forall r \in [0,1].
\]
Then we can do infinitely many integration by parts as above, so that we get for any $n \geq 1$,
\[
\begin{split}
 &\abs{\int_{r=\frac{1}{2}}^1 e^{\i \brac{t r^{2s} -r|x|}} a(r|x|)\,  \psi(r)\, r^{k-1} dr}\\
 \aleq&\sup_{r \in [0,1]} \abs{\frac{\|a(r|x|)\|_{L^\infty}+\ldots+\|a^{(n)}(r|x|)\|_{L^\infty} |x|^n }{\abs{t \phi'(r)}^n}}\\
 \aeq&|x|^{-n} \aleq t^{-n}.
 \end{split}
\]
Again using \eqref{eq:applla:aest}.

Lastly, \underline{assume $|x| \ll t$}.
Then 
\[
 |\phi'(r)| \aeq 1 \quad\text{for all $r \in [\frac{1}{2},1]$}
\]
i.e.
\[
 |t\phi'(r)| \aeq t \quad\text{for all $r \in [\frac{1}{2},1]$}
\]
Then we have by the same integration by parts, using 
\[
\begin{split}
 &\abs{\int_{r=1/2}^1 e^{\i t\phi'(r)} a(r|x|)\,  \psi(r)\, r^{k-1} dr}\\
 \aleq&\sup_{r \in [1/2,1]} \abs{\frac{1}{\abs{t\phi'(r)}^n}}\\
 \aeq& t^{-n}.
 \end{split}
\]
This proves \eqref{eq:xi2sestgoal}.

As of \eqref{eq:xi2sestgoalwithn}, observe that 
\[
\begin{split}
 &\abs{\int_{\R^n} e^{\i \brac{ t |\xi|^{2s} - \langle x,\xi\rangle}} \eta (|\xi|-\Lambda) d\xi }\\
 =&\abs{\int_{\R^n} e^{\i \brac{ t |\xi|^{2s} - \langle x,\xi\rangle}} \psi(\xi/\Lambda) \eta (|\xi|-\Lambda) d\xi }\\
\end{split}
 \]
Then writing $\eta (|\xi|-\Lambda) = \int_{\R^n} e^{\i \langle y,\xi\rangle} \mathcal{F}(\eta(|\cdot|-\Lambda))(y) dy$ and observing that $\|\mathcal{F}(\eta(|\cdot|-\Lambda))(y)\|_{L^1} \aleq_{\eta} 1$ we conclude \eqref{eq:xi2sestgoalwithn}.

\end{proof}

The following is the part where we use $s < 1$ (although an alternative method can be used for the $s=1$ case).
\begin{proposition}\label{pr:statphaseestschrod}
Let $\frac{1}{2} < s < 1$. Let $\eta \in C_c^\infty(B(0,1))$ with $\eta \equiv 1$ in a neighborhood of $0$. Assume $N \geq 1$, Then we have for any $x \in \R^k$, $t \in \R$
\begin{equation}\label{ex:pr:statphaseestschrodv1}
 \abs{\int_{\R^N} e^{\i \brac{t |\xi-\tn|^{2s} - \langle x,\xi\rangle}} \eta(\xi) d\xi }\aleq_{\eta} \min \{1,\max\{|\tn|,1\}^{N(1-s)-1} |t|^{-\frac{N}{2}}\}.
\end{equation}
which holds any $\tn \in \R^N$.
\end{proposition}
\begin{proof}
First we consider the case where $\eta(\xi) = \eta(|\xi|)$ and $\tn = 0$.

We assume that $\eta$ is constant in a neighborhood of $B(1/2)$. If this were not the case we can slice $\supp \eta \cap \supp (\eta-1)$ into finitely many annuli, choose appropriate cutoff functions to reduce to this situation. 

Since we have \Cref{la:eq:xi2sestgoal}, \eqref{eq:xi2sestgoal}, by scaling we get 
\[
\begin{split}
 &\abs{\int_{\R^k} e^{\i (t |\xi|^{2s} - \langle x,\xi\rangle)} \psi(\xi/\lambda) d\xi }\\
 =&\abs{\int_{\R^k} e^{\i (t\lambda^{2s} |\xi/\lambda|^{2s} - \langle \lambda x,\xi/\lambda\rangle)} \psi(\xi/\lambda) d\xi }\\
 =&\lambda^{k} \abs{\int_{\R^k} e^{\i (t\lambda^{2s} |\xi|^{2s} - \langle \lambda x,\xi\rangle)} \psi(\xi) d\xi }\\
 \overset{\eqref{eq:xi2sestgoal}}{\aleq}&\lambda^{k} |\lambda^{2s} t|^{-\frac{k}{2}}\\
 =&\lambda^{k(1-s)} |t|^{-\frac{k}{2}}.
  \end{split}
\]
Thus, since $\eta(\xi) = \sum_{\ell =-K} \psi(\xi/2^\ell) + \eta(\xi/2^L)$ we then have 
\[
\begin{split}
 &\abs{\int_{\R^k} e^{\i (t |\xi|^{2s} - \langle x,\xi\rangle)} \eta(\xi) d\xi }\\
 \leq&\abs{\int_{\R^k} e^{\i (t |\xi|^{2s} - \langle x,\xi\rangle)} \eta(\xi/2^L) d\xi }\\
 &+\sum_{\ell \leq 0}\abs{\int_{\R^k} e^{\i (t |\xi|^{2s} - \langle x,\xi\rangle)} \psi(\xi/2^\ell) d\xi }\\
 \aleq&2^{-Lk} + \sum_{\ell \leq 0} 2^{k(1-s)\ell} |t|^{-\frac{k}{2}}\\
  \end{split}
\]
Since $s < 1$ we find that 
\[
\begin{split}
 &\abs{\int_{\R^k} e^{\i (t |\xi|^{2s} - \langle x,\xi\rangle)} \eta(\xi) d\xi }\\
 \aleq&2^{-Lk} + |t|^{-\frac{k}{2}}\\
 \xrightarrow{L \to \infty}&|t|^{-\frac{k}{2}}.
  \end{split}
\]
Since the estimate 
\[
 \abs{\int_{\R^k} e^{\i (t |\xi|^{2s} - \langle x,\xi\rangle)} \eta(\xi) d\xi }
 \aleq 1\\
\]
is trivial, we can conclude since we have \eqref{eq:xi2sestgoal}.

This establishes \eqref{ex:pr:statphaseestschrodv1} for $\tn = 0$ and $\eta$ radial.

If $\eta$ is not radial, we simply write for the standard radial cutoff function $\chi$ that is constantly on on the ball of radius $1$.
\[
 \eta(\xi) = \chi(|\xi|) \int_{\R^N} e^{\i \langle y,\xi\rangle} \hat{\eta}(y) dy
\]
Then
\[
\begin{split}
&\abs{\int_{\R^N} e^{\i \brac{t |\xi|^{2s} - \langle x,\xi\rangle}} \eta(\xi) d\xi }\\
\leq& \int_{\R^N} \abs{\int_{\R^N} e^{\i \brac{t |\xi|^{2s} - \langle x-y,\xi\rangle}} \chi(|\xi|) d\xi \abs{\hat{\eta}(y)} }dy\\
\aleq&\|\hat{\eta}\|_{L^1(\R^N)} |t|^{-\frac{N}{2}}.
\end{split}
\]
This establishes \eqref{ex:pr:statphaseestschrodv1} for $\tn = 0$ and $\eta$ possibly nonradial.

Now we consider general $\tn \in \R^N$. We have
\[
\begin{split}
 &\abs{\int_{\R^N} e^{\i \brac{t |\xi-\tn|^{2s} - \langle x,\xi\rangle}} \eta(\xi) d\xi}\\
 =&\abs{\int_{\R^N} e^{\i \brac{t |\xi|^{2s} - \langle x,\xi\rangle}} \eta(\xi+\tn) d\xi}\\
 \end{split}
\]
If $|\tn| \aleq 10$ we use a scaling argument and \eqref{ex:pr:statphaseestschrodv1} (for $\tn = 0$ and the nonradial $\xi \mapsto \eta(\xi+\tn)$).

If $|\tn| \geq 10$ we instead use that $|\xi+\tn| \leq 1$ implies $||\xi| - |\tn|| \leq 1$, so we have 
\[
\begin{split}
 &\abs{\int_{\R^N} e^{\i \brac{t |\xi-\tn|^{2s} - \langle x,\xi\rangle}} \eta(\xi) d\xi}\\
 =&\abs{\int_{\R^N} e^{\i \brac{t |\xi|^{2s} - \langle x,\xi\rangle}} \chi(|\xi|-|\tn|) \eta(\xi+\tn) d\xi}\\
 \end{split}
\]
As before we write 
\[
 \eta(\xi+\tn) = \int_{\R^N} e^{\i \langle y,\xi\rangle} e^{\i \langle y,\tn\rangle} \hat{\eta}(y) dy
\]
and use \Cref{la:eq:xi2sestgoal}, \eqref{eq:xi2sestgoalwithn}, to estimate 
\[
\begin{split}
 &\abs{\int_{\R^N} e^{\i \brac{t |\xi-\tn|^{2s} - \langle x,\xi\rangle}} \eta(\xi) d\xi}\\
 \aleq_{\eta}& |\tn|^{N(1-s)-1} |t|^{-\frac{N}{2}}\\
 \end{split}
\]

\end{proof}

\begin{lemma}\label{la:twointegrals}
Let $s \in (\frac{1}{2},1)$. There exists a constant $C_{s} >0$ such that the following holds.

Let $\eta \in C_c^\infty(B(0,1))$ be radial with $\eta \equiv 1$ in $B(0,1/2)$

Fix $x_1 \in \R$, $x = (x^1,x') \in \R^n$, $t > 0$ and assume that we have $|x_1| \geq C_{s} t$.

Then we have
\begin{equation}\label{eq:la:twointegralsv1}
 \abs{\int_{\R^n} e^{\i \brac{t |\xi|^{2s} - \langle x,\xi\rangle } } \eta(\xi) d\xi}\aleq_{\eta}  \min\{1,
  |x_1|^{-2}+t|x_1|^{-2}\}.
\end{equation}
Also, if $\tn \in \R^n$, $|\tn| \leq \Lambda$ and $\frac{|x_1|}{|t|} \geq C_{s} \Lambda^{2s-1}$ for $\Lambda \geq 1$
\begin{equation}\label{eq:la:twointegralsv1withn}
 \abs{\int_{\R^n} e^{\i \brac{t |\xi+\tn|^{2s} - \langle x,\xi\rangle } } \eta(\xi) d\xi}\aleq_{\eta}  \min\{1,
  |x_1|^{-2}+t|x_1|^{-2}\}.
\end{equation}

\end{lemma}
\begin{proof}
Firstly, we may assume that $|x_1| \geq 100$, otherwise both estimates are trivial.

The case $|\tn| \aleq 1$ follows from the case $|\tn| = 0$ by conducting the argument below with $\eta(|\xi-\tn|)$, so in the following we assume that either $|\tn| = 0$ or $|\tn| \geq 100$.

Set 
\[
 \phi_{1;x_1}(\xi_1,\xi') := t \brac{|\xi_1+\tn_1|^{2}+|\xi'+\tn'|^2}^s - x_1 \xi_1,
\]
then 
\[
t |\xi+\tn|^{2s} - \langle x,\xi\rangle =  \phi_{1;x_1}(\xi_1,\xi') - \langle x',\xi'\rangle.
\]

Thus, we have to estimate for $j=1,2$
 \[
\begin{split}
 &\sup_{|\xi'| \aleq 1} \abs{\int_{-1}^1 e^{\i \phi_{j;x_1}(r,\xi')} \eta((r^2+|\xi'|^2)^{\frac{1}{2}}) dr} \aleq \frac{1}{|x_1|^2}
 \end{split}
\]

We have, since $|\frac{x_1}{t}| \gg \Lambda^{2s-1} \ageq |\xi+\tn|^{2s-1}$
\[
\begin{split}
 \partial_{\xi_1 } \phi_{1;x_1} (\xi_1,|\xi'|) =&  2st \brac{(\xi_1+\tn_1)^{2}+|\xi'+\tn'|^2 }^{s-1}  (\xi_1+\tn_1)  - x_1\\
 =&  t  \brac{2s\brac{(\xi_1+\tn_1)^{2}+|\xi'+\tn'|^2 }^{s-1}  (\xi_1+\tn_1)  - \frac{x_1}{t}}
\\
\aeq& -x_1 \quad \forall \xi_1\in [-1,1],|\xi'| \aleq 1 
 \end{split}
\]
Also we have (recall that either $|\tn| = 0$ or $|\tn| \gg |\xi|$
\[
\begin{split}
 |\partial_{\xi_1 \xi_1}\phi_{1;x_1}(\xi_1,|\xi'|)| =& 2s\abs{t} \abs{2(s-1)\brac{(\xi_1+\tn_1)^2 + |\xi'+\tn'|^2}^{s-2} (\xi_1+\tn_1)^2 +\brac{(\xi_1+\tn_1)^2 + |\xi'+\tn|^2}^{s-1}}\\
 \aleq& \abs{t} \brac{|\xi+\tn|^{2s-2}  +|\xi+\tn|^{2s-2}}\\
 \overset{s<1}{\aleq}&\begin{cases}
                       \abs{t} |\xi|^{2s-2} \quad &\tn = 0\\
                       \abs{t} \quad &|\tn|\gg 1.
                      \end{cases}
 \end{split}
\]

By an integration by parts argument we have (observe that $r \mapsto \eta((r^2+|\xi'|^2)^{\frac{1}{2}})$ is uniformly $C^\infty$ since $\eta(\rho) \equiv 1$ for $\rho \aeq 0$ and $\rho \ageq 1$)
\[
\begin{split}
 &\abs{\int_{-1}^1 e^{\i \phi_{x_1}(r,|\xi'|)} \eta((r^2+|\xi'|^2)^{\frac{1}{2}})dr} \\
 =&\abs{\int_{-1}^1 e^{\i \phi_{x_1}(r,|\xi'|)}  \partial_r \brac{\frac{1}{\partial_r \phi_{x_1}(r,|\xi'|)} \eta((r^2+|\xi'|^2)^{\frac{1}{2}})} dr} \\
 \leq&\abs{\int_{-1}^1 e^{\i \phi_{x_1}(r,|\xi'|)}  \frac{\partial_{rr} \phi_{x_1}(r,|\xi'|)}{(\partial_r \phi_{x_1}(r,|\xi'|))^2} \eta((r^2+|\xi'|^2)^{\frac{1}{2}}) dr} \\
 &+\abs{\int_{-1}^1 e^{\i \phi_{x_1}(r,|\xi'|)}  \frac{1}{\partial_r \phi_{x_1}(r,|\xi'|)} \partial_r \brac{\eta((r^2+|\xi'|^2)^{\frac{1}{2}})} dr} \\
 \aleq& \int_{-1}^1 \frac{ t |r|^{2s-2}}{|x_1|^2}  \\
 &+\abs{\int_{-1}^1 e^{\i \phi_{x_1}(r,|\xi'|)}  \partial_r \brac{\frac{1}{(\partial_r \phi_{x_1}(r,|\xi'|))^2} \partial_r \brac{\eta((r^2+|\xi'|^2)^{\frac{1}{2}})}} dr} \\
\aleq& \int_{-1}^1 \frac{ t |r|^{2s-2}}{|x_1|^2}  \\
 &+\abs{\int_{-1}^1 e^{\i \phi_{x_1}(r,|\xi'|)}  \brac{\frac{1}{\partial_r \phi_{x_1}(r,|\xi'|)} \frac{1}{\partial_r \phi_{x_1}(r,|\xi'|)} \partial_{rr} \brac{\eta((r^2+|\xi'|^2)^{\frac{1}{2}})}} dr} \\
&+\abs{\int_{-1}^1 e^{\i \phi_{x_1}(r,|\xi'|)}  \brac{\frac{\partial_{rr} \phi_{x_1}(r,|\xi'|))}{(\partial_r \phi_{x_1}(r,|\xi'|))^3} \partial_r \brac{\eta((r^2+|\xi'|^2)^{\frac{1}{2}})}} dr} \\%
\aleq& \int_{-1}^1 \frac{ t |r|^{2s-2}}{|x_1|^2} dr +\frac{1}{|x_1|^2} + \int_{-1}^1 \frac{t |r|^{2s-2}}{|x_1|^3}dr
\end{split}
\]
Since $s > \frac{1}{2}$ we have $2s-2 > -1$ and thus the above integrals converge, and we have
\[
\begin{split}
 &\abs{\int_{-1}^1 e^{\i \phi_{x_1}(r,|\xi'|)} \eta((r^2+|\xi'|^2)^{\frac{1}{2}})dr} \\
\aleq& \frac{t}{|x_1|^2}  +\frac{1}{|x_1|^2} 
\end{split}
\]
\end{proof}

\begin{proposition}\label{pr:eq:DIE:4.7}
Let $n \geq 3$, $s \in (1/2,1)$ and $\ell \in \Z$, $k \in \Z$ such that $2^{k} \ageq 2^\ell$.

For any $\tn \in \R^n$ such that $|\tn| \aleq 2^k$,
\begin{equation}\label{eq:DIE:4.7v1gen}
 \int_{\R} \sup_{x',t}
  \abs{\int_{\R^{n}} e^{\i \langle x,\xi\rangle} e^{-\i t |\xi-\tn|^{2s}} \eta_{|\xi| \aleq 
  2^\ell} d\xi} dx_1 
  \aleq 2^{k(n-1)} 2^{-(n-2)(k-\ell)} = 2^{(n-1)\ell} 2^{k-\ell} 
\end{equation}
In particular (for $\tn=0$ and $\ell=k$)
\begin{equation}\label{eq:DIE:4.7v1}
 \int_{\R} \sup_{x',t}
  \abs{\int_{\R^{n}} e^{\i \langle x,\xi\rangle} e^{-\i t |\xi|^{2s}} \eta_{|\xi| \aleq 
  2^\ell} d\xi} dx_1 
  \aleq 2^{\ell(n-1)}  
\end{equation}
 \end{proposition}
\begin{proof}
We are going to prove
\begin{equation}\label{eq:DIE:4.7v1genv2}
 \int_{\R} \sup_{x',t}
  \abs{\int_{\R^{n}} e^{\i \langle x,\xi\rangle} e^{-\i t |\xi-\tn|^{2s}} \eta_{|\xi| \aleq 
  1} d\xi} dx_1 
  \aleq \max\{|\tn|,1\}
\end{equation}
for all $\tn \in \R^n$.

Once we have \eqref{eq:DIE:4.7v1genv2} we have \eqref{eq:DIE:4.7v1gen} by scaling for any $|\tn| \aleq 2^k$, $2^k \ageq 2^\ell$:
\[
 \int_{\R} \sup_{x',t}
  \abs{\int_{\R^{n}} e^{\i \langle x,\xi\rangle} e^{-\i t |\xi-\tn|^{2s}} \eta_{|\xi| \aleq 
  2^\ell} d\xi} dx_1 
  =2^{k(n-1)} 2^{-(k-\ell) (n-2)}.
\]
In order to prove \eqref{eq:DIE:4.7v1genv2} we show that for any $x_1 \in \R$, $\theta \in (0,1)$ suitably close to $1$, we have
\begin{equation}\label{eq:DIE:4.7:goalv1}
 \sup_{x',t}   \abs{\int_{\R^n} e^{\i \langle x,\xi\rangle} e^{-\i t |\xi-\tn|^{2s}} \eta_{|\xi| \aleq 
  1} d\xi} \aleq (1+|x_1|)^{\theta-2} + \frac{1}{1+|x_1/\max\{|\tn|,1\}|^{\theta \frac{n}{2}}}. 
\end{equation}
Once we have \eqref{eq:DIE:4.7:goalv1} we obtain  \eqref{eq:DIE:4.7v1genv2} by integrating in $x_1$, choosing $\theta$ suitably close to $1$ so that $\theta \frac{n}{2} > 1$ (recall that $n \geq 3$).

\eqref{eq:DIE:4.7:goalv1} is trivial for $|x_1| \aleq 1$, so we assume $|x_1| > 1$.

Let $\theta \in (0,1)$, $\gamma > 0$. Assume first \underline{$\frac{|x_1|^\theta}{|t|} \aleq \max\{|\tn|,1\}^{\gamma(2s-1)}$}. Apply \Cref{pr:statphaseestschrod} \eqref{ex:pr:statphaseestschrodv1} and we have  
\[
\begin{split}
 &\abs{\int_{\R^n} e^{\i (t |\xi+\tn|^{2s} - \langle x,\xi\rangle)} \eta_{|\xi| \aleq 1} d\xi }\\
 \aleq& \max\{|\tn|,1\}^{n(1-s)-1} |t|^{-\frac{n}{2}}\\
 \aleq &\max\{|\tn|,1\}^{n(1-s)-1} \max\{|\tn|,1\}^{\frac{n}{2}\gamma (2s-1)}         |x_1|^{-\theta\frac{n}{2}}\\
 \aleq &\max\{|\tn|,1\}^{\frac{n}{2}}      \frac{1}{|x_1|^{\theta \frac{n}{2}}}
 \end{split}
\]
If we choose (for $\theta$ suitably close to $1$ we have $\gamma$ is close to $1$)
\[
 n(1-s)-1+ \frac{n}{s} \gamma(2s-1) = \theta \frac{n}{2}
\]
we obtain
\[
 \abs{\int_{\R^n} e^{\i (t |\xi|^{2s} - \langle x,\xi\rangle)} \eta_{|\xi| \aleq 1} d\xi } \aleq \frac{1}{|x_1/\max\{|\tn|,1\}|^{\frac{n}{2}}}.  
\]
So \underline{assume now that $\frac{|x_1|^\theta}{|t|} \gg \max\{|\tn|,1\}^{\gamma(2s-1)}$}. 
In particular $|x_1|^\theta \gg |t|$, and since $\theta < 1$ this implies in particular $|x_1| \gg |t|$. Thus we can apply \Cref{la:twointegrals}, \eqref{eq:la:twointegralsv1withn}, and have 
\[
 \abs{\int_{\R^n} e^{\i (t |\xi-\tn|^{2s} - \langle x,\xi\rangle)} \eta_{|\xi| \aleq 1} d\xi }\aleq |x_1|^{-2}  +t|x_1|^{-2} \aleq |x_1|^{\theta-2}.
\]
This implies \eqref{eq:DIE:4.7:goalv1}. 
\end{proof}

\begin{proof}[Proof of \Cref{th:L2eLinftyeperpvsL2est}]
Follows from \Cref{pr:eq:DIE:4.7} and \eqref{eq:DIE:4.7v1}.
\end{proof}

\bibliographystyle{abbrv}%
\bibliography{bib}%

\begin{thebibliography}{10}

\bibitem{BIK07}
I.~Bejenaru, A.~D. Ionescu, and C.~E. Kenig.
\newblock Global existence and uniqueness of {S}chr\"{o}dinger maps in
  dimensions {$d\geq 4$}.
\newblock {\em Adv. Math.}, 215(1):263--291, 2007.

\bibitem{BIKTAnnals}
I.~Bejenaru, A.~D. Ionescu, C.~E. Kenig, and D.~Tataru.
\newblock Global {Schr{\"o}dinger} maps in dimensions {{\(d\geq 2\)}}: small
  data in the critical {Sobolev} spaces.
\newblock {\em Ann. Math. (2)}, 173(3):1443--1506, 2011.

\bibitem{Berntson_2020}
B.~K. Berntson, R.~Klabbers, and E.~Langmann.
\newblock Multi-solitons of the half-wave maps equation and
  calogero{\textendash}moser spin{\textendash}pole dynamics.
\newblock {\em Journal of Physics A: Mathematical and Theoretical},
  53(50):505702, nov 2020.

\bibitem{ESRS}
E.~Eyeson, S.~R. Farina, and A.~Schikorra.
\newblock On uniqueness for half-wave maps in dimension {{\(d \geq 3\)}}.
\newblock {\em Trans. Am. Math. Soc., Ser. B}, 11:508--539, 2024.

\bibitem{GL18}
P.~G\'{e}rard and E.~Lenzmann.
\newblock A {L}ax pair structure for the half-wave maps equation.
\newblock {\em Lett. Math. Phys.}, 108(7):1635--1648, 2018.

\bibitem{GerardLenzmann}
P.~{G{\'e}rard} and E.~{Lenzmann}.
\newblock {Global Well-Posedness and Soliton Resolution for the Half-Wave Maps
  Equation with Rational Data}.
\newblock {\em arXiv e-prints}, page arXiv:2412.03351, Dec. 2024.

\bibitem{IKDIE}
A.~D. Ionescu and C.~E. Kenig.
\newblock Low-regularity {S}chr\"odinger maps.
\newblock {\em Differential Integral Equations}, 19(11):1271--1300, 2006.

\bibitem{IKCMP}
A.~D. Ionescu and C.~E. Kenig.
\newblock Low-regularity {S}chr\"{o}dinger maps. {II}. {G}lobal well-posedness
  in dimensions {$d\geq 3$}.
\newblock {\em Comm. Math. Phys.}, 271(2):523--559, 2007.

\bibitem{KK21}
A.~Kiesenhofer and J.~Krieger.
\newblock Small data global regularity for half-wave maps in {$n=4$}
  dimensions.
\newblock {\em Comm. Partial Differential Equations}, 46(12):2305--2324, 2021.

\bibitem{KS18}
J.~Krieger and Y.~Sire.
\newblock Small data global regularity for half-wave maps.
\newblock {\em Anal. PDE}, 11(3):661--682, 2018.

\bibitem{LS2018}
E.~Lenzmann and A.~Schikorra.
\newblock On energy-critical half-wave maps into {$\mathbb{S}^2$}.
\newblock {\em Invent. Math.}, 213(1):1--82, 2018.

\bibitem{LenzmannSok20}
E.~Lenzmann and J.~Sok.
\newblock {Derivation of the Half-Wave Maps Equation from Calogero--Moser Spin
  Systems}, 2020.

\bibitem{Liu2021}
Y.~{Liu}.
\newblock {Global Well-Posedness For Half-Wave Maps With $S^2$ and
  $\mathbb{H}^2$ Targets For Small Smooth Initial Data}.
\newblock {\em arXiv e-prints}, page arXiv:2109.13657, Sept. 2021.

\bibitem{Liu2023}
Y.~{Liu}.
\newblock {Global Weak Solutions for the Half-Wave Maps Equation in
  $\mathbb{R}$}.
\newblock {\em arXiv e-prints}, page arXiv:2308.06836, Aug. 2023.

\bibitem{Marsden24}
K.~{Marsden}.
\newblock {Global Solutions to the 3D Half-Wave Maps Equation with Angular
  Regularity}.
\newblock {\em arXiv e-prints}, page arXiv:2403.14567, Mar. 2024.

\bibitem{Silvino24}
S.~{Reyes Farina}.
\newblock {On uniqueness for hyperbolic half-wave maps in dimension $d \geq
  3$}.
\newblock {\em arXiv e-prints}, page arXiv:2407.06448, July 2024.

\bibitem{FS24}
S.~{Reyes Farina} and A.~{Schikorra}.
\newblock {On wave systems with antisymmetric potential in dimension d >= 4 and
  well-posedness for (half-)wave maps}.
\newblock {\em arXiv e-prints}, page arXiv:2404.19421, Apr. 2024.

\bibitem{SS02}
J.~Shatah and M.~Struwe.
\newblock The {C}auchy problem for wave maps.
\newblock {\em Int. Math. Res. Not.}, (11):555--571, 2002.

\bibitem{SWZ21}
Y.~Sire, J.~Wei, and Y.~Zheng.
\newblock {Infinite time blow-up for half-harmonic map flow from {$\mathbb{R}$}
  into {$\mathbb{S}^1$}}.
\newblock {\em Amer. J. Math.}, 143(4):1261--1335, 2021.

\bibitem{Stein}
E.~M. Stein.
\newblock {\em Harmonic analysis: {Real}-variable methods, orthogonality, and
  oscillatory integrals. {With} the assistance of {Timothy} {S}. {Murphy}},
  volume~43 of {\em Princeton Math. Ser.}
\newblock Princeton, NJ: Princeton University Press, 1993.

\bibitem{Tao01I}
T.~Tao.
\newblock Global regularity of wave maps. {I}. {S}mall critical {S}obolev norm
  in high dimension.
\newblock {\em Internat. Math. Res. Notices}, (6):299--328, 2001.

\bibitem{Tao01II}
T.~Tao.
\newblock Global regularity of wave maps. {II}. {S}mall energy in two
  dimensions.
\newblock {\em Comm. Math. Phys.}, 224(2):443--544, 2001.

\bibitem{T98}
D.~Tataru.
\newblock Local and global results for wave maps. {I}.
\newblock {\em Comm. Partial Differential Equations}, 23(9-10):1781--1793,
  1998.

\end{thebibliography}

\end{document}